\pdfoutput=1
\documentclass{amsart}

\pdfoutput=1

\usepackage[T1]{fontenc}
\usepackage[mathscr]{eucal}% so-called Euler fonts, allows \mathscr
\usepackage{amssymb}% allows special arrows like >-> e.g.
\usepackage{booktabs} %pretty tables
\usepackage[usenames,dvipsnames]{xcolor} %changed to xcolor; color was giving ``incompatible color definition'' warnings
\usepackage[normalem]{ulem}% allows \sout to strike-out, cross-out text.
\usepackage{amsthm}% allows theorem* , e.g.
\usepackage{bbold}% allows \unit \mathbb{1}
\usepackage{comment}% allows multiple lines to be commeted out
\usepackage[perpage]{footmisc} %resets footnotes per page
\usepackage{enumitem}% allows resuming enumerate
\usepackage{amsmath}% allows pmatrix
\usepackage{centernot}% allows \centernot\Longrightarrow for better "not" placement

\swapnumbers
\usepackage{tikz}
\usepackage{tikz-cd}
\usetikzlibrary{arrows}

\usepackage{etoolbox} %this is needed for the hack removing address indentation

\usepackage[unicode]{hyperref} %added unicode; was getting warnings otherwise

\definecolor{dark-red}{rgb}{0.5,0.15,0.15}
\definecolor{dark-blue}{rgb}{0.15,0.15,0.6}
\definecolor{dark-green}{rgb}{0.15,0.6,0.15}

\hypersetup{
    colorlinks, linkcolor=Blue,
    citecolor=Blue, urlcolor=Blue
}
\usepackage{comment}

\usepackage[nameinlink,capitalise,noabbrev]{cleveref}
\usepackage{microtype}

\numberwithin{equation}{section}% makes equat numb contain the section
\usepackage[all]{xy}
\xyoption{line}
\newdir{ >}{{}*!/-10pt/\dir{>}}
\usepackage{graphicx}
\usepackage{mathtools}

\newtheorem{Thm}[equation]{Theorem}
\newtheorem*{Thm*}{Theorem}
\newtheorem*{MainThm*}{Main Theorem}
\newtheorem{Prop}[equation]{Proposition}
\newtheorem{Lem}[equation]{Lemma}
\newtheorem{Cor}[equation]{Corollary}

\newtheorem*{Que*}{Question}
\newtheorem*{Goal*}{Goal}

\theoremstyle{remark}
\newtheorem{Def}[equation]{Definition}
\newtheorem{Ter}[equation]{Terminology}
\newtheorem{Not}[equation]{Notation}
\newtheorem{Exa}[equation]{Example}

\newtheorem{Rem}[equation]{Remark}
\tikzset{
    labelrotatebelow/.style={anchor=north, rotate=90, inner sep=1.0mm}
}
\tikzset{
    labelrotateabove/.style={anchor=south, rotate=90, inner sep=1.0mm}
}

\usetikzlibrary{decorations.markings}
\tikzset{negated/.style={
        decoration={markings,
            mark= at position 0.5 with {
                \node[transform shape] (tempnode) {$\backslash \! \! \backslash$};
            }
        },
        postaction={decorate}
    }
}

\newcommand{\nc}{\newcommand}
\nc{\dmo}{\DeclareMathOperator}

\renewcommand{\emptyset}{\varnothing}

\nc{\Beren}[1]{{\color{MidnightBlue}#1}}
\nc{\Drew}[1]{{\color{Orange}#1}}
\nc{\Tobi}[1]{{\color{Green}#1}}
\nc{\Natalia}[1]{{\color{Yellow}#1}}
\nc{\Dout}[1]{\Drew{\sout{#1}}}
\nc{\Bout}[1]{\Beren{\sout{#1}}}
\nc{\Tout}[1]{\Tobi{\sout{#1}}}
\nc{\Nout}[1]{\Natalia{\sout{#1}}}
\nc{\Greg}[1]{{\color{magenta}#1}}

\usepackage{todonotes}

\usepackage{pdflscape}

\newcommand{\eqv}{\mathrel{\makebox[\widthof{$\Longleftrightarrow$}][c]{$\Longleftrightarrow$}}}
\newcommand{\imp}{\mathrel{\makebox[\widthof{$\Longleftrightarrow$}][c]{$\Longrightarrow$}}}

\dmo{\kos}{kos}
\dmo{\BL}{BL}
\nc{\ABbar}{\bar{\cat A}_{\cat B}}
\nc{\ACbar}{\bar{\cat A}_{\cat C}}
\nc{\EBbar}{\bar{E}_{\cat B}}
\nc{\ECbar}{\bar{E}_{\cat C}}
\nc{\unitbar}{\bar{\unit}}
\nc{\Fhat}{\hat{F}}
\nc{\Uhat}{\hat{U}}
\nc{\Fbar}{\bar{F}}
\nc{\Ubar}{\bar{U}}
\nc{\EB}{E_{\cat B}}
\nc{\EC}{E_{\cat C}}
\nc{\varphiiC}{\varphi_i^h\hspace{-0.1ex}(\cat C)}
\nc{\varphiC}{\varphi^h\hspace{-0.1ex}(\cat C)}
\nc{\varphiB}{\varphi^h\hspace{-0.1ex}(\cat B)}
\nc{\EvarphiiC}{E_{\varphiiC}}
\nc{\EvarphiC}{E_{\varphiC}}
\nc{\EvarphiB}{E_{\varphiB}}
\nc{\GammaA}{\Gamma_{\hspace{-0.4ex}A}}
\nc{\LA}{L_{\hspace{-0.1ex}A}}
\nc{\LB}{L_{\hspace{-0.1ex}B}}
\nc{\LeT}{L_e\cat T}
\nc{\LfT}{L_{\hspace{-0.1ex}f}\cat T}
\nc{\LAT}{\LA\cat T}
\nc{\LBT}{\LB\cat T}
\nc{\LAS}{\LA\cat S}
\nc{\LBS}{\LB\cat S}
\nc{\LKBT}{L_{\KB}\cat T}
\nc{\KB}{K(\cat B)}
\nc{\KS}{K\hspace{-0.2ex}(S)}
\nc{\abcX}{a\hspace{-0.3ex}\bc{X}}
\nc{\abcx}{a\hspace{-0.3ex}\bc{x}}
\nc{\meet}{\wedge}%\curlywedge
\nc{\overbar}[1]{\mkern 1.5mu\overline{\mkern-1.5mu#1\mkern-1.5mu}\mkern 1.5mu}

\usepackage[a]{esvect}
\nc{\weaklyfinite}{weakly closed}
\nc{\finite}{closed}
\nc{\BCdual}[1]{{#1}^*}
\nc{\LCore}{\mathrm{LCore}}
\nc{\Stovicek}{\v{S}\v{t}ov\'{i}\v{c}ek}
\nc{\ftriple}{f_{\natural}}
\nc{\unitC}{\unit_{\cat C}}%
\nc{\unitD}{\unit_{\cat D}}%
\nc{\Pone}{{\mathbb{P}^1}}
\nc{\InvSupp}[1]{\Supp_{\cat T}^{-1}(#1)}%this is pretty bad notation; beren
\nc{\InvCosupp}[1]{\Cosupp_{\cat T}^{-1}(#1)}
\nc{\closureP}{\overbar{\{\cat P\}}}
\nc{\closureQ}{\overbar{\{\cat Q\}}}
\nc{\singP}{\{\cat P\}}
\nc{\singQ}{\{\cat Q\}}
\nc{\singm}{\{\frak m\}}
\dmo{\Inj}{Inj}
\dmo{\tfib}{tfib}
\dmo{\tcof}{tcof}
\dmo{\Aut}{Aut}
\dmo{\tofib}{tofib}
\dmo{\surj}{surj}
\dmo{\Excs}{Exc}
\dmo{\Homog}{Homog}
\dmo{\PSh}{PSh}
\dmo{\Epis}{Epi}

\dmo{\KInjdmo}{K}
\dmo{\Dbdmo}{mod}
\dmo{\sur}{sur}
\nc{\KInj}[1]{\KInjdmo(\Inj #1)}
\nc{\Dbmod}[1]{\Der^b(\Dbdmo #1)}
\dmo{\Viss}{vis}%lol
\nc{\Vis}{\Viss}
\nc{\vis}{\Vis}
\nc{\kappaaux}{g}
\nc{\kappaCh}{{\kappaaux(\cat C_h)}}
\nc{\kappam}{{\kappaaux({\frak m})}}
\nc{\kappaP}{{\kappaaux_{\cat P}}}
\nc{\kappaQ}{{\kappaaux(\cat Q)}}
\nc{\kappaCP}{{\kappaaux_{\cat C}(\cat P)}}
\nc{\kappaDP}{{\kappaaux_{\cat D}(\cat P)}}
\nc{\kappaCQ}{{\kappaaux_{\cat C}(\cat Q)}}
\nc{\kappaDQ}{{\kappaaux_{\cat D}(\cat Q)}}
\nc{\kappaphiB}{{\kappaaux(\phi(\cat B))}}
\nc{\kappaphiQ}{{\kappaaux(\varphi(\cat Q))}}
\nc{\halfplus}{{\scriptscriptstyle\top}}
\dmo{\Sub}{Sub}
\nc{\SpEn}{\cat S_{E(n)}}
\nc{\SpEnf}{\cat S_n}
\nc{\Lcomp}{L^{\mathrm{com}}} %I made this and the next one commands because I'm unsure of the choice of notation
\nc{\Ucomp}{U^{\mathrm{com}}}
\nc{\bbullet}{{\scriptscriptstyle\hspace{-1pt}\bullet}}
\nc{\bullett}{{\scriptscriptstyle\bullet}\hspace{-1pt}}
\nc{\LF}{L\hspace{-0.2ex}F}
\dmo{\StMod}{StMod}
\dmo{\Proj}{Proj}
\nc{\SpG}{\Sp^G}
\nc{\EG}{\bbE_G}
\nc{\DEG}{\Der(\EG)}
\nc{\DE}{\Der(\bbE)}
\nc{\Prst}{{\cat P}\mathrm{r^{st}}}
\nc{\Mack}{\mathrm{Mack}}
\nc{\SC}{S\cat C}
\dmo{\fin}{{fin}}
\dmo{\DM}{DM}
\dmo{\fp}{fp}
\nc{\DMQ}{\DM_Q}
\dmo{\DerKal}{DMack}
\dmo{\coh}{coh}
\dmo{\Der}{D}
\dmo{\DMot}{DMot}
\dmo{\Cell}{Cell}
\dmo{\rmH}{H}
\dmo{\piu}{\underline{\pi}}
\dmo{\Sphere}{\mathbb{S}}
\nc{\HA}{{\rmH \hspace{-0.2em}\bbA}}
\nc{\HZ}{{\rmH \hspace{-0.2em}\bbZ}}
\nc{\HZbar}{{\rmH \hspace{-0.2em}\underline{\bbZ}}}
\nc{\HbbF}{{\rmH \hspace{-0.15em}\mathbb{F}}}
\nc{\Fp}{{\bbF_{\hspace{-0.1em}p}}}
\nc{\HFp}{{\rmH \hspace{-0.15em}\bbF_{\hspace{-0.1em}p}}}
\nc{\HQ}{{\rm H \bbQ}}
\nc{\DHZG}{\Der(\HZ_G)}
\nc{\DHZH}{\Der(\HZ_H)}
\nc{\DHZK}{\Der(\HZ_K)}
\nc{\DHZGN}{\Der(\HZ_{G/N})}
\nc{\DHZGG}{\Der(\HZ_{G/G})}
\nc{\DHZCp}{\Der(\HZ_{C_p})}
\nc{\DHZGprime}{\Der(\HZ_{G'})}
\nc{\DHZ}{\Der(\HZ)}
\nc{\frakp}{\mathfrak{p}}
\nc{\frakq}{\mathfrak{q}}
\nc{\frakS}{\mathfrak{S}}
\nc{\frakT}{\mathfrak{T}}
\nc{\Z}{\mathbb{Z}}
\nc{\F}{\mathbb{F}}
\nc{\SSG}{\text{sSet}_*^G}
\nc{\sSet}{\text{sSet}}

\dmo{\csupp}{csupp}
\dmo{\Con}{Conj}
\dmo{\Id}{Id}
\dmo{\rmK}{\textrm{\rm K}}
\dmo{\Spc}{Spc}
\dmo{\thick}{thick}
\dmo{\thickid}{thickid}
\nc{\thicko}[1]{\thickid\langle #1 \rangle}
\nc{\thickt}[1]{\thick_\otimes\langle #1 \rangle}
\dmo{\cone}{cone}
\dmo{\End}{End}
\dmo{\Derperf}{D_{perf}}
\dmo{\Mor}{Mor}
\dmo{\id}{id}
\dmo{\incl}{incl}
\dmo{\Img}{Im}
\dmo{\im}{im}
\dmo{\Ker}{Ker}
\dmo{\ind}{ind}
\dmo{\CoInd}{coind}
\dmo{\GH}{GH}
\dmo{\idem}{e}
\dmo{\res}{res}
\dmo{\infl}{infl}
\dmo{\Derqc}{D_{qc}}
\nc{\DbcohX}{\Der^b(\coh X)}
\dmo{\triv}{triv}
\dmo{\Pic}{Pic}
\dmo{\dual}{dual}
\dmo{\Tel}{Tel} %telescope
\dmo{\grMod}{grMod}%
\dmo{\Mod}{Mod}%
\dmo{\opname}{op}
\dmo{\SH}{SH}% ground name for cat of spectra
\dmo{\smallb}{b}% ground exponent for ``bounded''
\dmo{\Spec}{Spec}
\dmo{\supp}{supp}
\dmo{\Supp}{Supp}
\dmo{\crosseffec}{cr}

\dmo{\thofib}{tofib}
\dmo{\hofib}{hofib}
\dmo{\cosupp}{cosupp}
\dmo{\Cosupp}{Cosupp}
\nc{\SHc}{{\SH^c}}
\nc{\SHp}{{\SH_{(p)}}}
\nc{\SHcp}{{\SH^c_{(p)}}}
\nc{\SHG}{\SH(G)}
\nc{\SHGp}{\SH(G)_{(p)}}
\nc{\SHGc}{\SHG^c}
\nc{\SHGcp}{\SHG^c_{(p)}}
\nc{\quadtext}[1]{\quad\textrm{#1}\quad}
\nc{\qquadtext}[1]{\qquad\textrm{#1}\qquad}
\nc{\adj}{\dashv}
\nc{\adjto}{\rightleftarrows}
\nc{\bbL}{\mathbb{L}}
\nc{\bbS}{\mathbb{S}}
\nc{\bbA}{\mathbb{A}}
\nc{\bbE}{\mathbb{E}}
\nc{\bbN}{\mathbb{N}}
\nc{\bbQ}{\mathbb{Q}}
\nc{\bbZ}{\mathbb{Z}}
\nc{\bbR}{\mathbb{R}}
\nc{\bbF}{\mathbb{F}}
\nc{\cat}[1]{\mathscr{#1}}%or: \nc{\cat}[1]{\mathcal{#1}}
\nc{\ie}{{\sl i.e.}, }
\nc{\into}{\mathop{\rightarrowtail}}
\nc{\inv}{^{-1}}
\nc{\isoto}{\mathop{\overset{\sim}\to}}
\nc{\isotoo}{\mathop{\overset{\sim}\too}}
\nc{\onto}{\mathop{\twoheadrightarrow}}
\nc{\too}{\mathop{\longrightarrow}\limits}
\nc{\mapstoo}{\longmapsto}
\nc{\adh}[1]{\overline{#1}}% adherence
\nc{\adhpt}[1]{\adh{\{#1\}}}% adherence of a pt
\nc{\aka}{{a.\,k.\,a.}\ }
\nc{\calF}{\mathcal{F}}
\nc{\eg}{{\sl e.\,g.}}
\nc{\hook}{\hookrightarrow}

\nc{\ideal}[1]{\langle #1\rangle}
\dmo{\red}{red}
\usepackage{makecell}
\dmo{\Hom}{Hom}
\nc{\Homcat}[1]{\Hom_{\cat #1}}
\nc{\iHom}{\mathcal{H}\mathrm{om}}
\nc{\Supph}{\Supp^h}
\nc{\Supphnaive}{\Supp^n}
\nc{\Cosupph}{\Cosupp^h}
\nc{\ihom}[1]{\mathsf{hom}(#1)}
\nc{\ihomC}[1]{\mathsf{hom}_{\cat C}(#1)}
\nc{\ihomD}[1]{\mathsf{hom}_{\cat D}(#1)}
\nc{\ihomsub}[2]{\mathsf{hom}_{#1}(#2)}
\usepackage{longtable}
\nc{\Mid}{\,\big|\,}
\nc{\MMod}{\,\text{-}\Mod}%
\nc{\GrMMod}{\,\text{-}\grMod}%
\nc{\op}{^{\opname}}
\nc{\oto}[1]{\overset{#1}\to}
\nc{\otoo}[1]{\overset{#1}{\,\too\,}}
\nc{\sminus}{\!\smallsetminus\!}
\nc{\poplus}[1]{^{\oplus #1}}%
\nc{\potimes}[1]{^{\otimes #1}}% tensor power
\nc{\sbull}{{\scriptscriptstyle\bullet}}%\mathbf{\cdot}}%{}}
\nc{\SET}[2]{\left\{\,#1\middle\vert#2\,\right\}}
\nc{\SETT}[2]{\left\{\,#1\,\middle\vert\,#2\,\right\}}
\nc{\SpcK}{\Spc(\cat K)}% most used
\nc{\then}{\Rightarrow}
\nc{\unit}{\mathbb{1}}% unit for \otimes
\nc{\xra}{\xrightarrow}
\nc{\phigeom}[1]{\widetilde{\Phi}^{#1}}
\dmo{\Oname}{O}
\dmo{\proper}{proper}% for proper subgroups
\dmo{\lenormal}{\unlhd}
\dmo{\fib}{fib}
\dmo{\cofib}{cofib}
\dmo{\lnormal}{\lhd}
\nc{\normal}{\trianglelefteq}%\lhd
\nc{\Op}{\Oname^p}% O^p for maximal p-normal subgroup
\nc{\Oq}{\Oname^q}% as above for p=q
\newcommand{\Sp}{{\mathscr{S}p}}

\dmo{\Ho}{Ho}
\dmo{\CB}{CB}
\dmo{\Fin}{Fin}
\dmo{\add}{add}
\dmo{\Fun}{Fun}
\dmo{\Ext}{Ext}
\dmo{\CAlg}{CAlg}
\dmo{\CMon}{CMon}
\dmo{\CC}{\cat C} %beren: I changed these, but left the O
\dmo{\DD}{\cat D}
\dmo{\OO}{\mathcal{O}}
\dmo{\Map}{Map}
\dmo{\Span}{Span}
\dmo{\Tot}{Tot}
\dmo{\N}{N}
\dmo{\Cat}{Cat}
\dmo{\colim}{colim}
\dmo{\hocolim}{hocolim}
\dmo{\Ch}{Ch}
\dmo{\A}{\mathbb{A}^{eff}}
\nc{\AGeff}{\mathbb{A}_G^{\mathrm{eff}}}
\nc{\BGeff}{\mathcal{B}_G^{\mathrm{eff}}}
\nc{\BG}{{\mathcal{B}_G}}
\nc{\NBGeff}{{\N}{\BGeff}}
\dmo{\Ab}{Ab}
\nc{\Smith}{\mathsf{Smith}}
\nc{\Floyd}{\mathsf{Floyd}}
\nc{\blue}{\beth^{\mathrm{geom}}}
\dmo{\Set}{Set}
\dmo{\ev}{ev}
\dmo{\Spcl}{Spcl}
\nc{\Funadd}{\Fun_{\add}}
\dmo{\proj}{proj}
\dmo{\cof}{cof}
\nc{\cPd}{\cat P_{\hspace{-.1em}d}}
\nc{\cPm}{\cat P_{\hspace{-.1em}m}}
\nc{\Chp}{\mathsf{Ch}_p}
\nc{\Pp}{\mathsf{P}_p}

\dmo{\Coideal}{Coideal}
\dmo{\gen}{gen}
\nc{\auxcoidealsymb}{\vartriangleleft}
\dmo{\Loc}{Loc}
\dmo{\Ind}{Ind}
\dmo{\Coloc}{Coloc}
\dmo{\Locideal}{Locid}
\dmo{\Colocideal}{Coloc}
\nc{\LOCO}{\Locideal}
\nc{\COLOCO}{\Colocideal}
\dmo{\Kos}{Kos}
\nc{\Loco}[1]{\LOCO\langle #1 \rangle}
\nc{\Coloco}[1]{\COLOCO\langle #1 \rangle}
\nc{\LambdaP}{\Lambda^{\cat P}} %beren: I've added this command here because we might need to make some spacing changes to make the typesetting less ugly
\nc{\LambdaQ}{\Lambda^{\cat Q}} %beren: I've added this command here because we might need to make some spacing changes to make the typesetting less ugly
\nc{\GammaP}{\Gamma_{\cat P}} %beren: I've added this command here because we might need to make some spacing changes to make the typesetting less ugly
\nc{\GammaQ}{\Gamma_{\cat Q}} %beren: I've added this command here because we might need to make some spacing changes to make the typesetting less ugly
\nc{\LambdaW}{\Lambda^{\hspace{-0.3ex}W}} %beren: I've added this command here because we might need to make some spacing changes to make the typesetting less ugly
\nc{\GammaW}{\Gamma_{\hspace{-0.3ex}W}} %beren: I've added this command here because we might need to make some spacing changes to make the typesetting less ugly
\nc{\gW}{g_W}
\nc{\gP}{g_{\cat P}}
\nc{\gQ}{g_{\cat Q}}
\nc{\cC}{{\cat C}}
\nc{\cT}{{\cat T}}
\nc{\cD}{{\cat D}}

\nc{\mT}{\kern-0.5em\mod\kern-0.1em\text{-}\cat{T}^c}
\nc{\mTc}{\kern-0.5em\mod\kern-0.1em\text{-}\cat{T}^c}
\nc{\MTc}{\Mod\kern-0.1em\text{-}\cat{T}^c}
\nc{\MT}{\Mod\kern-0.1em\text{-}\cat{T}}
\newcounter{enum-resume-hack}
\usepackage{wasysym}%
\usepackage{caption}
\Crefname{Thm}{Theorem}{Theorems}
\Crefname{Prop}{Proposition}{Propositions}

\usepackage{makecell}
\usepackage{tablefootnote}
\usepackage{arydshln}
\usepackage{booktabs}
\usepackage{quiver}
\usepackage[all]{xy}
\newdir{ >}{{}*!/-10pt/\dir{>}}

\makeatletter
\providecommand*{\twoheadrightarrowfill@}{%
  \arrowfill@\relbar\relbar\twoheadrightarrow
}
\providecommand*{\twoheadleftarrowfill@}{%
  \arrowfill@\twoheadleftarrow\relbar\relbar
}
\providecommand*{\xtwoheadrightarrow}[2][]{%
  \ext@arrow 0579\twoheadrightarrowfill@{#1}{#2}%
}
\providecommand*{\xtwoheadleftarrow}[2][]{%
  \ext@arrow 5097\twoheadleftarrowfill@{#1}{#2}%
}
\makeatother

\nc{\cL}{\mathcal{L}}
\nc{\cP}{\mathcal{P}}

\nc{\tblue}{\beth^{\mathrm{Tate}}}

\nc{\cA}{\mathcal{A}}
\nc{\cF}{\mathcal{F}}
\nc{\Fnt}{\cF_{\mathrm{nt}}}
\nc{\Fall}{\cF_{\mathrm{all}}}
\nc{\Ftriv}{\cF_{\mathrm{triv}}}

\nc{\bE}{\underline{E}}

\usepackage{adjustbox}

\Crefname{Thm}{Theorem}{Theorems}
\Crefname{Prop}{Proposition}{Propositions}
\Crefname{Lem}{Lemma}{Lemmas}
\Crefname{Cor}{Corollary}{Corollaries}
\Crefname{Exa}{Example}{Examples}
\Crefname{ThmAlpha}{Theorem}{Theorems}
\Crefname{Rem}{Remark}{Remarks}
\nc{\htimes}{\widehat{\otimes}}
\nc{\frakm}{\mathfrak{m}}
\usepackage{pict2e}
\usepackage[most]{tcolorbox}

\makeatletter
\newcommand{\cveewedge@measure}[2]{%
  \sbox\z@{$#1\m@th#2$}%
  \dimen@=1.05\ht\z@
  \unitlength=.005\wd\z@
  \count@=\dimen@ 
  \divide\count@\unitlength
  \ifx#1\scriptstyle
    \linethickness{0.8\@wholewidth}%
  \else
    \ifx#1\scriptscriptstyle
      \linethickness{0.65\@wholewidth}%
    \fi
  \fi
}

\newcommand{\cwedge}{\mathbin{\mathpalette\do@cwedge\relax}}
\newcommand{\do@cwedge}[2]{%
  \cveewedge@measure{#1}{\wedge}
  \begin{picture}(200,\count@)
  \roundjoin
  \polygon(25,0)(100,\count@)(175,0)
  \end{picture}%
}
\newcommand{\cvee}{\mathbin{\mathpalette\do@cvee\relax}}
\newcommand{\do@cvee}[2]{%
  \cveewedge@measure{#1}{\vee}
  \begin{picture}(200,\count@)
  \roundjoin
  \polygon(25,\count@)(100,0)(175,\count@)
  \end{picture}%
}
\makeatother
\makeatletter
\newcommand{\leqnomode}{\tagsleft@true\let\veqno\@@leqno}
\newcommand{\reqnomode}{\tagsleft@false\let\veqno\@@eqno}
\makeatother
\newcommand{\bc}[1]{\mathinner{\left\langle #1 \right\rangle}_{\mkern-2mu *}}
\newcommand{\cbc}[1]{\mathinner{\langle #1 \rangle}^{\mkern-2mu *}}
\newcommand{\bcnothuge}[1]{\mathinner{\Bigl\langle #1 \Bigl\rangle_*}}

\title{Local Bousfield classes via homological support} 

\author[T.~Barthel]{Tobias Barthel}
\author[N.~Castellana]{Nat{\`a}lia Castellana}
\author[D.~Heard]{Drew Heard}
\author[B.~Sanders]{\\Beren Sanders}
\author[C.~Zou]{Changhan Zou}

\makeatletter
\patchcmd{\@setaddresses}{\indent}{\noindent}{}{}
\patchcmd{\@setaddresses}{\indent}{\noindent}{}{}
\patchcmd{\@setaddresses}{\indent}{\noindent}{}{}
\patchcmd{\@setaddresses}{\indent}{\noindent}{}{}
\makeatother

\address{Tobias Barthel, Max Planck Institute for Mathematics, Vivatsgasse 7, 53111 Bonn, Germany}
\email{tbarthel@mpim-bonn.mpg.de}
\urladdr{\href{https://sites.google.com/view/tobiasbarthel/home}{https://sites.google.com/view/tobiasbarthel/home}}

\address{Nat{\`a}lia Castellana, Departament Matem\`atiques, Universitat Aut\`onoma de Barcelona and Centre de Recerca Matem\`atica, 08193 Bellaterra, Spain}
\email{Nnatalia.Castellana@mat.uab.cat}
\urladdr{\href{https://mat.uab.cat/~natalia}{https://mat.uab.cat/$\sim$natalia}}

\address{Drew Heard, Department of Mathematical Sciences, Norwegian University of Science and Technology, Trondheim}
\email{drew.k.heard@ntnu.no}
\urladdr{\href{https://folk.ntnu.no/drewkh/}{https://folk.ntnu.no/drewkh}}

\address{Beren Sanders, Mathematics Department, UC Santa Cruz, 95064 CA, USA}
\email{beren@ucsc.edu}
\urladdr{\href{https://people.ucsc.edu/~beren/}{https://people.ucsc.edu/$\sim$beren}}

\address{Changhan Zou, Mathematics Department, UC Santa Cruz, 95064 CA, USA}
\email{czou3@ucsc.edu}
\urladdr{\href{https://people.ucsc.edu/~czou3/}{https://people.ucsc.edu/$\sim$czou3}}

\AddToHook{env/Thm/begin}{     \crefalias{equation}{Thm}}
\AddToHook{env/Exa/begin}{     \crefalias{equation}{Exa}}
\AddToHook{env/Prop/begin}{    \crefalias{equation}{Prop}}
\AddToHook{env/Lem/begin}{     \crefalias{equation}{Lem}}
\AddToHook{env/Cor/begin}{     \crefalias{equation}{Cor}}
\AddToHook{env/Conj/begin}{    \crefalias{equation}{Conj}}
\AddToHook{env/Rem/begin}{    \crefalias{equation}{Rem}}
\AddToHook{env/Def/begin}{    \crefalias{equation}{Def}}

\begin{document}
\begin{abstract}
	Given an object $A$ 
	in
	a
	big tensor-triangulated category,
	we study the homological and cohomological Bousfield classes of the associated localization: the tensor-triangulated category of $A$-local objects.
	We show that the homological support classifies the homological Bousfield classes of the $A$-local category precisely when an $A$-relative form of the homological detection property holds.
	Moreover, we prove that this 
	holds
	if and only if $A$ is Bousfield equivalent to a coproduct of homological residue fields.
	The analogous classification of
	cohomological Bousfield classes by homological cosupport is strictly stronger: it is equivalent to an $A$-relative form of homological stratification.
	This equivalence between 
	stratification and the 
	classification of cohomological 
	Bousfield
	classes is new even in the absolute case.
	A further surprise is that stratification is also equivalent 
	to the classification of homological Bousfield classes
	together with the statement that every cohomological Bousfield class is homological.
	Applied to chromatic homotopy theory, these results 
	classify
	the homological Bousfield classes of any localization 
	of spectra
	with respect to a coproduct of Morava $K$-theories.
	This covers many localizations of interest.
	We also completely characterize when such chromatic localizations are relatively homologically stratified.
	This yields new examples of cohomological Bousfield classes that are not homological.
	In particular, it answers
	a question of Wolcott 
	concerning the category of harmonic spectra.
	Our examples are produced by exhibiting local spectra with empty homological cosupport.

\end{abstract}
\maketitle

{
\hypersetup{linkcolor=black}
\tableofcontents
}

\section{Introduction}
A fundamental structural problem 
for a 
tensor-triangulated category 
is to understand its localizations, which concretely amounts to classifying its localizing ideals.
While this has been achieved 
for 
many significant examples, it remains wide open for others, including the prototypical example of the
stable homotopy category.
Confronted with this complexity, we can narrow our focus to two special classes of localizing ideals: the homological Bousfield classes and the cohomological Bousfield classes. 
These are the localizing ideals of the form
\[
  \bc{E}=\SETT{X}{E\otimes X=0}
  \quad\text{ and }\quad
  \cbc{E}=\SETT{X}{\ihom{X,E}=0}
\]
for some generating object $E$.
Morally, 
they
correspond to 
the localizations 
which are
induced by homology and cohomology theories, respectively.

The aim of this paper is to study these classification problems for categories which arise as Bousfield localizations~$\LAT$ of rigidly-compactly generated tt-categories~$\cat T$, for any object $A$ in $\cat T$. These categories of $A$-local objects are almost never rigidly-compactly generated,
so existing frameworks for studying localizing ideals do not apply directly.
On the other hand, the general results of this paper specialize to absolute statements about~$\cat T$ by taking $A=\unit$.

Every homological Bousfield class in $\LAT$ is cohomological (\cref{thm:injection-hbc-cbc})
so we have 
inclusions
\[
\resizebox{\columnwidth}{!}{$
\begin{tikzcd}[column sep=small, row sep=small, ampersand replacement=\&]
    \left\{\begin{array}{c}\text{homological Bousfield}\\[-2pt]\text{classes of $\LAT$}\end{array}\right\}
    \arrow[phantom,"\subseteq", r]
    \&
    \left\{\begin{array}{c}\text{cohomological Bousfield}\\[-2pt]\text{classes of $\LAT$}\end{array}\right\}
    \arrow[phantom,"\subseteq", r]
    \&
    \left\{\begin{array}{c}\text{localizing}\\[-2pt]\text{ideals of $\LAT$}\end{array}\right\}.
\end{tikzcd}
$}
\]
For the stable homotopy category of spectra, 
the homological Bousfield classes exhibit remarkable complexity,
as seen in the pioneering work 
of Bousfield \cite{Bousfield79_boolean} and studied further in \cite{HoveyPalmieri99} and \cite{Strickl2019Combinatorial}.
In this example, Hovey~\cite{hovey-cbc} conjectured that every cohomological Bousfield class is homological, 
making the first inclusion an equality;
later, Hovey and Palmieri made the stronger conjecture that all three collections coincide \cite[Conjecture~9.1]{HoveyPalmieri99}.

These conjectures are false for general tt-categories.
In fact, we will see that many localizations of the category of spectra have cohomological Bousfield classes 
that are not homological.
Thus the relationship between the Bousfield classes of $\cat T$ and those of $\LAT$ can be 
subtle: a localization $\LAT$ may exhibit a greater divergence between homological and cohomological classes 
compared with~$\cat T$ itself.

Nevertheless, our general approach to understanding $\LAT$ is by using tools
which are available for~$\cat T$.
In particular,
we define a support theory on~$\LAT$ by using Balmer's homological support $\Supph$ for the 
ambient category $\cat T$:
\[\Supp^h_A(x) \coloneqq \Supp^h(A\otimes x) \subseteq \Supp^h(A)\]
for each $x \in \LAT$.
We say that relative h-detection holds for $\LAT$ if $\Supp^h_A(x) = \emptyset$ implies $x=0$.
There is also a ``naive'' homological support $\Supphnaive_A$
which coincides with $\Supp^h_A$ in all known examples.

Our first main result (\cref{thm:bousfield-classes})
identifies relative h-detection as the precise 
condition under which 
this
support theory
classifies homological Bousfield classes:

\begin{Thm}\label{thmx:intro-hbc}
Let $\cat T$ be a rigidly-compactly generated tt-category and let $A\in\cat T$. The following are equivalent:
\begin{enumerate}
	\item The relative h-detection property holds for $\LAT$.
  \item We have $\Supph_A=\Supphnaive_A$ and 
the map $\bc{x}\longmapsto \Supph_A(x)$ provides
	  a lattice isomorphism
  \[
    \left\{\begin{array}{c}\text{homological Bousfield}\\[-2pt]\text{classes of $\LAT$}\end{array}\right\}
    \xrightarrow{\ \sim\ }\bigl\{\text{subsets of }\Supph(A)\bigr\}.
  \]
 \end{enumerate}
\end{Thm}

\pagebreak[2]
We then show in 
\cref{thm:generalized-harmonic}
that the 
h-detection property can be forced.
Namely, for any
set of homological primes~$S \subseteq \Spc^h(\cat T^c)$,
the localization
$L_{\KS}\cat T$ always satisfies relative h-detection,
where $\KS\coloneqq \coprod_{\cat B \in S} \EB$ is the coproduct of the corresponding pure-injective objects.
In fact, the converse also holds:

\begin{Thm}
	Let $\cat T$ be a rigidly-compactly generated tt-category with the property that $\Supp^h(t) = \Supphnaive(t)$ for all $t \in \cat T$. For any $A\in \cat T$, the following statements are equivalent:
	\begin{enumerate}
		\item The relative h-detection property holds for $\LAT$.
		\item There is an equality $\bc{A} = \bc{\KS}$ for some $S \subseteq \Spc^h(\cat T^c)$.
	\end{enumerate}
\end{Thm}

\noindent
In particular, we obtain a classification of the homological Bousfield classes of~$\LAT$ whenever $A \sim \KS$ is Bousfield equivalent to a coproduct of homological residue fields.
This gives a conceptual explanation for the many examples in which the homological Bousfield lattice is known to be Boolean.

Taking $S=\Spc^h(\cat T^c)$, we obtain that $\cat T \to L_{\KS}\cat T$
is the universal Bousfield localization which satisfies relative h-detection.
At the other extreme, we can consider the 
case of a singleton: $S=\{\cat B\}$.
Under some hypotheses, 
the localization~$L_{K(\cat B)}\cat T$ 
can be interpreted
as a ``completed stalk'' of~$\cat T$.
See~\cref{cor:almost-strat}.

To apply these theorems when the homological residue fields are not explicitly known, 
we establish some novel base-change results in \cref{prop:base-change},
which may be of independent interest.
In~\cref{thm:bc-chromatic-equiv-localizations}, we apply these results to localizations of $G$-spectra for any compact Lie group~$G$.
As an explicit illustration of these techniques,
	we compute the Bousfield lattice
	of the localization of $C_2$-spectra with respect to the real Johnson--Wilson spectrum $E_{\bbR}(n)$.
	See \cref{exa:johnson-wilson}.

There is a parallel story for cohomological Bousfield classes,
but it exhibits a surprising phenomenon.
We show in~\cref{thm:cbc-equivalence} that
the classification of cohomological Bousfield classes by homological cosupport is equivalent to the classification of \emph{all} localizing ideals of~$\LAT$ by relative homological support, a condition we call relative h-stratification
(\cref{def:relative-stratification}).
Moreover, the classification of all localizing ideals is also equivalent to the classification of homological Bousfield classes together with the statement that every cohomological Bousfield class is homological.
These results
appear to be new even in the absolute case $A=\unit$.

\begin{Thm}
	Let $\cat T$ be a rigidly-compactly generated tt-category and let $A\in\cat T$.
	The following are equivalent:
    \begin{enumerate}
        \item $\LAT$ is relatively h-stratified. 
  \item $\Supph(A) = \Supphnaive(A)$ and the map
	  $\cbc{x} \mapsto \Cosupph(x)$ provides a bijection

  \[
    \bigl\{\text{cohomological Bousfield classes of }\LA\cat T\bigr\}
    \xrightarrow{\ \sim\ }
    \bigl\{\text{subsets of }\Supph(A)\bigr\}.
  \]
		\item $\LAT$ has relative h-detection and every cohomological Bousfield class of $\LAT$ is homological.
    \end{enumerate}
\end{Thm}

\noindent
The first characterization connects the classification of localizing ideals
with the question of whether every localizing ideal is a cohomological Bousfield class,
a question known to lead into deep set-theoretic waters \cite{CasacubertaGutierrezRosicky14}.
The second highlights the significance of asking whether 
every cohomological 
class is homological.

\smallskip
We may summarize the picture as follows:

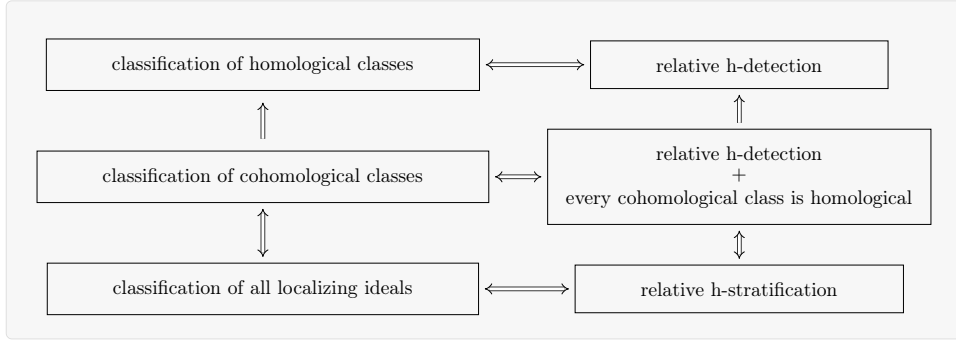
\begin{center}
\begin{tcolorbox}[
  colback=black!3,
  colframe=black!15,
  boxrule=0.3pt,
  arc=2pt,
  width=1.0\textwidth,
  left=6pt,
  right=6pt,
  top=6pt,
  bottom=6pt
]
\begin{equation*}\label{eq:big-figure}\resizebox{\columnwidth}{!}{
		\begin{tikzpicture}[mybox/.style={draw, inner sep=5pt}]
		\node[mybox] (row1col1) at (-3.5,-0.5){%
		  \begin{tikzcd}[ampersand replacement=\&]
			\& \text{classification of homological classes} \&
		  \end{tikzcd}
		};
		\node[mybox] (row1col2) at (5,-0.5){%
		  \begin{tikzcd}[ampersand replacement=\&]
			\& \text{relative h-detection} \&
		  \end{tikzcd}
		};
		\node[mybox] (row2col1) at (-3.5,-2.5) {
		  \begin{tikzcd}[ampersand replacement=\&]
			\& \text{classification of cohomological classes} \&
		  \end{tikzcd}
			};
		\node[mybox] (row2col2) at (5,-2.5) {
		  \begin{tikzcd}[ampersand replacement=\&,row sep=-3pt]
			\text{relative h-detection} \\
			+ \\
			\text{every cohomological class is homological} 
		  \end{tikzcd}
			};
		\node[mybox] (row3col1) at (-3.5,-4.5) {
		  \begin{tikzcd}[ampersand replacement=\&]
			\& \text{classification of all localizing ideals} \&
		  \end{tikzcd}
			};
		\node[mybox] (row3col2) at (5,-4.5) {
		  \begin{tikzcd}[ampersand replacement=\&]
			\& \text{relative h-stratification} \&
		  \end{tikzcd}
		};
		\draw[Implies-Implies,double equal sign distance] ([xshift=3pt]row1col1.east) -- ([xshift=-3pt]row1col2.west);
		\draw[-{Implies},double equal sign distance] ([yshift=6pt]row2col1.north) -- ([yshift=-6pt]row1col1.south);
		\draw[-{Implies},double equal sign distance] ([yshift=3pt]row2col2.north) -- ([yshift=-3pt]row1col2.south);
		\draw[Implies-Implies,double equal sign distance] ([xshift=3pt]row2col1.east) -- ([xshift=-3pt]row2col2.west);
		\draw[Implies-Implies,double equal sign distance] ([xshift=3pt]row3col1.east) -- ([xshift=-3pt]row3col2.west);
		\draw[Implies-Implies,double equal sign distance] ([yshift=3pt]row3col1.north) -- ([yshift=-3pt]row2col1.south);
		\draw[Implies-Implies,double equal sign distance] ([yshift=3pt]row3col2.north) -- ([yshift=-3pt]row2col2.south);
		\end{tikzpicture}
}\end{equation*}
\end{tcolorbox}
\end{center}

Our motivating examples come from chromatic homotopy theory:

\begin{Thm}
	Let $p$ be a prime and let 
	$S \subseteq \bbN \cup \{\infty\}$ be any subset.
Let $K(S)\coloneqq \coprod_{i\in S}K(i)$ denote the associated wedge of $p$-local Morava $K$-theories.
	\begin{enumerate}
		\item The category $L_{\KS}\Sp$ has relative h-detection
			and there is a bijection
\[
	\big\{\text{homological Bousfield classes of }L_{\KS}\Sp\big\}
  \overset{\cong}{\longrightarrow}
  \big\{\text{subsets of }S\big\}.
\]
\item The category $L_{\KS}\Sp$ is relatively h-stratified if and only if $S$ is finite and does not contain $\infty$. In this case there is a bijection
\[
	\big\{\text{cohomological Bousfield classes of }L_{\KS}\Sp\big\}
  \overset{\cong}{\longrightarrow}
  \big\{\text{subsets of }S\big\}.
\]
\item There are cohomological Bousfield classes in $\Sp_{\KS}$ which are not homological if and only if $\Sp_{\KS}$ is not relatively h-stratified, that is, if and only if $S$ is infinite or contains $\infty$.
	\end{enumerate}
\end{Thm}

This is the content of
\cref{cor:bc-chromatic-localizations} and \cref{cor:harmonic-failure}.
The first statement recovers the known classification of homological Bousfield classes in the harmonic category (i.e., in the case $S = \mathbb{N}$) and extends it to all chromatic localizations of the 
form~$L_{\KS}\Sp$. 
The second gives a negative answer to a question of Wolcott concerning the harmonic category.

Our examples of ``exotic'' cohomological Bousfield classes, 
i.e., cohomological classes that are not homological,
are produced by exhibiting local spectra with empty homological cosupport.
Two key examples are $\mathbb S/p$ and $BP/p$
but we also obtain a countably infinite family of spectra $T(n)/p$ which interpolate between these two extremes; see~\cref{thm:harmonic-failure-Tn}
and~\cref{thm:exotic-HFp}.
In particular, $L_{\HFp}\Sp$ is 
a category
which has infinitely many cohomological Bousfield classes but only two homological Bousfield classes.
\Cref{prop:bpj} provides another infinite collection of exotic cohomological classes.
We conjecture that this collection is uncountable; see~\cref{rem:uncountable}.

In~\cref{sec:bad-behaviour}, we
study how such exotic cohomological Bousfield classes behave under 
localization.
In particular, we establish in~\cref{cor:anti-nilpotent-coho-lift}
that if~$\cat T$ has no nonzero tensor-nilpotent objects and every cohomological Bousfield class of $\cat T$ is homological, then the same is true for every localization $\LAT$.
This provides one perspective on why our exotic cohomological classes 
for
$L_{K(S)}\Sp$ do not
provide counterexamples to the Hovey--Palmieri conjecture concerning~$\Sp$.
The heart of the issue lies in the existence of tensor-nilpotent spectra.
We establish several additional results in this direction, which we will not detail here.
See~\cref{prop:idempotent-gluing} for example.

	A particularly interesting class of localizations are the smashing localizations. These are precisely the localizations $\cat T \to \LAT$ for which $\LAT$ is itself rigidly-compactly generated, so one can compare the relative notion of h-stratification inherited from $\cat T$ with the absolute notion internal to $\LAT$ introduced in~\cite{BarthelHeardSandersZou26}.
	Fortunately, these coincide by \cref{cor:relative-absolute}.
	In pursuing this analysis, we demonstrate some auxiliary results about smashing localizations
	which should be of independent interest. For
	example,
in \cref{prop:smash-injective} we 
describe the map on homological spectra induced by a smashing localization, which seems to be missing from the literature.
We also demonstrate that the \mbox{h-detection} property does not hold for corings by relating this property to the telescope conjecture. See~\cref{cor:detection-fails-corings}.

\subsection*{Acknowledgments}\label{ssec:thanks}

TB is supported by the European Research Council (ERC) under Horizon Europe (grant No.~101042990) and would like to thank the Max Planck Institute for its hospitality. NC is partially supported by Spanish State Research Agency project PID2024-158573NB-I00, the Severo Ochoa and Mar\'ia de Maeztu Program for Centers and Units of Excellence in R\&D (CEX2020-001084-M), and the CERCA Programme/Generalitat de Catalunya. DH would like to thank Utrecht University and the Max Planck Institute for their hospitality.

\section{Homological and cohomological Bousfield classes}

The standard reference for well generated triangulated categories is \cite{Neeman01}.

\begin{Def}
	Let $\cat T$ be a well generated tensor-triangulated category and let $A \in \cat T$. The \emph{homological Bousfield class} of $A$ is
	\[
		\bc{A} \coloneqq \SETT{X \in \cat T}{A \otimes X = 0}
	\]
	and the \emph{cohomological Bousfield class} of $A$ is
	\[
		\cbc{A} \coloneqq \SETT{X \in \cat T}{\ihom{X,A} = 0}.
	\]
	These are both localizing ideals of $\cat T$.
\end{Def}

\begin{Ter}
	We will say that the Bousfield classes $\bc{A}$ and $\cbc{A}$ are \emph{generated} by the object~$A$. Occasionally we will drop the ``Bousfield'' and simply speak of homological classes and cohomological classes.
\end{Ter}

\begin{Rem}
	Ohkawa \cite{Ohkawa89} proved that the collection of homological Bousfield classes in the category of spectra $\Sp$ forms a set of cardinality at most $2^{2^{\aleph_0}}$.
    This is now known to be the precise cardinality; see~\cite{Putzstuck2026pp}.
    The upper bound was 
    generalized by Iyengar and Krause \cite{IyengarKrause13}:
    If $\cat T$ is an $\alpha$-well generated tensor-triangulated category, then the set of homological Bousfield classes has cardinality at most $2^{2^{|\cat T^{\alpha}|}}$ where~$\cat T^{\alpha}$ denotes the full subcategory of $\alpha$-compact objects.
\end{Rem}

\begin{Def}\label{def:bousfield-lattice}
	Let $\cat T$ be a well generated tensor-triangulated category. The set of homological Bousfield classes $\BL(\cat T)$ is a complete lattice when ordered by reverse inclusion:
	\[
	  \bc{A}\le \bc{B}\ \Longleftrightarrow\ \bc{A}\supseteq \bc{B}.
	\]
	The top and bottom elements are $\bc{\unit}=\{0\}$ and $\bc{0}=\cat T$ respectively. The join is given by
	\begin{equation}\label{eq:join}
		\bigvee\SETT{\bc{A_i}}{i \in I} = 
		\big\langle\hspace{-0.2ex}
		\coprod\nolimits_{i \in I} \hspace{-0.2ex}A_i \big\rangle_*
	\end{equation}
	and the meet is given by 
	\begin{equation}\label{eq:meet}
		\bigwedge\SETT{\bc{A_i}}{i\in I} = \bigvee\SETT{\bc{X}}{\bc{X} \le \bc{A_i} \text{ for all } i \in I}.
	\end{equation}
	In general, the meet operation $\bc{A} \meet \bc{B}$ is different from $\bc{A} \otimes \bc{B} \coloneqq \bc{A \otimes B}$ although we always have $\bc{A} \otimes \bc{B}\leq \bc{A} \meet \bc{B}$. For example, the Brown--Comenetz dual of the sphere $I\in \Sp$ satisfies $I \otimes I = 0$; see \cite[Lemma~7.1]{HoveyPalmieri99}. Hence $\bc{I}\otimes \bc{I} = \bc{0}$ but $\bc{I} \meet \bc{I}=\bc{I}$ in $\BL(\Sp)$.
\end{Def}

\begin{Ter}
	We say that a localizing ideal $\cat L \subseteq \cat T$ is \emph{strictly localizing} if it is the kernel of a Bousfield localization on $\cat T$.
\end{Ter}

\begin{Exa}\label{exa:localization-is-well-generated}
	For any object $A \in \cat T$, the homological Bousfield class $\bc{A}$ is strictly localizing. This was originally proved for $\cat T=\Sp$ by Bousfield~\cite{Bousfield79} and subsequently generalized by Neeman~\cite[Appendix~D]{Neeman01}. For an arbitrary well generated tensor-triangulated category $\cat T$, see \cite[Proposition~2.1]{IyengarKrause13}. We will write $L_A\colon\cat T\to \cat T$ for the associated localization functor. The localized category~$\LAT$ inherits a tensor-triangulated structure from $\cat T$ since the kernel $\bc{A}$ is an ideal. Moreover, $\LAT$ is again well generated; see \cite[Theorem~7.2.1]{Krause10}. Note that if~$\bc{A} \le \bc{B}$ then the localization $\cat T\to \LAT$ factors uniquely through the localization $\cat T \to \LBT$. The initial localization is $\cat T \to L_{\unit}\cat T\cong \cat T$ and the final localization is $\cat T \to L_{0}\cat T\cong 0$.
\end{Exa}

\begin{Rem}
	In contrast, it is not known whether every cohomological Bousfield class $\cbc{A}$ is strictly localizing, nor is it known in general whether there is only a set of cohomological Bousfield classes. These questions have been related to foundational set-theoretic concerns; see \cite{Casacuberta23,CasacubertaGutierrezRosicky14,CasacubertaGutierrez24pp}. The study of cohomological Bousfield classes is thus subtle and intriguing.
\end{Rem}

\begin{Rem}\label{rem:bad-bousfield-classes}
	For $\cat T=\Sp$, Hovey proved that every homological class is cohomological; namely $\bc{X}=\cbc{IX}$ where $IX$ is the Brown--Comenetz dual of~$X$; see \cite[Proposition~1.1]{hovey-cbc}. This result holds quite generally, as we will see in \cref{thm:injection-hbc-cbc} below. On the other hand, Hovey conjectured that every cohomological class is homological and later Hovey--Palmieri made the stronger conjecture that \emph{every} localizing ideal is a homological Bousfield class; see \cite[Conjecture~1.2]{hovey-cbc} and \cite[Conjecture~9.1]{HoveyPalmieri99}. These conjectures remain open for $\Sp$ but we will provide many examples showing that they are false for general tensor-triangulated categories~$\cat T$.
\end{Rem}

\begin{Exa}\label{exa:bad}
	The first examples of cohomological Bousfield classes that are not homological are due to Stevenson~\cite{Stevenson14}. Let $R$ be an absolutely flat ring which is not semi-artinian. Then
	\[
		\Loco{\kappa(\mathfrak p) \mid \mathfrak p \in \Spec R} \subsetneq \Der(R)
	\]
	is a proper localizing ideal of $\Der(R)$ which is not a homological Bousfield class. Moreover, there exists a superdecomposable injective $R$-module $E$ with the property that $\cbc{E}$ is not a homological Bousfield class. See \cite[Section~4]{Stevenson14} for details. In this example, the Bousfield lattice is isomorphic to the lattice of subsets of $\Spec(R)$; see \cite[Example~4.21]{BarthelHeardSandersZou26}. To the best of our knowledge, this is the only known example (of a category with cohomological classes that are not homological) where the category is rigidly-compactly generated, i.e., compactly generated with the dualizable objects coinciding with the compact objects. 
\end{Exa}

\begin{Exa}\label{exa:hovey-wolcott}
	A second family of examples is due to Hovey and Wolcott \cite[Section~6.1]{Wolcott15}. The Bousfield lattice of the category $L_{\HFp}\Sp$ has exactly two elements: $\bc{0}$ and $\bc{\HFp}$. On the other hand, the cohomological Bousfield class of the \mbox{mod $p$} Moore spectrum $\cbc{\bbS/p}$ is not a homological Bousfield class. Another example is given by the cohomological Bousfield class of $L_{\HFp}(BP)$. In \cref{sec:oddball_cohomological_bousfield_classes} we will give an infinite family of cohomological Bousfield classes in $L_{\HFp}\Sp$ that are not homological.
\end{Exa}

\begin{Rem}\label{rem:local-bc}
	More generally, we will be interested in understanding the homological and cohomological Bousfield classes of an arbitrary Bousfield localization $\LAT$. Recall that if $x,y \in \LAT$ then their tensor product in $\LAT$ is given by $\LA(x\otimes y)$. We then see that if $\bc{y}$ denotes the homological Bousfield class of $y$ internal to $\LAT$, then
	\begin{equation}\label{eq:local-bc}
		x \in \bc{y} \Longleftrightarrow \LA(x\otimes y) = 0 \Longleftrightarrow A \otimes x \otimes y=0.
	\end{equation}
	From this, one readily checks that there is an equality of Bousfield classes $\bc{x}=\bc{y}$ in $\LAT$ if and only if there is an equality of Bousfield classes $\bc{A\otimes x}=\bc{A\otimes y}$ in~$\cat T$. Moreover, it is helpful to keep in mind that $A \otimes t \simeq A \otimes \LA t$ for any $t \in \cat T$. From this perspective, a classification of the homological Bousfield classes of $\LAT$ amounts to a classification of the homological Bousfield classes of $\cat T$ of the form $\bc{A \otimes t}$ for some~$t \in \cat T$.
\end{Rem}

\begin{Rem}\label{rem:local-coho}
	The internal hom in $\LAT$ is computed in $\cat T$. In fact, if $y \in\LAT$ is local then $\ihom{t,y} \in \LAT$ is also local for any $t \in \cat T$. Moreover, $\ihom{t,y}=0$ if and only if $\ihom{\LA t,y}=0$. This implies that we have an equality ${\cbc{x}=\cbc{y}}$ of cohomological Bousfield classes in $\LAT$ if and only if we have an equality $\cbc{x}=\cbc{y}$ of cohomological Bousfield classes in $\cat T$. A classification of the cohomological classes of $\LAT$ thus amounts to a classification of the cohomological classes of $\cat T$ generated by $A$-local objects. Observe that this differs from the situation for homological Bousfield classes discussed in \cref{rem:local-bc}.
\end{Rem}

\begin{Rem}
	We will mostly study localizations $\LAT$ under the hypothesis that $\cat T$ is rigidly-compactly generated. Although such localizations $\LAT$ are well generated by \cref{exa:localization-is-well-generated}, they are only rigidly-compactly generated in very special cases; see \cref{sec:smashing}. Consequently, many standard techniques of tt-geometry do not directly apply. For example $\cat T \to \LAT$ is not an example of a geometric functor in the sense of \cite{barthel2023cosupport,BalmerDellAmbrogioSanders16}. In any case, we will next show that for these localized categories, every homological Bousfield class is cohomological;~cf.~\cref{rem:bad-bousfield-classes}. This requires the following variant of Brown--Comenetz duality.
\end{Rem}

\begin{Def}\label{def:compact-BC}
	Let $\cat T$ be a rigidly-compactly generated tensor-triangulated category. For arbitrary $A\in\cat T$ the functor
	\begin{equation}\label{eq:homological}
		\Hom_{\bbZ}(\cat T(\unit,A\otimes-),\bbQ/\bbZ)\colon \cat T\op\longrightarrow \Ab
	\end{equation}
	is homological and sends coproducts in $\cat T$ to products in $\Ab$. It is thus represented by an object $I_A\in\cat T$ by Brown representability. Consequently, there is an isomorphism
	\begin{equation}\label{eq:BC-homs}
		\cat T(t,I_A)\simeq \Hom_{\bbZ}(\cat T(\unit,A\otimes t),\bbQ/\bbZ)
	\end{equation}
	which is natural in $t\in\cat T$. It follows formally from Yoneda that $A \mapsto I_A$ provides a functor $\cat T\op \to \cat T$.
\end{Def}

\begin{Rem}
	Let $I_{\unit}$ denote the Brown--Comenetz dual of the unit in $\cat T$, so
	\[
		\cat T(t,I_{\unit})\simeq \Hom_{\bbZ}(\cat T(\unit,t),\bbQ/\bbZ).
	\]
	Then, for any $A\in\cat T$ and $t\in\cat T$, we have
	\[
		\cat T(t,\ihom{A,I_{\unit}})\simeq \cat T(A\otimes t,I_{\unit}) \simeq \Hom_{\bbZ}(\cat T(\unit,A\otimes t),\bbQ/\bbZ).
	\]
	It follows that there is an isomorphism
	\[
		\ihom{A,I_{\unit}} \simeq I_A
	\]
	which is natural in $A$.
\end{Rem}

\begin{Def}
	For any $t \in \cat T$, let $I_A(t) \coloneqq \ihom{t,I_A} \simeq \ihom{t\otimes A,I_{\unit}}$.
\end{Def}

\begin{Rem}\label{rem:bc-vanishing}
	The following facts are routine:
	\begin{enumerate}
		\item $I_A(t)$ is $A$-local.
		\item $I_A(t) \simeq I_A(L_A(t))$.
		\item $I_A(t) = 0$ if and only if $t \otimes A = 0$.
	\end{enumerate}
	The last part uses that $\ihom{-,I_\unit}$ is conservative; see \cite[Example~12.7]{barthel2023cosupport}.
\end{Rem}

\begin{Prop}\label{thm:injection-hbc-cbc}
	Let $\cat T$ be a rigidly-compactly generated tensor-triangulated category and let $A \in \cat T$. For each $x\in \LA\cat T$ we have
	\begin{equation}\label{eq:bc-to-cbc}
		\bc{x}= \cbc{I_A(x)} \quad\text{in }\LA\cat T.
	\end{equation}
	Consequently, there is an injective map
	\[
		\left\{\begin{array}{c}
		\text{homological Bousfield classes}\\
		\text{of }\LA\cat T
		\end{array}\right\}
		\hookrightarrow
		\left\{\begin{array}{c}
		\text{cohomological Bousfield classes}\\
		\text{of }\LA\cat T
		\end{array}\right\}
	\]
	given by $\bc{x}\longmapsto \cbc{I_A(x)}.$
\end{Prop}

\begin{proof}
	Let $z\in \LA\cat T$. Since the internal hom in $\LA\cat T$ can be computed in $\cat T$, we have $z \in \cbc{I_A(x)}$ if and only if $I_A(z\otimes x)=\ihom{z,I_A(x)} = 0$ in $\cat T$. By \cref{rem:bc-vanishing}, this holds if and only if $A \otimes z \otimes x =0$ in $\cat T$, which by \eqref{eq:local-bc} is equivalent to $z \in \bc{x}$ in $\LA\cat T$. This establishes \eqref{eq:bc-to-cbc}. It immediately follows that $\bc{x} \mapsto \cbc{I_A(x)}$ is a well-defined injective map.
\end{proof}

We finish this section by discussing a couple of general examples.

\begin{Exa}\label{exa:localization-completion}
	Let $Y \subseteq \Spc(\cat T^c)$ be a Thomason subset and let ${e_Y \to \unit\to f_Y \to \Sigma e_Y}$ be the associated idempotent triangle in $\cat T$. Recall that $e_Y$ is an idempotent coring while~$f_Y$ is an idempotent ring. Taking $A=f_Y$ provides the finite localization 
	\newlength{\middleWidth}
	\settowidth{\middleWidth}{$\cat T|_{Y^c}$}
	\begin{align*}
		\cat T \;&\to\; \cat T|_{Y^c}= L_{f_Y}\cat T.\\
	\intertext{On the other hand, taking $A=e_Y$ provides the ``Bousfield completion''}
		\cat T \;&\to\; \makebox[\middleWidth][l]{$\cat T_Y^{\wedge}$} =L_{e_Y}\cat T.
	\end{align*}
	See \cite[Definition~2.5]{BalmerSanders_perfect} and the discussion therein. In other words:
	\[
		\bc{f_Y} = \Ker(\cat T \to \cat T|_{Y^c}) \quad\text{ and }\quad \bc{e_Y} = \Ker(\cat T \to \cat T_Y^{\wedge}).
	\]
\end{Exa}

\begin{Not}
	We write $\cat T_{\cat P} \coloneqq \cat T|_{\gen(\cat P)}$ for the local category at $\cat P \in \Spc(\cat T^c)$. Also, if the point $\cat P$ is visible then we will simply write $\cat T_{\cat P}^\wedge \coloneqq \smash{\cat T_{\overbar{\{\cat P\}}}^\wedge}$ for the completion along the Thomason closed set $\smash{\overbar{\{\cat P\}}}$.
\end{Not}

\begin{Exa}\label{exa:gP}
	Let $\cat P \in \Spc(\cat T^c)$ be a weakly visible point and let $g_{\cat P} \in \cat T$ be its associated Balmer--Favi idempotent. We can readily check that the localization $\cat T \to L_{g_{\cat P}}\cat T$ coincides with the composite $\cat T \to \cat T_{\cat P} \to (\cat T_{\cat P})_{\cat P}^\wedge$, i.e., localization at $\cat P$ followed by completion at the unique closed point.\footnote{The weakly visible point $\cat P$ of $\cat T$ becomes visible as a point of the local category $\cat T_{\cat P}$; see \cite[Remark~2.9]{bhs1}.} In other words,
	\[
		\bc{g_{\cat P}} = \Ker(\cat T \to \cat T_{\cat P} \to (\cat T_{\cat P})_{\cat P}^\wedge).
	\]
	We thus think of $L_{g_{\cat P}}\cat T$ as a ``completed stalk'' at $\cat P$.
\end{Exa}

\begin{Rem}\label{rem:homo-coho-different}
	In general, the homological and cohomological Bousfield classes $\bc{A}$ and $\cbc{A}$ are very different. For instance, if $A \neq 0$ then $A \notin \cbc{A}$ but if $A \otimes A = 0$ then $A \in \bc{A}$. In fact, we have an inclusion
	\begin{equation}\label{eq:bc-incl}
		\bc{A} \subseteq \cbc{A}
	\end{equation}
    if and only if $A$ is $A$-local. In particular, this holds when $A$ is a weak ring; see \cite[Lemma~2.6(b)]{BarthelHeardSandersZou26}. However, this inclusion can be strict. For example, it was proved by Lin \cite[Theorem~3.2]{Lin76} that the stable cohomotopy groups of $\HFp$ are trivial. In other words, $\ihom{\HFp,\unit}=0$ in $\Sp$. Hence 
	\begin{equation}\label{eq:unit-not-in-unit}
		\cbc{\unit} \not\subseteq \bc{\unit}
	\end{equation}
	in $\Sp$. Also, if $\Spc(\cat T^c)$ is connected then
	\begin{equation}\label{eq:eY-not-in-eY}
		\bc{e_Y} \not\subseteq \cbc{e_Y}
	\end{equation}
	for any nonempty proper Thomason subset $Y \subsetneq \Spc(\cat T^c)$. Indeed, $f_Y \otimes e_Y = 0$ but $\ihom{f_Y,e_Y} = 0$ would imply that the spectrum is disconnected; see \cite[Proposition~2.29]{PatchkoriaSandersWimmer22}, for example.
\end{Rem}

\begin{Rem}\label{rem:coring}
	Nevertheless, for any object $A$, we have the curious fact that 
	\begin{alignat}{2}\label{eq:Cbccbc}
		C \in \bc{A} &\Longleftrightarrow C &&\in \cbc{A}
	\end{alignat}
	for any weak coring $C$. See \cite[Lemmas~2.6(a) and~2.7]{BarthelHeardSandersZou26}. Moreover, in general neither of the implications in~\eqref{eq:Cbccbc} are true with the weak coring $C$ replaced by a weak ring~$R$. This is demonstrated by the $R=f_Y$ example of \eqref{eq:eY-not-in-eY} and the $R=\HFp$ example of \eqref{eq:unit-not-in-unit}. In any case, a localizing ideal contains a dualizable object $c$ if and only if it contains the coring $c\otimes c^\vee$. Hence $\bc{A}$ and $\cbc{A}$ contain precisely the same dualizable objects.
\end{Rem}

\begin{Rem}
	Classically, homology and cohomology theories in algebraic topology arise in pairs and correspond to objects $E \in \Sp$ in the category of spectra. Explicitly, the associated cohomology theory is given by $E^i(X) = \Hom_{\Sp}(X,\Sigma^i E)$ and the associated homology theory is given by $E_i(X) = \Hom_{\Sp}(\Sigma^i \unit, E \otimes X)$. Since $\Sp$ is generated by $\unit$, we see that the cohomological Bousfield class $\cbc{E}$ coincides with the kernel of the cohomology theory $E^*\colon \Sp{\op} \to \Ab^{\bbZ}$ and the homological Bousfield class $\bc{E}$ coincides with the kernel of the homology theory $E_*\colon \Sp\to \Ab^{\bbZ}$. These annihilate the same finite spectra (\aka dualizable spectra), as is well-known. More generally, they annihilate the same weak corings (\cref{rem:coring}).
\end{Rem}

\begin{Exa}
	The suspension spectrum $\Sigma_+^\infty X$ of any unreduced space $X$ is a coring in the homotopy category of spectra. Hence, for any choice of generalized (co)homology theory, the spaces with vanishing unreduced homology coincide with the spaces with vanishing unreduced cohomology.
\end{Exa}

\begin{Rem}\label{rem:auto-local}
	As mentioned above, if $A$ is a weak ring then $A$ is itself $A$-local. This need not be true if $A$ is a weak coring by~\eqref{eq:eY-not-in-eY}. It is a special property of localization with respect to a \emph{multiplicative} homology theory.
\end{Rem}

\section{Classifying homological Bousfield classes}

Throughout this section we fix a rigidly-compactly generated tensor-triangulated category~$\cat T$. Recall from \cite[Construction~2.11]{Balmer20_bigsupport} that there is a pure-injective object~$\EB$ in $\cat T$ associated to each homological prime $\cat B \in \Spc^h(\cat T^c)$ and that~$\EB$ has the structure of a weak ring. The following support theories play an important role in \cite{BarthelHeardSandersZou26}. 

\begin{Def}\label{rem:support-and-cosupports}
    For each object $t \in \cat T$, we define
    \begin{itemize}
		\item	The \emph{naive homological support}
				\[
					\Supphnaive(t) \coloneqq \SETT{\cat B \in \Spc^h(\cat T^c)}{\EB \otimes t\neq 0}.
				\]
		\item	The \emph{(genuine) homological support} 
				\[
					\Supph(t) \coloneqq \SETT{\cat B \in \Spc^h(\cat T^c)}{\ihom{t,\EB}\neq 0}.
				\]
	\end{itemize}
\end{Def}

\begin{Rem} 
	Observe that
	\[
		\bc{\EB} = \SETT{t \in \cat T}{\cat B \notin \Supphnaive(t)}
		\;\text{ and }\;
		\cbc{\EB} = \SETT{t\in \cat T}{\cat B \notin \smash{\Supph(t)}}
	\]
	for any $\cat B \in \Spc^h(\cat T^c)$.
\end{Rem}

\begin{Rem}
	We will repeatedly use the fact that (genuine) homological support satisfies the \emph{tensor-product property}: for all $t_1,t_2 \in \cat T$ we have
    \[
        \Supph(t_1 \otimes t_2) = \Supph(t_1) \cap \Supph(t_2).
    \]
	This is established in \cite[Theorem~4.5]{Balmer20_bigsupport}.
\end{Rem}

\begin{Def}\label{def:hdetection}
	We say that the \emph{h-detection property} holds for $\cat T$ (short for \emph{homological detection property}) if $\Supph(t) = \emptyset$ implies $t=0$.
\end{Def}

\begin{Rem}\label{rem:hdetection-nilpotent}
	Since $\Supph$ satisfies the tensor-product property, the h-detection property cannot hold for any $\cat T$ which has nonzero tensor-nilpotent objects. For example, it fails for the category of spectra $\Sp$ since the Brown--Comenetz dual of the sphere $I$ is tensor-nilpotent: $I\otimes I = 0$. This is a key obstruction to understanding the Bousfield lattice of $\Sp$.
\end{Rem}

\begin{Rem}\label{rem:detection-for-weakring}
	The h-detection property always holds for weak rings: if $A \in \cat T$ is a weak ring then $\Supph(A) = \emptyset$ implies $A=0$. See \cite[Theorem~1.8]{Balmer20_bigsupport}. In contrast, the h-detection property does not always hold for weak corings; see \cref{cor:detection-fails-corings}.
\end{Rem}

\begin{Rem}\label{rem:inclusion}
	There is always an inclusion 
	\begin{equation}\label{eq:supph-inclusion}
		\Supph(t) \subseteq \Supphnaive(t)
	\end{equation}
	and this is an equality if the object $t$ is compact or a weak ring or a weak coring; see \cite[Section~4]{Balmer20_bigsupport} and \cite[Proposition~2.8]{BarthelHeardSandersZou26}. It is an open question whether~\eqref{eq:supph-inclusion} is always an equality for all objects~$t\in\cat T$. This is known to be true in many cases:
\end{Rem}

\begin{Prop}\label{prop:sufficient-supph-coincide}
	Suppose any of the following conditions hold:
	\begin{enumerate}
		\item $\cat T$ has the h-detection property.
		\item $\cat T$ has enough tt-fields in the sense of \cite[Definition~2.13]{BarthelHeardSandersZou26}.
		\item Each localizing ideal $\Loco{\EB}$ is minimal.
		\item Each $\EB$ is a ``field-object'' in the sense that $\EB\otimes t$ is a direct sum of suspensions of $\EB$ for any $t \in \cat T$.
	\end{enumerate}
	Then $\Supph(t) = \Supphnaive(t)$ for each $t \in \cat T$.
\end{Prop}

\begin{proof}
	The sufficiency of $(a)$ is \cite[Lemma~2.11]{BarthelHeardSandersZou26} while $(b)$ is \cite[Proposition~2.16]{BarthelHeardSandersZou26} and $(c)$ is \cite[Lemma~2.18]{BarthelHeardSandersZou26}. For condition $(d)$, suppose that $\cat B \not\in\Supph(t)$. Then $\cat B \not\in \Supph(\EB \otimes t) \subseteq \Supph(t)$. That is, $\ihom{\EB\otimes t,\EB}=0$. If $\EB\otimes t \neq 0$ then some shift of $\EB$ is a summand of $\EB\otimes t$. This would imply that $\ihom{\EB,\EB}=0$ which is a contradiction. We conclude that $\EB\otimes t=0$, i.e., $\cat B \not\in\Supphnaive(t)$.
\end{proof}

\begin{Exa}\label{exa:D(R)-Supp-coincide}
	The derived category $\Der(R)$ of a commutative ring $R$ admits enough \mbox{tt-fields,} which are given by the usual residue fields $\Der(R) \to \Der(\kappa(p))$; see \cite[Example~2.14]{BarthelHeardSandersZou26}. Therefore, we have
    \[
		\Supph(t) = \Supphnaive(t) = \SETT{\frakp \in \Spec(R)}{t \otimes \kappa(\frakp) \neq 0}
    \]
    for any $t \in \Der(R)$. Here we have implicitly used that $\Spc^h(\Der(R)^c) \cong \Spc(\Der(R)^c) \cong \Spec(R)$; see \cite[Corollary~5.11]{Balmer20_nilpotence}.
\end{Exa}

\begin{Exa}\label{exa:spectrum-of-Sp}
	Let $\Sp$ denote the stable homotopy category of spectra. The Balmer spectrum $\Spc(\Sp^c)$ consists of points $\cat C_{p,n}$ which are indexed by primes $p$ and $n\in\{0,1,2,\ldots,\infty\}$ with the collision $\cat C_{p,0}=\cat C_{q,0}\eqqcolon \cat C_0$ for all primes $p$ and~$q$. The topology is described by $\overline{\{\cat C_{p,n}\}}=\{\cat C_{p,k}\mid n\le k\le \infty\}$ for $1 \le n\le \infty$ and $\overline{\{\cat C_0\}}=\Spc(\Sp^c)$. In particular, $\Spc(\Sp^c)$ is irreducible with generic point~$\cat C_0$ and with closed points $\cat C_{p,\infty}$ indexed by the primes $p$. The comparison map $\Spc^h(\Sp^c) \to \Spc(\Sp^c)$ is a bijection by \cite[Corollary~5.10]{Balmer20_nilpotence}. Moreover, by \cite[Corollary~3.6]{BalmerCameron20pp}, the pure-injective corresponding to $\cat C_{p,n}$ is precisely the Morava $K$-theory $K(p,n)$ with the understanding that $K(p,\infty)=\HFp$ and $K(p,0)=\HQ$. Since these are field-objects, we have $\Supph=\Supphnaive$.
\end{Exa}

\begin{Rem}
	We will also use the same notation $\Sp$ to denote the stable homotopy category of $p$-local spectra for some fixed prime $p$. Its Balmer spectrum is the
	\[
		\cat C_{\infty} - \dots - \cat C_{n+1} - \cat C_{n} - \cdots - \cat C_1 - \cat C_0
	\]
	part of \cref{exa:spectrum-of-Sp} and will usually be identified with $\{0,1,2,\ldots,\infty\}=\mathbb{N}\cup\{\infty\}$. We will drop the $p$ from all notation when working in the $p$-local category. For example, $K(n)$ will denote the $p$-local Morava $K$-theory $K(p,n)$. Despite using~$\Sp$ to denote both categories, it should be clear (either implicitly or explicitly) which category we are considering at any given time.
\end{Rem}

We now investigate the behavior of the pure-injectives under Bousfield localization.

\begin{Lem}\label{lem:completion-of-eb}
	Let $A \in \cat T$ be any object. The following statements hold:
	\begin{enumerate}
		\item $\EB$ is $A$-acyclic if and only if $\cat B \not \in \Supphnaive(A)$.
		\item If $\cat B \in \Supph(A)$ then $\EB$ is $A$-local.
		\item If $\EB$ is $A$-local then $\cat B \in \Supphnaive(A)$.
	\end{enumerate}
\end{Lem}

\begin{proof}
	Part $(a)$ is immediate from the definition.

	$(b)$:
	If $\cat B \in \Supph(A)$ then for any $A$-acyclic object $t\in \cat T$, we have 
	\[
		\Supph(t) \cap \Supph(A) = \Supph(t \otimes A)=\emptyset
	\]
	since $t\otimes A=0$. Hence $\cat B \not\in \Supph(t)$ so that $\ihom{t,\EB} = 0$. Thus $\EB$ is $A$-local.

	$(c)$: If $\EB$ is $A$-local then $\ihom{\EB,\EB}\neq 0$ means that $\EB$ is not $A$-acyclic. Hence $\cat B\in\Supphnaive(A)$ by part $(a)$.
\end{proof}

\pagebreak[2]
\begin{Cor}\label{cor:local-acyclic}
	Suppose $\Supph(A)=\Supphnaive(A)$. Then
	\begin{enumerate}
		\item $\EB$ is $A$-acyclic if and only if $\cat B \not\in\Supph(A)$.
		\item $\EB$ is $A$-local if and only if $\cat B \in\Supph(A)$.
	\end{enumerate}
\end{Cor}

\begin{Rem}\label{rem:completeness}
	It will be helpful to bear in mind that $\Supph(A)=\Supphnaive(A)$ whenever~$A$ is a weak ring, as mentioned in \cref{rem:inclusion}. This is the case for many examples of interest in applications.
\end{Rem}

\begin{Def}\label{def:supphA}
	For any $t \in \cat T$, define
	\begin{align}
		\Supph_A(t)&\coloneqq \Supph(A\otimes t)=\Supph(A)\cap \Supph(t), \text{ and }\\
		\Supphnaive_A(t)&\coloneqq \Supphnaive(A\otimes t)\subseteq \Supphnaive(A).
	\end{align}
\end{Def}

\begin{Rem}\label{rem:complete-vs-uncomplete}
	Since  $A\otimes t\to A\otimes \LA t$ is an isomorphism, we have
	\[
		\Supph_A(t)=\Supph_A(\LA t)
		\quad\text{and}\quad
		\Supphnaive_A(t)=\Supphnaive_A(\LA t).
	\]
\end{Rem}

\begin{Lem}\label{lem:Supph-EB}
	The following statements hold:
	\begin{enumerate}
		\item If $\cat B \in \Supph(A)$ then $\Supph_A(\EB) = \{\cat B\}$.
		\item If $\cat B \in \Supphnaive(A)$ then $\Supphnaive_A(\EB) = \{\cat B\}$.
	\end{enumerate}
\end{Lem}
\begin{proof}
	If $\cat B \in \Supph(A)$ then
	\[ \{\cat B\} = \Supph(\EB) \cap \Supph(A) = \Supph(\EB \otimes A) = \Supph_A(\EB)\]
	using the tensor-product formula.
	On the other hand, if $\cat B \in \Supphnaive(A)$ then $\EB \otimes A \neq 0$
	by definition. 
	This implies $\EB \otimes \EB \otimes A \neq 0$ since $\EB$ is a direct summand of $\EB \otimes \EB$ because it is a weak ring.
	Hence $\cat B \in \Supphnaive_A(\EB)$.
	Thus $\{\cat B\} \subseteq \Supphnaive_A(\EB) \subseteq \Supphnaive(\EB) = \{\cat B\}$.
\end{proof}

\begin{Def}\label{def:hbc-map}
	Let $A\in\cat T$ and define a map
	\begin{equation}\label{eq:hbc-map}
		\big\{\text{homological Bousfield classes of } \LA\cat T\big\}
		\longrightarrow
		\big\{\text{subsets of }\Supphnaive(A)\big\}
	\end{equation}
	by sending $\bc{x}$ to $\Supphnaive_A(x)$.
\end{Def}

\begin{Lem}\label{lem:well-defined}
	The map \eqref{eq:hbc-map} is a well-defined order-preserving and join-preserving map.
\end{Lem}

\begin{proof}
	Let $x$ be an object in $\LAT$ and consider its Bousfield class $\bc{x}$ in the local category $\LAT$. Observe that for any $\cat B \in \Spc^h(\cat T^c)$ we have
	\[
		\LA \EB \in \bc{x} \Longleftrightarrow \LA(\EB \otimes x) = 0 \Longleftrightarrow \EB \otimes x \otimes A=0 \Longleftrightarrow \cat B \not\in \Supphnaive_A(x).
	\]
	Thus, if we have an inclusion $\bc{x} \subseteq \bc{y}$ of Bousfield classes in $\LAT$ then the condition $\LA \EB \in \bc{x}$ implies $\LA \EB \in \bc{y}$. Hence $\Supphnaive_A(y) \subseteq \Supphnaive_A(x)$. This shows that the map $\bc{x} \mapsto \Supphnaive_A(x)$ is well-defined and inclusion-reversing. This is an order-preserving map since the partial order of Bousfield classes is defined by reverse-inclusion. It readily follows from~\eqref{eq:join} that the map preserves joins.
\end{proof}

\begin{Lem}\label{lem:hbc-splitting}
	If $\Supph(A)=\Supphnaive(A)$ then the map \eqref{eq:hbc-map} admits an order-preserving and join-preserving section $\sigma$ defined by
	\[
		\sigma(S) \coloneqq \bcnothuge{\LA\bigl(\coprod_{\cat B\in S} \EB\bigr)}.
	\]
	In particular, \eqref{eq:hbc-map} is surjective.
\end{Lem}

\begin{proof}
	For each $\cat B \in \Supph(A)$, the object $\EB$ is $A$-local by \cref{cor:local-acyclic},  and
	$\Supph_A(\EB)=\Supphnaive_A(\EB)=\{\cat B\}$ by \cref{lem:Supph-EB}.
	Given any subset $S\subseteq \Supph(A)$, set $x\coloneqq \LA(\coprod_{\cat B\in S} \EB)\in \LA\cat T$. This is simply the coproduct in $\LAT$ of the local objects $\EB \in \LAT$. Using \cref{rem:complete-vs-uncomplete}, we have
	\[
		\Supphnaive_A(x) = \Supphnaive_A(\coprod\nolimits_{\cat B\in S} \EB) =\bigcup\nolimits_{\cat B\in S}\Supphnaive_A(\EB)=  S.
	\]
	Thus $\sigma$ is a section of \eqref{eq:hbc-map}. It is order-preserving since $S\subseteq S'$ implies
    \[
		\bcnothuge{\LA\bigl(\coprod_{\cat B\in S}\EB\bigr)}
		\le
		\bcnothuge{\LA\bigl(\coprod_{\cat B\in S'}\EB\bigr)}.
	\]
	Similarly, it is join-preserving because

	\begin{align*}
	  \bigvee_i \sigma(S_i)
	  &=
	  \bcnothuge{\LA\bigl(\coprod_i \LA\bigl(\coprod_{\cat B\in S_i}\EB\bigr)\bigr)} \\
	  &=
	  \bcnothuge{\LA\bigl(\coprod_i\coprod_{\cat B\in S_i}\EB\bigr)} \\
	  &=
	  \bcnothuge{\LA\bigl(\coprod_{\cat B\in \bigcup_i S_i}\EB\bigr)}
	  =
	  \sigma\bigl(\bigcup_i S_i\bigr).\qedhere
	\end{align*}
\end{proof}

\begin{Def}\label{def:relative-hdetection}
	We say that the \emph{relative h-detection property} holds for $\LA\cat T$ if $\Supph_A(t) = \emptyset$ implies $\LA t=0$ for all $t \in \cat T$. Note that for $A=\unit$ this reduces to \cref{def:hdetection}.
\end{Def}

\begin{Rem}
	Recall that the h-detection property does not hold for the category of spectra (\cref{rem:hdetection-nilpotent}). However, the relative h-detection property holds for a number of interesting Bousfield localizations of spectra, as we will explore in \cref{sec:first-examples} below.
\end{Rem}

\begin{Lem}\label{lem:local-h-detection}
	If $\LA\cat T$ has the relative h-detection property then $\Supph_A = \Supphnaive_A$. 
\end{Lem}

\begin{proof}
	For any $\cat B \in \Spc^h(\cat T^c)$, we have
	\[
	\begin{aligned}
		\cat B\in\Supphnaive(A\otimes t)
		&\iff \EB\otimes A\otimes t\neq 0 \\
		&\iff \LA(\EB\otimes t)\neq 0 \\
		&\iff \Supph_A\big(\EB\otimes t\big)\neq \emptyset \\
		&\iff \Supph(A\otimes \EB\otimes t)\neq \emptyset \\
		&\iff \{\cat B\}\cap \Supph(A\otimes t)\neq \emptyset
	\end{aligned}
	\]
	using relative h-detection for the third equivalence.
\end{proof}

The following generalizes \cite[Proposition~4.18]{BarthelHeardSandersZou26} to the $A$-local case.

\begin{Thm}\label{thm:bousfield-classes}
	Let $\cat T$ be a rigidly-compactly generated tt-category and let $A \in \cat T$. The following are equivalent:
	\begin{enumerate}
		\item relative h-detection holds for $\LA\cat T$.
		\item $\Supph_A=\Supphnaive_A$ and the Bousfield lattice of $\LA\cat T$ is isomorphic to the lattice of subsets of $\Supph(A)$ via $\bc{x}\mapsto \Supph_A(x)$.
	\end{enumerate}
\end{Thm}

\begin{proof}
	$(a) \Rightarrow (b)$: \Cref{lem:local-h-detection} gives $\Supph_A=\Supphnaive_A$. In particular, $\Supph(A)=\Supph_A(\unit)=\Supphnaive_A(\unit) = \Supphnaive(A)$ and so the map admits a section by \cref{lem:hbc-splitting}. Using relative h-detection and the tensor-product formula, we have
	\begin{equation}\label{eq:bc-in-supph}
	\begin{aligned}
		\bc{x}
		&= \SETT{y\in \LA\cat T}{\LA(x\otimes y)=0} \\
		&= \SETT{y\in \LA\cat T}{\Supph_A(x\otimes y)=\emptyset} \\
		&= \SETT{y\in \LA\cat T}{\Supph_A(x)\cap \Supph(y)=\emptyset}.
	\end{aligned}
	\end{equation}
	An immediate consequence is that the map is injective. We have already established in \cref{lem:well-defined} that the map is order-preserving. It follows from \eqref{eq:bc-in-supph} that it also reflects the order: $\bc{x_1} \le \bc{x_2}$ if $\Supph_A(x_1) \subseteq \Supph_A(x_2)$. The bijection $\bc{x} \mapsto \Supph_A(x)$ is thus an isomorphism of lattices.

	$(b) \Rightarrow (a)$: If $\Supph_A(t)=\emptyset$ then $\Supph_A(\LA t)=\emptyset$ by \cref{rem:complete-vs-uncomplete}. Since $\Supph_A(0)=\emptyset$, we conclude that $\bc{\LA t}=\bc{0}$. Therefore $\LA t=0$.
\end{proof}

\begin{Rem}
	Under the lattice isomorphism $\BL(\LAT) \simeq \mathcal P(\Supph(A))$ provided by \cref{thm:bousfield-classes}, the join and meet of Bousfield classes correspond to union and intersection, respectively. The tensor-product property for $\Supph$ then implies that the meet of Bousfield classes is given by the tensor-product: ${\bc{x}\wedge \bc{y} = \bc{L_A(x \otimes y)}}$. Moreover, note that in this situation, the Bousfield lattice is a Boolean lattice. It follows that the sublattices of the Bousfield lattice considered in \cite{HoveyPalmieri99} all coincide with the full Bousfield lattice. In particular, every Bousfield class is complemented, with the complement operator given by $\abcx \coloneqq \bigvee \{ \bc{y} \mid \bc{y}\otimes \bc{x} = \bc{0} \}$.
\end{Rem}

\begin{Rem}\label{rem:hdetect-not-restricted}
	It follows from~\cref{lem:completion-of-eb}(b) that
	\begin{equation}\label{eq:supph-A-LA}
		\Supph(A) \subseteq \Supph(\LA\unit)
	\end{equation}
	but this is, in general, a strict inclusion.\footnote{For example, take $\cat T = \Der(\bbZ)$ and $A = \Fp$. Then $\Supph(A) = \{(p)\}$ is a point, while $\Supph(L_A\unit) = \{ (0),(p) \}$ is two points.} Moreover, for any $A$-local object $x \in \LAT$, we have
	\begin{equation}\label{eq:supph-local}
		\Supphnaive(x) \subseteq \Supphnaive(\LA\unit) = \Supph(\LA\unit).
	\end{equation}
	This follows from the fact that $\LA\unit$ is a ring and each $x$ is an $\LA\unit$-module. If~\eqref{eq:supph-A-LA} is an equality then $\Supph_A(x) = \Supph(x)$. Thus, in this case, $A$-relative h-detection is equivalent to the h-detection property for $\cat T$ restricted to $A$-local objects. In general, however, $A$-relative h-detection is a weaker property than h-detection restricted to $A$-local objects.
\end{Rem}

\section{Homological Bousfield classes generated by residue fields}\label{sec:first-examples}

We now give some examples where we can classify the Bousfield lattice.

\begin{Not}
	For any subset $S \subseteq \Spc^h(\cat T^c)$, we set
	\[
		K(S) \coloneqq \coprod_{\cat B\in S} \EB.
	\]
	For any particular $\cat B \in \Spc^h(\cat T^c)$, we will also write
	\[
		\KB\coloneqq K(\{\cat B\}) = \EB.
	\]
	This notation evokes the connection between the pure-injectives $\EB$ and residue fields. For example, in the derived category of a ring, $K(\frakp) \simeq \kappa(\frakp)$ is the usual residue field (\cref{exa:DR}). Moreover, for the $p$-local stable homotopy category, the~$K(\cat B)$ notation coincides with the usual notation for the Morava $K$-theories, which are the ``residue fields'' for this category (\cref{exa:spectrum-of-Sp}). In contrast, the standard $\EB$ notation relates to the construction of these pure-injectives via an injective hull; see \cite[Section~3]{BalmerKrauseStevenson19}.
\end{Not}

\begin{Thm}\label{thm:generalized-harmonic}
	Let $\cat T$ be a rigidly-compactly generated tt-category with the property that $\Supph(t)=\Supphnaive(t)$ for all $t\in\cat T$. For any subset $S\subseteq \Spc^h(\cat T^c)$, the relative h-detection property holds for $L_{\KS}\cat T$ and there is a bijection
	\[
		\big\{\text{homological Bousfield classes of }L_{\KS}\cat T\big\}\xrightarrow{ \sim }\big\{\text{subsets of }S\big\}.
	\]
\end{Thm}

\begin{proof}
	By the tensor-product formula and the definition of localized support,
	\[
		\Supph_{\KS}(t)=\Supph(\KS)\cap\Supph(t)=S\cap\Supph(t).
	\]
	An object $t$ vanishes in $L_{\KS}\cat T$ if and only if $\KS\otimes t =0$, that is, if and only if $\EB\otimes t= 0$ for all $\cat B\in S$. Since $\Supph(t) = \Supphnaive(t)$ we have
	\[
		\EB\otimes t= 0 \Longleftrightarrow \cat B\notin \Supph(t).
	\]
	Therefore $\KS\otimes t= 0$ if and only if $S\cap\Supph(t)=\emptyset$, i.e., if and only if $\Supph_{\KS}(t)=\emptyset$. This proves relative h-detection for $L_{\KS}\cat T$. The classification of homological Bousfield classes then follows from \cref{thm:bousfield-classes}.
\end{proof}

\begin{Cor}\label{cor:bc-chromatic-localizations}
	Fix a prime $p$. For any subset $S \subseteq \mathbb{N}\cup\{\infty\}$, let $K(S)\coloneqq \coprod_{i\in S}K(i)$ be the coproduct of the corresponding Morava $K$-theories with the convention that $K(0)=\HQ$ and $K(\infty)=\HFp$. Then the relative h-detection property holds for $\Sp_{K(S)}\coloneqq L_{K(S)}\Sp$ and there is a bijection
	\[
		\big\{\text{homological Bousfield classes of }\Sp_{K(S)}\big\}\xrightarrow{ \sim }\big\{\text{subsets of }S\big\}.
	\]
\end{Cor}

\begin{proof}
	This is just \cref{thm:generalized-harmonic} specialized to the category of $p$-local spectra, bearing in mind \cref{exa:spectrum-of-Sp}.
\end{proof}

\begin{Exa}
	The $K(n)$-local category $L_{K(n)}\Sp$ is well-known to have only two Bousfield classes. This is the singleton $S=\{n\}$ case of the corollary.
\end{Exa}

\begin{Exa}
	By definition, the category of harmonic spectra is $L_{K(S)}\Sp$ for $S=\mathbb{N}$. In this example, the theorem recovers the computation of the Bousfield lattice of harmonic spectra due to Beardsley~\cite{BeardsleyHarmonic2013}.
\end{Exa}

\begin{Exa}\label{exa:familiar-spectra}
	Many familiar $p$-local spectra have Bousfield class of the form $\bc{K(S)}$ for some $S\subseteq\mathbb{N}\cup\{\infty\}$. For example:
	\begin{align*}
		\bc{\HZ_{(p)}} &= \bc{K(0) \vee K(\infty)}\\
		\bc{BP\langle n \rangle} &= \bc{K([0,n]) \vee K(\infty)}\\
		\bc{KU_{(p)}} &= \bc{KO} =\bc{K(0) \vee K(1)} \\
		\bc{ku_{(p)}} &=\bc{K(0) \vee K(1) \vee K(\infty)}\\
		\bc{TMF_{(p)}} &= \bc{K(0) \vee K(1) \vee K(2)}\\
		\bc{tmf_{(p)}} &= \bc{K(0) \vee K(1) \vee K(2) \vee K(\infty)}.
	\end{align*}
	All of these can be found in \cite[Section~7]{Strickl2019Combinatorial}. Consequently, for each such $A$ the associated localized category $\LA\Sp \simeq L_{K(S)}\Sp$ satisfies relative h-detection and its homological Bousfield classes correspond to subsets of $S$ via the relative homological support. In fact, we have the following general result:
\end{Exa}

\begin{Thm}\label{thm:only-for-residues}
	Let $A \in \cat T$ and assume that $\Supph=\Supphnaive$. The following statements are equivalent:
	\begin{enumerate}
	\item	The relative h-detection property holds for $\LAT$.
	\item	For any $t \in \cat T$, we have an equality of Bousfield classes 
			\[
				\bc{A\otimes t} = \bc{K(S)}
			\]
			with $S=\Supph_A(t)$.
	\item	There is an equality $\bc{A} = \bc{K(S)}$ for some $S \subseteq \Spc^h(\cat T^c)$.
	\end{enumerate}
\end{Thm}

\begin{proof}
	First let $x \in \LAT$ be a local object and recall \cref{rem:local-bc}. It follows from the definitions that 
	\[
		\bc{A\otimes x} = \SETT{t \in \cat T}{\LA t \in \bc{x} \text{ in } \LAT}
	\]
	which shows that $\bc{x} \mapsto \bc{A\otimes x}$ is a well-defined injective map $\BL(\LAT) \hookrightarrow \BL(\cat T)$.

	$(a) \Rightarrow (b)$: Let $t \in \cat T$. By the classification of Bousfield classes in $\LAT$, we have
	\[
		\bc{\LA t} = \bc{\LA K(S)}
	\]
	with $S=\Supph_A(\LA t) = \Supph_A(t)$. It follows that we have the following equalities of Bousfield classes in $\cat T$:
	\[
		\bc{A\otimes t} = \bc{A\otimes \LA t} = \bc{A\otimes \LA K(S)} = \bc{A \otimes K(S)}.
	\]
	It remains to show that $\bc{A \otimes K(S)} = \bc{K(S)}$ for any $S \subseteq \Supph(A)$. First observe that for any $s \in \cat T$ we have
	\begin{equation}\label{eq:supphsA}
		\Supph(s\otimes A) \cap S = \Supph(s) \cap \Supph(A) \cap S = \Supph(s) \cap S
	\end{equation}
	since $S \subseteq \Supph(A)$. Moreover, note that
	\[
		\begin{aligned}
			s \in \bc{A\otimes K(S)} \iff& \Supphnaive(s\otimes A) \cap S = \emptyset, \text{ and}\\
			s\in \bc{K(S)} \iff& \Supphnaive(s) \cap S = \emptyset.
		\end{aligned}
	\]
	Thus $\bc{A\otimes K(S)}=\bc{K(S)}$ follows from \eqref{eq:supphsA} together with our hypothesis that $\Supph=\Supphnaive$.

	$(b)\Rightarrow (c)$: This is trivial: take $t=\unit$.

	$(c) \Rightarrow (a)$: First note that $S=\Supph(A)$ since for any $\cat B\in\Spc^h(\cat T^c)$, we have $\cat B\in S$ if and only if $K(S)\otimes\EB\neq 0$, if and only if $A\otimes\EB\neq 0$ by the equality $\bc{A}=\bc{K(S)}$, if and only if $\cat B\in\Supphnaive(A)=\Supph(A)$. Let $t \in \cat T$ and suppose $\Supph_A(t) = \emptyset$. Then $\Supph(t) \cap S = \emptyset$. It follows that $K(S) \otimes t = 0$ so that $t \in \bc{K(S)}$. Thus $t \in \bc{A}$ by hypothesis. That is, $A \otimes t=0$, so that $\LA t=0$. This establishes $(a)$.
\end{proof}

\begin{Rem}\label{rem:only-for-residues}
	The classification of Bousfield classes of $\LAT$ via homological support is thus essentially equivalent to $\bc{A}$ being of the special form $\bc{K(S)}$, i.e., of the Bousfield class~$\bc{A}$ being generated by a coproduct of $\EB$'s. Absolute h-detection is essentially equivalent to this being true for all $A$, i.e., for all Bousfield classes.\footnote{In fact, this is equivalent to the single Bousfield class $\bc{\unit}$ being generated by a coproduct of $\EB$'s, namely the coproduct of all of them.} Although $\Sp$ has Bousfield classes $\bc{A}$ that cannot be expressed in this way, many Bousfield classes of interest \emph{can be}, as we saw in~\cref{exa:familiar-spectra}. Thus, they satisfy \emph{relative} h-detection, and we have a classification of their Bousfield lattice.
\end{Rem}

\begin{Exa}
	In the category of $p$-local spectra, the Bousfield class of the Brown--Comenetz dual of the sphere $I$ is not generated by Morava $K$-theories in the above sense. Consequently, the Bousfield classes of $L_I\Sp$ are not classified by homological support. Indeed, there are precisely two homological Bousfield classes, but $\smash{\Supph(I)} = \emptyset$. It seems plausible that the tensor-triangular support considered in \cite{Zou25supp} may be a suitable support theory in this case. Indeed, by \cite[Example~6.12]{Zou25supp} we have $\Supp(I) = \{\cat C_{\infty}\}$.
\end{Exa}

\begin{Exa}\label{exa:DR}
	Let $R$ be a commutative ring and let $\cat T=\Der(R)$. By \cref{exa:D(R)-Supp-coincide}, we can apply \cref{thm:generalized-harmonic}: For any subset $S \subseteq \Spec(R)\cong\Spc^h(\Der(R)^c)$ there is a bijection
	\[
		\big\{\text{homological Bousfield classes of }L_{K(S)}\Der(R)\big\}\xrightarrow{\ \sim\ }\big\{\text{subsets of }S\big\}
	\]
	where $K(S) = \coprod_{\mathfrak p \in S}\kappa(\mathfrak p)$.
\end{Exa}

\begin{Rem}\label{rem:DR-h-detection}
	If $R$ is noetherian then $\Der(R)$ has the h-detection property. This follows from \cite[Lemma~2.12]{Neeman92a} by \cref{exa:D(R)-Supp-coincide}. Hence, we have
    \[
    \bc{K(\Spec(R))}=\bc{R}=\{0\}
    \]
    and therefore $L_{K(\Spec(R))}\Der(R) \cong \Der(R)$ itself. However, this can fail in the non-noetherian case, as explained in \cite[Remark~6.10]{BarthelHeardSandersZou26}.
\end{Rem}

\begin{Prop}\label{prop:K(S)-universal}
	Suppose $\Supph=\Supphnaive$ and let $S\coloneqq\Spc^h(\cat T^c)$. The localization $\cat T \to L_{K(S)}\cat T$ is the universal homological localization satisfying relative \mbox{h-detection}.
\end{Prop}

\begin{proof}
	If $t \in \bc{K(S)}$ then $\Supph(t)=\Supphnaive(t) = \emptyset$ and so certainly $\Supph_A(t)=\emptyset$ for any $A \in \cat T$. Thus, if $A$-relative h-detection holds, then $\LA t=0$. In other words, $\bc{K(S)} \subseteq \bc{A}$. Hence the localization $\cat T\to \LAT$ factors uniquely through the localization $\cat T \to L_{K(S)}\cat T$.
\end{proof}

\begin{Rem}
	The universal localization $L_{K(S)}\cat T$ is always nonzero --- provided $\cat T$ itself is nonzero. This follows simply from the fact that $K(S) \neq 0$ as soon as $S$ is nonempty. We will next consider the opposite extreme case: when $S=\{\cat B\}$ is a singleton. First we recall the following notation and terminology:
\end{Rem}

\begin{Def}\label{def:steel-condition}
	We say that $\cat T$ satisfies the \emph{steel condition} when the canonical map
	\begin{equation}
		\pi\colon\Spc^h(\cat T^c)\to \Spc(\cat T^c)
	\end{equation}
	is a bijection. It is always surjective by \cite[Corollary~3.9]{Balmer20_nilpotence}.
\end{Def}

\begin{Lem}\label{lem:complete-to-residue}
	Let $\cat B \in \Spc^h(\cat T^c)$ be such that $\cat P\coloneqq\pi(\cat B)\in \Spc(\cat T^c)$ is weakly visible. The localization $\cat T\to \LKBT$ induces a canonical functor
	\begin{equation}\label{eq:comp-to-loc}
		(\cat T_{\cat P})_{\cat P}^\wedge \to \LKBT.
	\end{equation}
	This is an equivalence if and only if
	\begin{equation}\label{eq:equiv-implication}
		\cat P \in \Supp(t) \Longrightarrow \cat B \in \Supphnaive(t)
	\end{equation}
	for all $t\in \cat T$.
\end{Lem}

\begin{proof}
	Recall from~\cref{exa:gP} that the kernel of $\cat T \to (\cat T_{\cat P})_{\cat P}^\wedge$ is $\bc{g_{\cat P}}$. The inclusion $\bc{g_{\cat P}} \subseteq \bc{K(\cat B)}$ follows from the fact that $K(\cat B) = E_{\cat B} \simeq E_{\cat B} \otimes g_{\cat P}$, which can be found in the proof of \cite[Lemma~8.6]{Zou25supp}. This provides the factorization~\eqref{eq:comp-to-loc} bearing in mind~\cref{exa:localization-is-well-generated}. On the other hand, the inclusion $\bc{K(\cat B)}\subseteq \bc{g_{\cat P}}$ is equivalent to condition~\eqref{eq:equiv-implication}.
\end{proof}

\begin{Rem}
	Note that the condition~\eqref{eq:equiv-implication} implies that $\pi^{-1}(\{\pi(\cat B)\}) = \{\cat B\}$.
\end{Rem}

\begin{Cor}\label{cor:almost-strat}
	Let $\cat T$ be a rigidly-compactly generated tt-category. Suppose that 
	\begin{enumerate}
		\item the spectrum $\Spc(\cat T^c)$ is weakly noetherian; and
		\item the h-detection property holds for $\cat T$; and
		\item the steel condition holds for $\cat T$.
	\end{enumerate}
	We have an equivalence
	\[
		L_{K(\cat P)}\cat T \cong (\cat T_{\cat P})_{\cat P}^\wedge.
	\]
	for every $\cat P \in \Spc(\cat T^c)\cong\Spc^h(\cat T^c)$.
\end{Cor}

\begin{proof}
	It follows from \cite[Theorem~C]{bhs2} that $\Supph(t) = \pi^{-1}(\Supp(t))$ for all $t \in \cat T$. Using \Cref{rem:inclusion} we see that condition~\eqref{eq:equiv-implication} holds and we can invoke~\cref{lem:complete-to-residue}.
\end{proof}

\begin{Cor}\label{cor:stratified-stalk}
	Suppose that $\cat T$ is tt-stratified. We have an equivalence
	\begin{equation}\label{eq:stratified-stalk}
		L_{K(\cat P)}\cat T \cong (\cat T_{\cat P})_{\cat P}^\wedge.
	\end{equation}
	for every $\cat P \in \Spc(\cat T^c)\cong\Spc^h(\cat T^c)$.
\end{Cor}

\begin{proof}
	We can apply~\cref{cor:almost-strat} by invoking \cite[Theorem~B]{bhs2}. Note that it is implicit in the tt-stratification hypothesis that the spectrum is weakly noetherian.
\end{proof}

\begin{Exa}
	The $E(n)$-local category $\cat T=\Sp_{E(n)}$ is tt-stratified; see \cite[Section~10]{bhs1} and \cite[Section~6]{HoveyStrickland99}. In this example, the equivalence \eqref{eq:stratified-stalk} recovers the fact that the completion of the $E(n)$-local category at its unique closed point coincides with $K(n)$-localization. This is originally due to Hovey--Strickland; see \cite[Theorem~6.19]{HoveyStrickland99} and \cite[Proposition~8.11]{BalmerSanders25}. However, \cref{cor:stratified-stalk} shows that it is just a manifestation of a very general phenomenon.
\end{Exa}

\begin{Exa}\label{exa:other-chromatic-completion}
	For any $0 \le n < \infty$, we can apply~\cref{lem:complete-to-residue} to the $p$-local stable homotopy category. We obtain
	\begin{equation}\label{eq:comp-chromatic}
		(\Sp|_{[0,n]})_n^\wedge \to L_{K(n)}\Sp
	\end{equation}
	which can be alternatively regarded as $L_{g(n)}\Sp \to L_{K(n)}\Sp$ by~\cref{exa:gP}. As explained in \cite[Remark~5.13]{bhs2}, we have $\bc{g(n)}=\bc{T(n)}$ where $T(n)$ is the telescope of a finite type $n$ spectrum. Hence \eqref{eq:comp-chromatic} identifies with
	\[
		L_{T(n)}\Sp \to L_{K(n)}\Sp
	\]
	which is an equivalence if and only if the telescope conjecture holds at height $n$.
\end{Exa}

\begin{Lem}\label{lem:local-compact}
	If $\cat T$ is local and $\cat P=\pi(\cat B)$ is the unique closed point then the condition~\eqref{eq:equiv-implication} holds if the localizing ideal $\Loco{\EB}$ contains a nonzero compact object.
\end{Lem}

\begin{proof}
	If $0 \neq c \in \Loco{\EB}$ is compact then
	\begin{equation}\label{eq:locoeY-locoEb}
		\Loco{e_Y} \subseteq \Loco{\EB}
	\end{equation}
	where $Y\coloneqq\supp(c)$. Moreover
	\[
		\emptyset \neq \Supph(e_Y) = \pi^{-1}(Y) \subseteq \{\cat B\}
	\]
	so that $\pi^{-1}(Y)=\{\cat B\}$. In particular, $\supp(c)=Y=\{\cat P\}$. Note that for this unique closed point, $g_{\cat P} = e_{\{\cat P\}}=e_Y$. Next observe that 
	\[
		\cat B \not\in \pi^{-1}(Y)^c = \Supph(f_Y)=\Supphnaive(f_Y)
	\]
	so that $\EB\otimes f_Y=0$. Hence $\EB \simeq \EB \otimes e_Y$. Thus~\eqref{eq:locoeY-locoEb} is an equality
	\[
		\Loco{g_{\cat P}} = \Loco{e_Y} = \Loco{\EB}.
	\]
	This implies that $\bc{g_{\cat P}} = \bc{\EB}$ which implies that \eqref{eq:equiv-implication} holds.
\end{proof}

\begin{Cor}
	Let $R$ be a commutative ring. We have an equivalence
	\[
		L_{\kappa(\frakp)} \Der(R) \cong \Der(R_\frakp)_{\frakp}^\wedge
	\]
	for any finitely generated prime ideal $\frakp \in \Spec(R)$.
\end{Cor}

\begin{proof}
	The key fact is that for any finitely generated ideal $I \subseteq R$, we have $\Loco{R/I}=\Loco{\kos(I)}$ in $\Der(R)$ where $\kos(I)$ denotes any choice of Koszul complex for $I$; see \cite[Proposition~6.4]{DwyerGreenlees02}. Also recall that the pure-injective $K(\frakp)=E_{\frakp}$ associated to $\frakp \in \Spec(R)$ coincides with the residue field $\kappa(\frakp)$. With this in hand, if we assume that our ring $R$ is local and that $\frakp$ is the maximal ideal, then $L_{\kappa(\frakp)}\Der(R) \cong \Der(R)_{\frakp}^{\wedge}$ by~\cref{lem:complete-to-residue} and~\cref{lem:local-compact}. The general statement reduces to the local case since $\Der(R) \to L_{K(\frakp)}\Der(R)$ factors as $\Der(R) \to \Der(R)_{\frakp} \cong \Der(R_{\frakp}) \to L_{K(\frakp)}\Der(R_\frakp)\cong L_{K(\frakp)}\Der(R)$.
\end{proof}

\begin{Rem}
	As we established in \cref{thm:only-for-residues}, the key to classifying the homological Bousfield classes of $\LAT$ is to understand whether $\bc{A}$ is generated by a coproduct of $\EB$'s. (Recall \cref{rem:only-for-residues}.) It may be the case that we cannot always explicitly identify the pure-injective objects $\EB$. However, for the purposes of classification, it suffices to determine them up to Bousfield equivalence. Our final goal of this section is to provide methods for achieving this. This will lead to additional examples where we can compute the Bousfield lattice. First we collect some auxiliary lemmas.
\end{Rem}

\begin{Lem}\label{lem:tensor-force}
	Let $\cat A$ be a closed symmetric monoidal additive category. If $A \to B$ is a unital map of weak rings in $\cat A$ then for any $X \in \cat A$ we have:
	\begin{align}
		A \otimes X = 0 \quad&\Longrightarrow\quad B\otimes X=0, \text{ and}\label{eq:tensor-force}\\
		\ihom{A,X}=0 \quad&\Longrightarrow\quad \ihom{B,X}=0.\label{eq:hom-force}
	\end{align}
\end{Lem}

\begin{proof}
	The identity map of $B$ factors through $B \otimes A$. Hence, the identity map of $B \otimes X$ factors through $B \otimes A \otimes X$ and the identity map of $\ihom{B,X}$ factors through $\ihom{B \otimes A,X} \simeq \ihom{B,\ihom{A,X}}$.
\end{proof}

\begin{Rem}\label{rem:third-formal-fails}
	In general, for a unital map of weak rings $A \to B$,
	\begin{align}\label{eq:third-formal-fails}
		\ihom{X,A}=0\quad \centernot\Longrightarrow \quad\ihom{X,B}=0.
	\end{align}
	See \cref{rem:homo-coho-different-finished} for a counterexample.
\end{Rem}

\begin{Lem}\label{lem:counit-ring}
	Let $F\colon\cat A\adjto \cat B:G$ be an adjunction of symmetric monoidal additive categories. If $B \in \cat B$ is a weak ring, then the counit $F(G(B)) \to B$ is a unital map of weak rings.
\end{Lem}

\begin{proof}
	This is routine from the definitions.
\end{proof}

\begin{Prop}\label{prop:base-change}
	Let $f^*\colon\cat T \to \cat S$ be a geometric functor of rigidly-compactly generated tt-categories and let $\varphi^h\colon\Spc^h(\cat S^c) \to \Spc^h(\cat T^c)$ be the induced map. For any~$\cat C \in \Spc^h(\cat S^c)$, the following statements hold:
	\begin{enumerate}
	\item	There is a split monomorphism $\EvarphiC \to f_*(\EC)$ in $\cat T$ which is a unital map of weak rings.
	\item	Consequently, for any $t \in \cat T$, we have
			\begin{align*}
				\EvarphiC \otimes t = 0 \quad&\Longleftrightarrow\quad f_*(E_{\cat C})\otimes t = 0, \text{ and}\\
				\ihom{\EvarphiC,t}=0 \quad&\Longleftrightarrow\quad \ihom{f_*(E_{\cat C}),t}=0.
			\end{align*}
	\item	For any $t \in \cat T$, we have
			\begin{align*}
				\Supphnaive_{\cat S}(f^*(t)) &= (\varphi^h)^{-1}(\Supphnaive_{\cat T}(t)), \text{ and}\\
				\Cosupp^h_{\cat S}(f^!(t)) &= (\varphi^h)^{-1}(\Cosupp^h_{\cat T}(t)).
			\end{align*}
	\item	Consequently, for any $t \in \cat T$, we have
			\begin{align*}
				\varphi^h(\Supphnaive_{\cat S}(f^*(t))) &= \Supphnaive_{\cat T}(t) \cap \im \varphi^h, \text{ and}\\
				\varphi^h(\Cosupp^h_{\cat S}(f^!(t))) &= \Cosupp^h_{\cat T}(t) \cap \im \varphi^h.
			\end{align*}
	\end{enumerate}
\end{Prop}

\begin{proof}
	We recall the definition of h-cosupport $\Cosupp^h$ in~\cref{def:h-cosupport} below.

	$(a)$: The map $\EvarphiC \to f_*(E_{\cat C})$ is constructed in \cite[Lemma~5.6]{Balmer20_bigsupport}. As explained there, setting $\cat B\coloneqq \varphiC$, we have a diagram
	\begin{equation}\label{eq:paul-diagram}\begin{tikzcd}[row sep=tiny]
		\cat T \ar[dd,shift right,"F"'] \ar[dd,<-,shift left,"U"] \ar[r] & \Mod(\cat T^c)\ar[dd,shift right,"\hat{F}"']\ar[dd,<-,shift left,"\hat{U}"]\ar[dr,shift right,"Q_{\cat B}"'] \ar[dr,<-,shift left,"R_{\cat B}"] &\\
		&& \Mod(\cat T^c)/\Loc(\cat B)\eqqcolon \bar{\cat A}_{\cat B}\ar[dd,shift right,"\bar{F}"'] \ar[dd,<-,shift left,"\bar{U}"]\\
		\cat S \ar[r] & \Mod(\cat S^c)\ar[dr,shift right,"Q_{\cat C}"'] \ar[dr,<-,shift left,"R_{\cat C}"]&\\
		&& \Mod(\cat S^c)/\Loc(\cat C)\eqqcolon \bar{\cat A}_{\cat C}
	\end{tikzcd}\end{equation}
	where $F\coloneqq f^*$ and $U\coloneqq f_*$. The object $\ECbar$ is constructed as the injective hull $\bar{\unit} \hookrightarrow \ECbar$ of the unit $\bar{\unit}$ in $\ACbar$. Note that the functor $\bar{U}$ is lax monoidal and preserves injectives, since its left adjoint $\bar{F}$ is strong monoidal and exact. The unit $\unitbar\to\Ubar(\ECbar)$ of the induced ring structure on $\Ubar(\ECbar)$ is a monomorphism by \cite[Proposition~3.5]{Balmer20_bigsupport}. Since $\Ubar(\ECbar)$ is injective, this monomorphism $\unitbar \to\Ubar(\ECbar)$ factors through the injective hull $\unitbar \hookrightarrow\EBbar$. The induced map $\EBbar \to \Ubar(\ECbar)$ is then a split monomorphism, and it is unital by construction. These two facts propagate backwards to the associated map $\EB \to U(\EC)=f_*(\EC)$ in $\cat T$ using basic facts about~\eqref{eq:paul-diagram} explained in \cite[Lemma~5.6]{Balmer20_bigsupport}.
	
	$(b)$: This follows from part $(a)$ together with \cref{lem:tensor-force}.

	$(c)$: Observe that $\EvarphiC \otimes t\neq 0$ is equivalent to $f_*(E_{\cat C} \otimes f^*(t))\simeq f_*(E_{\cat C})\otimes t \neq 0$ by part $(b)$. This evidently implies that $E_{\cat C}\otimes f^*(t) \neq 0$. On the other hand, since~$E_{\cat C}$ is a weak ring, \cite[Remark~13.12]{barthel2023cosupport} shows that $E_{\cat C}\otimes f^*(t) \neq 0$ implies $f_*(E_{\cat C} \otimes f^*(t)) \neq 0$. Hence, $\EvarphiC\otimes t \neq 0$ if and only if $\EC\otimes f^*(t)\neq 0$. This establishes the first equality.

	Next note that we have a unital map $f^*(\EvarphiC) \to \EC$ by part~$(a)$ and \cref{lem:counit-ring}. Then, invoking \cref{lem:tensor-force}, we have
	\begin{align*}
		\ihom{\EC,f^!(t)} \neq 0 &\imp \ihom{f^*(\EvarphiC),f^!(t)} \neq 0 \\
								 &\eqv f^!\ihom{\EvarphiC,t}\neq 0\\
								 &\imp \ihom{\EvarphiC,t}\neq 0\\
								 &\eqv \ihom{f_*\EC,t} \neq 0 \\
								 &\eqv f_*\ihom{\EC,f^!(t)} \neq 0 \\
								 &\imp \ihom{\EC,f^!(t)} \neq 0
	\end{align*}
	where we have also invoked part~$(b)$. This establishes the second equality in part~$(c)$.

	$(d)$: This is an immediate consequence of part~$(c)$.
\end{proof}

\begin{Rem}
	Note that \cref{prop:base-change} provides an equality of Bousfield classes
	\begin{equation}
		\bc{\EvarphiC} = \bc{f_*(\EC)}
	\end{equation}
	for any geometric functor $f^*\colon\cat T\to \cat S$ and $\cat C \in \Spc^h(\cat S^c)$.
\end{Rem}

\begin{Exa}
	Suppose $\cat T$ has enough tt-fields, so that for each $\cat B\in \Spc^h(\cat T^c)$ there exists a geometric functor $f^*\colon \cat T\to \cat F$ to a tt-field $\cat F$ which maps the unique point of $\Spc^h(\cat F^c)$ to $\cat B$. Then there is an equality of Bousfield classes
	\[
		\bc{\EB}=\bc{f_*(\unit)}.
	\]
	Note that $\EB$ and $f_*(\unit)$ need not coincide. For a field extension $K/F$ and base-change $f^*\colon\Der(F)\to\Der(K)$ we have $\EB=F$ but $f_*(\unit) = F^{\dim_F(K)}$.
\end{Exa}

\begin{Cor}\label{cor:induced-residue-fields}
	Let $\cat T$ be a rigidly-compactly generated tt-category such that $\Supph=\Supphnaive$. Suppose $(f_i^*\colon\cat T \to \cat S_i)_{i \in I}$ is a family of geometric functors. For each $i \in I$, let $S(i) \subseteq \Spc^h(\cat S_i^c)$ be a subset and define
	\[
		A\coloneqq \coprod_{i \in I} \coprod_{\cat C \in S(i)} (f_i)_*(\EC).
	\]
	Then relative h-detection holds for $\LAT$.
\end{Cor}

\begin{proof}
	By \cref{prop:base-change}(b), we have an equality of Bousfield classes $\bc{A} = \bc{B}$ where
	\[
		B\coloneqq\coprod_{i \in I}\coprod_{\cat C \in S(i)} \EvarphiiC.
	\]
	Note that $B=K(S)$ for $S=\bigcup_{i \in I}\varphi_i^h(S(i))$. We may then apply \cref{thm:generalized-harmonic}.
\end{proof}

We will also need to be able to descend the condition that $\Supph = \Supphnaive$. 

\begin{Lem}\label{lem:descent-for-supph=suppn}
	Let $(f_i^* \colon \cat T \to \cat S_i)_{i\in I}$ be a family of geometric functors whose induced maps $\varphi_i^h \colon \Spc^h(\cat S_i^c) \to \Spc^h(\cat T^c)$ are jointly surjective. Suppose that for each $i \in I$ and~$\cat C \in \Spc^h(\cat S_i^c)$, one of the following conditions holds:
    \begin{enumerate}
		\item The localizing ideal $\Loco{\EC} \subseteq \cat S_i$ is a minimal localizing ideal; or
		\item The pure-injective object $\EC$ is a field-object.
    \end{enumerate}
    Then $\Supph(t) = \Supphnaive(t)$ for any $t \in \cat T$.
\end{Lem}

\begin{proof}
	Let $t \in \cat T$ and suppose that $\ihom{t,\EB}=0$ for some $\cat B \in \Spc^h(\cat T^c)$. We need to prove that $t \otimes \EB=0$. Choose an $i \in I$ and $\cat C \in \Spc^h(\cat S_i^c)$ such that $\cat B = \varphi_i^h(\cat C)$. Then $\EB$ is a direct summand of $(f_i)_*(\EC)$ by \cite[Lemma~5.6]{Balmer20_bigsupport}. Suppose for a contradiction that $t \otimes \EB \neq 0$. Then $(f_i)_*(f_i^*(t) \otimes \EC)\simeq t\otimes (f_i)_*(\EC) \neq 0$. It follows that $f_i^*(t) \otimes \EC \neq 0$. Case $(a)$ would then imply that $\EC \in \Loco{f_i^*(t)}$. On the other hand, case $(b)$ would imply that a shift of $\EC$ is a summand of $f_i^*(t)\otimes \EC$, which also gives $\EC \in \Loco{f_i^*(t)}$. In both cases, we obtain $(f_i)_*(\EC) \in \Loco{t}$ by \cite[Remark~13.3]{barthel2023cosupport}. Hence the hypothesis $\ihom{t,\EB}=0$ would imply $\ihom{(f_i)_*(\EC),\EB}=0$ so that $\ihom{\EB,\EB}=0$ which is a contradiction. We conclude that $t\otimes \EB =0$, as desired.
\end{proof}

\begin{Exa}
	Let $G$ be a compact Lie group and let $\Sp_G$ denote the category of $p$-local $G$-spectra. Then $\Spc^h(\Sp_G^c)\cong \Spc(\Sp_G^c)$, and the homological primes are indexed by pairs $(H,n)$ with $H\le G$ a closed subgroup (up to conjugacy) and $n\in\mathbb N\cup\{\infty\}$; see \cite{BalmerSanders17,bgh_balmer}. We denote the corresponding pure-injective objects by $E_{(H,n)}$. The homological primes arise as pullbacks of the non-equivariant primes along the geometric fixed-point functors
	\begin{equation}\label{eq:geom-fixed}
		\Phi^H\colon \Sp_G\longrightarrow \Sp .
	\end{equation}
	Let $\Psi^H\colon \Sp\to \Sp_G$ denote the right adjoint of $\Phi^H$. We have:
\end{Exa}

\begin{Thm}\label{thm:bc-chromatic-equiv-localizations}
	Fix a prime $p$. For each conjugacy class of closed subgroups $H\le G$ choose a subset $S(H)\subseteq \mathbb N\cup\{\infty\}$. Let
	\[
		K_G(S) \coloneqq \coprod_{[H]\in \Sub(G)/G}\;\coprod_{i\in S(H)} \Psi^H K(i) \in \Sp_G
	\]
	and define the set
	\[
		\mathcal D_S \coloneqq \SETT{(H,i)}{[H]\in \Sub(G)/G,\ i\in S(H)}.
	\]
	The relative h-detection property holds for the localized category $L_{K_G(S)}\Sp_G$ and there is a bijection
	\[
		\Big\{\text{homological Bousfield classes of }L_{K_G(S)}\Sp_G\Big\} \xrightarrow{ \sim }\ 
		\Big\{\text{subsets of }\mathcal D_S\Big\}.
	\]
\end{Thm}

\begin{proof}
	The geometric fixed point functors \eqref{eq:geom-fixed} are jointly conservative. Hence, by \cite[Theorem~1.9]{BarthelCastellanaHeardSanders24}, the induced maps on homological spectra are jointly surjective. Since $E_{\cat C_n} = K(n)$ is a field-object for all points $\cat C_n\in\Spc^h(\Sp^c)$, \cref{lem:descent-for-supph=suppn} implies that $\Supph(t) = \Supphnaive(t)$ for all $t \in \Sp_G$. We can then invoke~\cref{cor:induced-residue-fields} bearing in mind~\cref{thm:bousfield-classes}.
\end{proof}

\begin{Rem}
	The homological Bousfield class of a $G$-spectrum $A$ is determined by the Bousfield classes of its geometric fixed points:
	\begin{equation}\label{eq:geom-bc}
		\bc{A} = \SETT{ t \in \Sp_G }{ \Phi^H(t) \in \bc{\Phi^H(A)} \text{ for all } H \le G}.
	\end{equation}
	This is \cite[Proposition~3.2]{Hill2019Equivariant}.
\end{Rem}

\begin{Exa}\label{exa:johnson-wilson}
	Consider the real Johnson--Wilson $C_2$-spectrum $E_{\mathbb{R}}(n)$. From~\eqref{eq:geom-bc} we deduce that
	\begin{equation}\label{bc-ERn}
		\bc{E_{\mathbb{R}}(n)} = \bc{{C_2}_+ \otimes E(n)} = \coprod_{0 \le i \le n}\bc{{C_2}_+ \otimes K(i)}
	\end{equation}
	since $\Phi^{C_2}(E_{\mathbb{R}}(n))=0$. For the trivial subgroup $H=e$, we have $\Phi^e \simeq \res^{C_2}_e$ and the right adjoint $\Psi^e$ is coinduction. By the Wirthmüller isomorphism (for finite groups) coinduction agrees with induction, so
	\[
		{C_2}_+ \otimes (-) = \Ind^{C_2}_e(-) = \Psi^e(-).
	\]
	Hence~\eqref{bc-ERn} may be expressed as
	\[
		\bc{E_{\mathbb{R}}(n)} = \coprod_{0 \le i \le n}\bc{\Psi^eK(i)}.
	\]
	Therefore, by \cref{thm:bc-chromatic-equiv-localizations} we have a bijection
	\[
		\Big\{\text{homological Bousfield classes in }L_{E_{\mathbb{R}}(n)}\Sp_{C_2}\Big\}
		\xrightarrow{\ \sim\ }
		\Big\{\text{subsets of }\{0,\ldots,n\}\Big\}
	\]
	given by relative homological support. This localization is not smashing for $n \ge 1$ by \cite[Theorem~4.1]{Carrick2022Smashing}. See also \cite[Example~2.4]{Carrick2022Smashing}.
\end{Exa}

\section{Relative homological stratification}

Throughout this section, we let $\cat T$ be a rigidly-compactly generated tt-category and fix $A\in\cat T$. We develop a relative version of h-stratification for the $A$-local category $\LA\cat T$. When $A=\unit$ these results specialize to the theory developed in \cite{BarthelHeardSandersZou26}. We begin with the definitions. 

\begin{Def}\label{def:relative-hLGP}
	We say that the \emph{relative homological local-to-global principle} holds for $\LA\cat T$ if for every $x \in \LAT$ we have an equality 
	\[
	  \Loco{x} = \Loco{\LA(\EB\otimes x) \mid \cat B\in \Supph_A(x)}
	\]
	of localizing ideals in $\LAT$.
\end{Def}

\begin{Rem}\label{rem:hLGP-implies-hdetection}
	The	relative h-detection property (\cref{def:relative-hdetection}) follows from the relative h-local-to-global principle.
\end{Rem}

\begin{Thm}\label{thm:hstratfundamental}
	Let $\cat T$ be a rigidly-compactly generated tt-category and let $A\in\cat T$. The following are equivalent:
	\begin{enumerate}
	\item \label{it:lgp+minimal}
		The relative homological local-to-global principle holds for $\LA\cat T$, and for every $\cat B\in \Supph(A)$ the localizing ideal $\Loco{\EB}\subseteq \LA\cat T$ is minimal among nonzero localizing ideals of $\LA\cat T$.
	\item \label{it:primewise-gen}
		For every $x\in \LAT$, we have an equality 
		\[
			\Loco{x} = \Loco{\EB\mid \cat B\in \Supph_A(x)}
		\]
		of localizing ideals in $\LAT$.
	\item \label{it:classification}
		Homological support yields an order-preserving bijection
		\[
			\Supph_A\colon
			\bigl\{\text{localizing ideals of }\LA\cat T\bigr\}
			\xrightarrow{\cong}
			\bigl\{\text{subsets of }\Supph(A)\bigr\}.
		\]
	\end{enumerate}
\end{Thm}

\begin{proof}
	The argument of \cite[Theorem~4.1]{BarthelHeardSandersZou26} carries over verbatim. Note that $\cat B \in \Supph(A)$ implies that $E_{\cat B}$ is $A$-local by \cref{lem:completion-of-eb}.
\end{proof}

\begin{Def}\label{def:relative-stratification}
	We say that $\LA\cat T$ is \emph{relatively h-stratified} if the equivalent conditions of \cref{thm:hstratfundamental} hold. 
\end{Def}

\begin{Rem}\label{rem:support-stratified}
	If $\LAT$ is relatively h-stratified then $\Supph_A(t) = \Supphnaive_A(t)$ for all $t \in \cat T$. This is a consequence of \cref{lem:local-h-detection} since h-stratification implies h-detection.
\end{Rem}

\begin{Prop}\label{prop:bousfield-stratified}
	If $\LAT$ is relatively h-stratified then every localizing ideal of~$\LAT$ is a homological Bousfield class.
\end{Prop}

\begin{proof}
	Let $\cat L$ be a localizing ideal of $\LA\cat T$ and set $S\coloneqq \Supph(A)\setminus\Supph_A(\cat L)$. We claim that $\cat L = \bc{\LA K(S)}$. Indeed:
	\reqnomode
	\begin{align*}
		x \in \bc{\LA K(S)} 
		&\iff A\otimes x\otimes \EB = 0 \text{ for all }\cat B \in S \qquad\qquad
		\\
		& \iff \Supphnaive_A(x) \cap S = \emptyset
		\hfill\tag*{\text{ (\cref{def:supphA})}}
		\\
		& \iff \Supph_A(x) \cap S = \emptyset
		\hfill\tag*{\text{ (\cref{rem:support-stratified})}}
		\\
		& \iff \Supph_A(x) \subseteq \Supph_A(\cat L) \\
		& \iff x \in \cat L
		\hfill
		\tag*{\text{ (h-stratification)}}
	\end{align*}
	for all $x \in \LA\cat T$.
\end{proof}

\begin{Rem}
	We saw in \cref{thm:generalized-harmonic} that the localizations $L_{K(S)}\cat T$ always have the relative h-detection property, at least under the assumption that $\smash{\Supph=\Supphnaive}$. However, we will see in \cref{sec:oddball_cohomological_bousfield_classes} that these categories can have localizing ideals which are not homological Bousfield classes. Consequently, $L_{K(S)}\cat T$ need not be relatively h-stratified.
\end{Rem}

\begin{Rem}
	For any well generated tensor-triangulated category $\cat T$, there is a spectrum $\Spc(\cat T)$ introduced by Krause and Letz \cite{KrauseLetz23}. It is defined to be the topological space associated to the frame of radical localizing ideals of~$\cat T$. If $\LA\cat T$ is relatively h-stratified, then its spectrum $\Spc(\LA\cat T)$ can be identified with $\Supph(A)$ regarded as a discrete topological space. Indeed, note that h-stratification ensures that every localizing ideal $\cat L$ of $\LA\cat T$ is radical:
	\[
		x^{\otimes n} \in \cat L \implies \Supph_A(x) = \Supph_A(x^{\otimes n}) \subseteq \Supph_A(\cat L) \implies x \in \cat L.
	\]
	Then, h-stratification states that the frame of (radical) localizing ideals of $\LAT$ is isomorphic to the frame of subsets of $\Supph(A)$, which proves the claim. 
\end{Rem}

\begin{Rem}
	As in \cite{BarthelHeardSandersZou26} we will need a theory of \emph{cosupport} in order to best study homological stratification. 
\end{Rem}

\begin{Def}\label{def:h-cosupport}
	The \emph{homological cosupport} of $t\in\cat T$ is
	\[
		\Cosupph(t)\coloneqq \SETT{\cat B\in\Spc^h(\cat T^c)}{\ihom{\EB,t}\neq 0}.
	\]
	We say that the \emph{relative h-codetection property} holds for $\LA\cat T$ if $\Cosupph(x) = \emptyset$ implies $x=0$ for all $x \in \LA\cat T$.
\end{Def}

\begin{Rem}\label{rem:cosupp-in-supp}
	An immediate consequence of \cref{lem:completion-of-eb} is that
	\[
		\Cosupph(x) \subseteq \Supphnaive(A)
	\]
	for any $x \in \LAT$.
\end{Rem}

\begin{Lem}\label{lem:cosupp-EB}
	For every $\cat B \in \Supphnaive(A)$, we have
	\[
		\Cosupp^h(\LA \EB) = \{\cat B\}.
	\]
\end{Lem}

\begin{proof}
	The endofunctor $\LA\colon\cat T \to \cat T$ is lax monoidal and $\EB \to \LA \EB$ is a unital map of weak rings. Hence
	\begin{align*}
		t \otimes \EB=0 \underset{\eqref{eq:tensor-force}}{\implies} t\otimes\LA\EB = 0  \underset{\eqref{eq:bc-incl}}{\implies} \ihom{t,\LA\EB} = 0 
	\end{align*}
	for any $t \in \cat T$. Applied to $t=\EC$, we conclude that $\Cosupp^h(\LA\EB) \subseteq \{\cat B\}$. On the other hand, the object $\EB$ is not $A$-acyclic since $\cat B \in \Supphnaive(A)$. Hence $\ihom{\EB,\LA\EB} \simeq \ihom{\LA\EB,\LA\EB} \neq 0$. Thus $\cat B \in \Cosupp^h(\LA\EB)$.
\end{proof}

\begin{Lem}\label{lem:cosupp-IAx}
	For every $x \in \LAT$, we have
	\[
		\Cosupp^h(I_A(x)) = \Supphnaive_A(x).
	\]
\end{Lem}

\begin{proof}
	Recall \cref{rem:bc-vanishing}. For any $\cat B \in \Spc^h(\cat T^c)$, we have $\ihom{\EB,I_A(x)}\cong I_A(\EB\otimes x) = 0$ if and only if $A \otimes \EB \otimes x =0$.
\end{proof}

\begin{Rem}\label{rem:homo-coho-different-finished}
	There exist categories with $\Cosupph(\unit) \subsetneq \Spc^h(\cat T^c)$. See \cite[Example~8.6]{barthel2023cosupport} and \cite[Proposition~6.8]{BarthelHeardSandersZou26}. We can then finish \cref{rem:third-formal-fails} by taking $A=\unit$ and $X=B=\EB$ for any $\cat B \not\in\Cosupph(\unit)$. Also, the converse of~\eqref{eq:hom-force} fails for $A=\unit$ and $X=B=\EB$. Finally, the converse of~\eqref{eq:tensor-force} fails for $A=X=\unit$ and $B=\EB$.
\end{Rem}

\begin{Rem}
	The relative h-codetection property is, by definition, simply the absolute h-codetection property restricted to the $A$-local objects. In contrast, we explained in \cref{rem:hdetect-not-restricted} that relative h-detection is not (in general) just the restriction of h-detection to the $A$-local objects. With this in mind, consider the following two statements:
	\begin{enumerate}
		\item For any $t \in \cat T$, if $\Cosupph(t)=\emptyset$ then $\LA t=0$.
		\item For any $t \in \cat T$, if $\Cosupph(t)=\emptyset$ then $\Cosupph(\LA t)=\emptyset$.
	\end{enumerate}
	Property (a) is \emph{a priori} stronger than relative h-codetection. In fact, (a) is equivalent to relative h-codetection together with property (b). Justification for our definition of relative h-codetection (over, say, property (a)) lies in the following result:
\end{Rem}

\begin{Prop}\label{prop:hLGP-hcodetect}
	The following are equivalent:
	\begin{enumerate}
		\item The relative h-local-to-global principle holds for $\LAT$.
		\item $\Supph_A = \Supphnaive_A$ and relative h-codetection holds for $\LAT$.
	\end{enumerate}
\end{Prop}

\begin{proof}
	Under the hypothesis $\Supph_A = \Supphnaive_A$, the h-LGP is equivalent to 
	\begin{equation}\label{eq:hLGP-unit}
		\LA \unit  \in \Loco{\EB \mid \cat B \in \Supphnaive(A)}
	\end{equation}
	bearing in mind \cref{cor:local-acyclic}. Also note that $(a)$ implies that $\Supph_A = \Supphnaive_A$ by \cref{rem:hLGP-implies-hdetection} and \cref{lem:local-h-detection}. We thus need to establish that \eqref{eq:hLGP-unit} is equivalent to h-codetection. Consider $x \in \LAT$ and observe that
	\[
		x=0 \Longleftrightarrow \ihom{\unit,x} = 0 \Longleftrightarrow \ihom{\LA \unit,x} = 0 \Longleftrightarrow x \in (\Loco{\LA\unit})^\perp
	\]
	while
	\[
		\Cosupph(x) = \emptyset \Longleftrightarrow x \in (\Loco{\EB \mid \cat B \in \Supphnaive(A)})^\perp.
	\]
	Thus, h-codetection is equivalent to 
	\[
		(\Loco{\EB \mid \cat B \in \Supphnaive(A)})^\perp \subseteq (\Loco{\LA\unit})^\perp.
	\]
	Recall that set-generated localizing ideals $\cat L$ are strictly localizing (since $\LAT$ is well generated) and hence ${\cat L = {}^\perp(\cat L^\perp)}$. Thus, h-codetection is equivalent to 
	\[
		\Loco{\LA \unit} \subseteq \Loco{\EB \mid \cat B \in \Supphnaive(A)}.
	\]
	We thus see that \eqref{eq:hLGP-unit} is equivalent to h-codetection, which completes the proof.
\end{proof}

The following generalizes \cite[Theorem~5.6]{BarthelHeardSandersZou26} to the relative context. 

\begin{Thm}\label{thm:equivalent-to-stratification} 
	Let $\cat T$ be a rigidly-compactly generated tt-category. For any $A \in \cat T$, the following are equivalent:
	\begin{enumerate}
	\item	relative h-stratification holds for $\LA\cat T$.
	\item	relative h-codetection holds for $\LA\cat T$ and for all $x,y\in \LA\cat T$,
			\[
				\Cosupph\big(\ihom{x,y}\big)=\Supph_A(x)\cap \Cosupph(y).
			\]
	\item	For all $x,y\in \LA\cat T$,
			\[
				\ihom{x,y}= 0 \;\;\Longleftrightarrow\;\; \Supph_A(x)\cap \Cosupph(y)=\emptyset.
			\]
	\end{enumerate}
\end{Thm}

\begin{proof}
	$(a)\Rightarrow(b)$: By \cref{prop:hLGP-hcodetect}, relative h-stratification implies that relative h-codetection holds and that $\Supph_A = \Supphnaive_A$. If $\cat B \not\in \Supph_A(x)=\Supphnaive_A(x)$ then $\EB\otimes x$ is $A$-acyclic and hence $\ihom{\EB\otimes x,y} = 0$ since $y$ is $A$-local. If $\cat B \notin \Cosupph(y)$, i.e., $\ihom{\EB, y} = 0$, then by adjunction
	\[
		\ihom{\EB, \ihom{x,y}} \simeq \ihom{x, \ihom{\EB, y}} = 0
	\]
	so $\cat B \notin \Cosupph(\ihom{x,y})$. We conclude that
	\[
		\Cosupph\big(\ihom{x,y}\big)\subseteq \Supph_A(x)\cap \Cosupph(y).
	\]
	For the reverse inclusion, if $\cat B\in \Supph_A(x)\cap \Cosupph(y)$ then $\LA(\EB\otimes x)\neq 0$ and, since $\Loco{\EB}$ is minimal under h-stratification (\cref{thm:hstratfundamental}), we have $\EB\in \Loco{\LA(\EB\otimes x)}$. Recall from \cite[Remark~2.5]{barthel2023cosupport} that since $\ihom{-,y}$ pulls colocalizing coideals back to localizing ideals, it follows that the nonzero object $\ihom{\EB,y}$ is contained in the colocalizing coideal generated by 
	\[
		\ihom{\LA(\EB\otimes x),y} \simeq \ihom{\EB,\ihom{x,y}}.
	\]
	Hence $\cat B \in \Cosupph(\ihom{x,y})$, as desired.

	$(b)\Rightarrow(c)$: This is immediate.

	$(c)\Rightarrow(a)$: From $(c)$ we immediately get both relative h-detection and relative h-codetection: if $x\neq 0$ then $\ihom{x,x}\neq 0$, hence $\Supph_A(x)\cap \Cosupph(x)\neq \emptyset$, so both are nonempty. In particular, $\Supph_A = \Supphnaive_A$ by \cref{lem:local-h-detection}. Hence, the relative h-LGP holds by \cref{prop:hLGP-hcodetect}.

	We now prove that the localizing ideal $\Loco{\EB}\subseteq\LAT$ is minimal for every $\cat B \in\Supph(A)$. Indeed, if $0\neq u,v\in \Loco{\EB}$ then $\cat B\in \Supph_A(u)$ by h-detection. If $\cat B\notin \Cosupph(v)$ then $\ihom{\EB,v}=0$. Since the kernel of $\ihom{-,v}$ is a localizing ideal which contains $\EB$, it would necessarily contain $v$ which would force the contradiction $v=0$. We conclude that $\cat B\in \Cosupph(v)$, and hence $(c)$ gives $\ihom{u,v}\neq 0$. Minimality then follows as in \cite[Lemma~7.12]{barthel2023cosupport}. One needs to be careful since $\LAT$ need not be rigidly-compactly generated. However, the required Bousfield localization exists in the well generated setting by \cite[Proposition~5.2.1]{Krause10}.

	This establishes that $\LAT$ is relatively h-stratified by \cref{thm:hstratfundamental}.
\end{proof}

\begin{Rem}\label{rem:hom-detection-formula}
    Although we have used $\Supph_A(x)$ in parts $(b)$ and $(c)$ of \cref{thm:equivalent-to-stratification}, note that
    \[
		\Supph_A(x) \cap \Cosupph(y) = \Supph(x) \cap \Cosupph(y)
    \]
    for $x,y \in \LA\cat T$ whenever $\Supph(A) = \Supphnaive(A)$. This is immediate from \cref{rem:cosupp-in-supp}. This allows for a quick proof that if $\cat T$ is h-stratified then so is~$\LAT$:
\end{Rem}

\begin{Thm}\label{thm:local-stratification}
	If $\cat T$ is h-stratified then $\LAT$ is relatively h-stratified for any $A \in \cat T$.
\end{Thm}

\begin{proof}
    We first note that if $\cat T$ is h-stratified then $\Supph(A) = \Supphnaive(A)$, for example by \cite[Lemma~2.11]{BarthelHeardSandersZou26}. By \cref{thm:equivalent-to-stratification,rem:hom-detection-formula} it suffices to show that
	\[
		\ihom{x,y}= 0 \;\;\Longleftrightarrow\;\; \Supph(x)\cap \Cosupph(y)=\emptyset
	\]
	for all $x,y\in \LA\cat T$. But since $\cat T$ is h-stratified, this already holds for all $x,y\in \cat T$.
\end{proof}

\begin{Exa}\label{exa:completion-stratified}
	Let $Y \subseteq \Spc(\cat T^c)$ be a Thomason subset. Recall from \Cref{exa:localization-completion} that Bousfield localization with respect to $A=e_Y$ provides the ``Bousfield completion''
	\[
		\cat T \to \cat T_Y^{\wedge}.
	\]
	If $\cat T$ is h-stratified then \cref{thm:local-stratification} implies that the localizing ideals of the completion $\cat T_Y^{\wedge}$ correspond to the subsets of $\Supph(e_Y)=\pi^{-1}(Y)$.
\end{Exa}

\begin{Rem}
	An alternative definition of ``completion'' would be to take an underlying model $\cat T=\Ho(\cat C)$ and consider the Ind-completion of the dualizable objects in the Bousfield completion:
	\[
		\cat T_{Y}^{\cwedge}\coloneqq \Ho(\Ind((\cat C_Y^{\wedge})^d)).
	\]
	This is nonstandard notation which we use here for ease of exposition. This definition is motivated by:\footnote{See \cite{NaumannPolRamzi24} and \cite[Remark~5.4]{BalmerSanders25} for further discussion.}
\end{Rem}

\begin{Exa}
	Let $R$ be a commutative noetherian ring and let $I \subseteq R$ be an ideal. Write $Y=V(I)$ for the associated closed subset of $\Spec(R)$. Since $\cat T=\Der(R)$ is h-stratified, \cref{thm:local-stratification} implies that $\cat T_Y^{\wedge} = \Der(R)_Y^{\wedge}$ is relatively h-stratified, as in~\cref{exa:completion-stratified}. Hence, its localizing ideals correspond to the subsets of~$Y$. On the other hand, by \cite[Theorem~1.4]{BalmerSanders_perfect}, we have an equivalence
	\[
		(\Der(R)_Y^{\wedge})^d \simeq \Der(R_I^{\wedge})^c.
	\]
	Hence $\cat T_Y^{\cwedge} \cong \Der(R_I^{\wedge})$ recovers the derived category of the ring-theoretic $I$-adic completion $R_I^{\wedge}$. This is also h-stratified since $R_I^{\wedge}$ is still noetherian. Hence the localizing ideals of $\cat T_Y^{\cwedge}$ correspond to arbitrary subsets of $\Spec(R_I^{\wedge})$. Moreover, the image of the induced map
	\begin{equation}\label{eq:completion}
		\Spec(R_I^{\wedge}) \to \Spec(R)
	\end{equation}
	is precisely the generalization closure of $Y$. For example, it is surjective precisely when $I$ is contained in the Jacobson radical of $R$, i.e., when $Y$ contains all the closed points. In general, if the image of~\eqref{eq:completion} coincides with~$Y$ then~$Y$ must be generalization closed which implies that it is both closed and open. This would imply that we are in a degenerate case where the ring splits and we are completing along a union of connected components. Moreover, $\Spec(R_I^{\wedge})\cong Y$ in this case. In summary, the localizing ideals of $\cat T_Y^{\wedge}$ correspond to the subsets of~$Y$ while the localizing ideals of~$\cat T_Y^{\cwedge}$ correspond to the subsets of~$\Spec(R_I^{\wedge})$ which is larger than~$Y$ except in the degenerate cases mentioned above. As an extreme example, take $R$ to be local and let $I =\mathfrak{m}$ be the maximal ideal. We see that~$\cat T_Y^{\wedge}$ has only two localizing ideals while $\cat T_Y^{\cwedge}$ will always have more, except in a trivial case where $\Spec(R)=\{\mathfrak{m}\}$, i.e., $R$ is artinian.
\end{Exa}

We now return to the chromatic example of \cref{cor:bc-chromatic-localizations}.

\begin{Prop}\label{prop:K(S)-stratified}
	Fix a prime $p$ and let $S \subseteq\mathbb{N}$ be a finite subset. Let $K(S)\coloneqq \coprod_{i\in S}K(i)$ be the wedge of the corresponding Morava $K$-theories. Then the category $\Sp_{K(S)}\coloneqq L_{K(S)}\Sp$ is relatively h-stratified. We have an order-preserving bijection
	\[
		\Supph_{K(S)}\colon
		\bigl\{\text{localizing ideals of }\Sp_{K(S)}\bigr\}
		\xrightarrow{\ \cong\ }
		\bigl\{\text{subsets of }S\bigr\}.
	\]
\end{Prop}

\begin{proof}
	First consider the case $S=\{0,1,\dots,n\}$. Then
	\[
		L_{K(S)}=L_{E(n)}\quad\text{and}\quad \Sp_{K(S)}=\Sp_{E(n)}.
	\]
	The category $\Sp_{E(n)}$ is h-stratified by \cite[Theorem~10.14]{bhs1} combined with \cite[Theorem~9.6]{BarthelHeardSandersZou26}. Now let $S\subseteq\{0,1,\dots,n\}$ be arbitrary with $n=\max(S)$. Since $\bc{K(S)}\le \bc{E(n)}$, we have a factorization
	\[
		\Sp_{K(S)} = L_{K(S)}\Sp \simeq  L_{K(S)}\Sp_{E(n)}.
	\]
	We now apply \cref{thm:local-stratification}.
\end{proof}

\begin{Rem}
	We will see in \cref{cor:harmonic-failure} that the above result is false whenever $S$ is infinite.
\end{Rem}

\begin{Rem}
	A relative notion of \emph{tt-stratification} for $\LAT$ can also be developed using a relative version of the Balmer--Favi support $\Supp$. Assuming that every point of $\Supp(A)$ is weakly visible, we define
	\[
		\Supp_A(t) \coloneqq \Supp(A\otimes t) \subseteq \Supp(A)
	\]
	and we say that $\LAT$ is \emph{relatively tt-stratified} if $\Supp_A$ induces a bijection between the localizing ideals of $\LAT$ and the set of all subsets of $\Supp(A)$.
\end{Rem}

\begin{Def}\label{def:relative-steel-condition}
	We say that the \emph{$A$-relative steel condition} holds for $\cat T$ if the comparison map $\pi\colon\Spc^h(\cat T^c)\to\Spc(\cat T^c)$ restricts to a bijection 
	\begin{equation}\label{eq:relative-steel}
		\Supph(A)\xra{\sim} \Supp(A).
	\end{equation}
	The $A=\unit$ case reduces to the usual steel condition (\cref{def:steel-condition}).
\end{Def}

\begin{Rem}
	We prove in \cite[Theorem~9.6]{BarthelHeardSandersZou26} that $\cat T$ is tt-stratified if and only if it is h-stratified and satisfies the steel condition. This theorem can be generalized to our relative setting:
\end{Rem}

\begin{Thm}\label{thm:tt=h+NS}
	Let $\cat T$ be a rigidly-compactly generated tt-category and let $A \in \cat T$. Assume that $\pi(\Supph(A))=\Supp(A)$ and that every point of $\Supp(A)$ is weakly visible in $\Spc(\cat T^c)$. The following are equivalent:
	\begin{enumerate}
		\item $\LA\cat T$ is relatively tt-stratified;
		\item $\LA\cat T$ is relatively h-stratified and the relative steel condition holds.
	\end{enumerate}
\end{Thm}

\begin{proof}[Sketch of proof]
	First one checks that if $\cat B \in \Supph(A)$ then $\Supp_A(\EB) = \{\pi(\cat B)\}$. Then one proves that relative tt-stratification implies the relative tensor-product property: $\Supp_A(t_1 \otimes t_2)=\Supp_A(t_1) \cap \Supp_A(t_2)$ for any $t_1,t_2 \in \cat T$. These two facts imply that $\pi|_{\Supph(A)}$ is injective. Next one shows that tt-stratification implies the relative h-detection property. This uses the hypothesis that $\pi(\Supph(A))=\Supp(A)$. With all these facts in hand, one obtains that the relative steel condition holds as in \cite[Proposition~3.13]{bhs2} and then we obtain h-stratification as in \cite[Theorem~9.6]{BarthelHeardSandersZou26}. This establishes $(a)\Rightarrow (b)$. Finally, relative h-stratification implies that relative h-detection holds, and this implies that $\pi(\Supph_A(t))=\Supp_A(t)$ for any $t\in \cat T$. With this in hand, relative h-stratification implies relative tt-stratification as in the proof of \cite[Theorem~9.6]{BarthelHeardSandersZou26}.
\end{proof}

\begin{Rem}
	The hypothesis $\pi(\Supph(A))=\Supp(A)$ deserves comment. We always have an inclusion $\pi(\Supph(A))\subseteq\Supp(A)$ but this need not be an equality. Note that the $A$-relative steel condition implies that it \emph{is} an equality. However, the absolute steel condition (\cref{def:steel-condition}) does \emph{not} necessarily imply the \mbox{$A$-relative} steel condition. For example, if $\Spc(\cat T^c)$ is noetherian then the Balmer--Favi support has the detection property by \cite[Theorem~3.22]{bhs1}. If $\cat T$ has a nonzero tensor-nilpotent object $A$ then $\Supph(A)=\emptyset$ while $\Supp(A) \neq \emptyset$. For example, the derived category $\Der(R)$ of any commutative ring satisfies the steel condition by \cite[Corollary~5.11]{Balmer20_nilpotence}. However, as explained in \cite[Example~5.7]{bhs2}, for the valuation domain $R$ of \cref{rem:keller} below, there exists a nonzero tensor-nilpotent object $A$ in $\Der(R)$ but the spectrum is a single point (hence noetherian). We conclude that $\Der(R)$ does not satisfy the $A$-relative steel condition.
\end{Rem}

\begin{Rem}
	The equality $\pi(\Supph(A))=\Supp(A)$ holds for any weak ring~$A$. This is \cite[Proposition~6.12]{BarthelHeardSandersZou26}. Moreover, for $A=\unit$, the hypothesis that every point of $\Supp(A)$ is weakly visible is equivalent to saying that $\Spc(\cat T^c)$ is weakly noetherian. Hence the $A=\unit$ case of \cref{thm:tt=h+NS} recovers~\cite[Theorem~9.6]{BarthelHeardSandersZou26}.
\end{Rem}

\begin{Rem}
	Suppose $S \subseteq \Spc^h(\cat T^c)$ is a subset such that $\pi(S) \subseteq \Spc(\cat T^c)$ consists of weakly visible points. Then $\pi(\Supph(K(S))) = \pi(S) = \Supp(K(S))$ using \cite[Lemma~3.7]{bhs2}. Hence $L_{K(S)}\cat T$ is relatively tt-stratified if and only if it is relatively h-stratified and the restriction $\pi|_S$ is injective.
\end{Rem}

\begin{Exa}
	The category of $p$-local spectra $\Sp$ has the steel condition by~\cite[Corollary~5.10]{Balmer20_nilpotence}. The only point of $\Spc^h(\Sp^c)\cong\Spc(\Sp^c)$ which is not weakly visible is~$\infty$. Hence $L_{K(S)}\Sp$ is relatively tt-stratified for any finite subset $S \subset \mathbb{N}$ by \cref{prop:K(S)-stratified} and \cref{thm:tt=h+NS}.
\end{Exa}

\section{Smashing localizations}\label{sec:smashing}

Recall that the localized category $\LAT$ is itself rigidly-compactly generated if and only if $\cat T \to \LAT$ is a smashing localization. Moreover, these are precisely the localizations whose Bousfield classes are generated by idempotent rings. In this case, since the category of local objects $\LAT$ is rigidly-compactly generated, we can also consider the ``absolute'' notion of stratification for $\LAT$. As we will show in \cref{cor:relative-absolute} below, this coincides with the relative notion. We will also record some facts about smashing localizations which may be of independent interest. 

\begin{Prop}\label{prop:smash-injective}
	Let $\cat T$ be a rigidly-compactly generated tt-category and let $A \in \cat T$ be an idempotent ring. The corresponding smashing localization $\cat T \to \LAT$ induces an inclusion of sets
	\begin{equation}\label{eq:smashing-embedding}
		\Spc^h((\LAT)^c) \xrightarrow{\sim} \Supph(A) \hookrightarrow \Spc^h(\cat T^c)
	\end{equation}
	on homological spectra.
\end{Prop}

\begin{proof}
	Let $f^*\colon\cat T\to \LAT\eqqcolon \cat S$ denote the localization and consider the induced map $\varphi^h\colon \Spc^h(\cat S^c) \to \Spc^h(\cat T^c)$. The image of this map is $\Supph(f_*(\unit_{\cat S}))$ by \cite[Theorem~5.12]{Balmer20_bigsupport} where $f_*$ denotes the fully faithful right adjoint. Note that $f_*(\unit_{\cat S}) = \LA(\unit_{\cat T}) \simeq A$ since $A$ is idempotent. We conclude that $\im(\varphi^h) = \Supph(A)$. Now consider any $\cat C \in \Spc^h(\cat S^c)$. The pure-injective $\EvarphiC$ is a direct summand of $f_*(\EC)$ by~\cite[Lemma~5.6]{Balmer20_bigsupport}. Hence, $f^*(\EvarphiC)$ is a direct summand of $f^*(f_*(E_{\cat C})) = E_{\cat C}$. It follows that
    \[
		\Supph(f^*(\EvarphiC)) \subseteq \Supph(E_{\cat C}) = \{\cat C\}
	\]
    while \cite[Corollary~7.5]{BarthelHeardSandersZou26} implies that
    \[
		(\varphi^h)^{-1}(\{\varphi^h(\cat C)\}) = (\varphi^h)^{-1}(\Supph(\EvarphiC)) = \Supph(f^*(\EvarphiC)).
	\]
    We conclude that $(\varphi^h)^{-1}(\{\varphi^h(\cat C)\}) \subseteq \{\cat C\}$. This establishes that~$\varphi^h$ is injective.
\end{proof}

\begin{Rem}\label{rem:keller}
	The set-theoretic embedding of~\eqref{eq:smashing-embedding} is not always a topological embedding. For example, let $R$ be the valuation domain denoted $A$ in \cite[Section~2]{Keller94b}. Bazzoni and {\Stovicek} classify the smashing localizations of $\cat T\coloneqq \Der(R)$ in~\cite[Example~5.24]{BazzoniStovicek17}. In particular, the canonical functor $\Der(R) \to \Der(Q\times k)$ is a smashing localization. Here $Q$ denotes the field of fractions and $k$ is the residue field at the unique closed point. This is a nontrivial smashing localization which on spectra maps two disconnected points onto a connected pair of points: $\Spec(Q\times k) \to \Spec(R)$.
\end{Rem}

\begin{Rem}\label{rem:failure-of-telescope}
	A nontrivial smashing localization can even induce a homeomorphism on spectra. Consider the finite localization $\cat T\coloneqq \Sp|_{[0,n]}$ of the category of \mbox{$p$-local} spectra. Localization with respect to $E(n)$ provides a smashing localization $\cat T \to L_{E(n)}\cat T$ which induces a homeomorphism on Balmer spectra; see \cite[Section~10]{bhs1}. Since the steel condition holds in both categories, naturality of the comparison map $\pi$ implies that the induced map on homological spectra is a bijection. However, this smashing localization is not the identity (except for $n=0,1$) due to the failure of the classical Telescope Conjecture~\cite{BurklundHahnLevySchlank23pp}.
\end{Rem}

\begin{Cor}\label{cor:detection-fails-corings}
	The h-detection property does not always hold for weak corings.
\end{Cor}

\begin{proof}
	Let $\cat T \to \cat S$ be a smashing localization which has no nonzero compact acyclic objects and whose induced map $\Spc^h(\cat S^c) \to \Spc^h(\cat T^c)$ is surjective. Let $e \to \unit \to f \to \Sigma e$ be the associated idempotent triangle in~$\cat T$. The image of the induced map is $\Supph(f)$ by~\cref{prop:smash-injective}. Hence $\Supph(f)=\Spc^h(\cat T^c)$ by the surjectivity hypothesis. It then follows from $e \otimes f=0$ and the tensor-product property that $\Supph(e)=\emptyset$. If the h-detection property holds for (idempotent) corings then it would follow that $e=0$, so that the smashing localization is the identity. However \cref{rem:failure-of-telescope} provides an example of a nontrivial smashing localization with the above hypotheses. In other words, the category $\Sp|_{[0,n]}$ does not have h-detection for corings, except for $n = 0,1$.
\end{proof}

\begin{Cor}
	Let $\cat T \to \cat S$ be a smashing localization. If $\cat T$ satisfies the steel condition then $\cat S$ satisfies the steel condition.
\end{Cor}

\begin{proof}
	This follows from the injectivity of~\eqref{eq:smashing-embedding}.
\end{proof}

\begin{Lem}\label{lem:smashing-EB}
	If $f^*\colon\cat T\to\cat S$ is a smashing localization then
	\begin{equation}\label{eq:smashing-EB}
		\EvarphiC \simeq f_*(\EC)
	\end{equation}
	for any $\cat C \in \Spc^h(\cat S^c)$.
\end{Lem}

\begin{proof}
	Let $\cat B \coloneqq \varphi^h(\cat C)$ and recall the diagram~\eqref{eq:paul-diagram} copied from \cite[(5.7)]{Balmer20_bigsupport}. Following that notation, we set $F\coloneqq f^*$ and write $U\coloneqq f_*$ for its fully faithful right adjoint. It follows from \cite[Theorem~4.4, Theorem~11.1 and Proposition~A.5]{Krause05} that $\hat{U}$ is fully faithful. Hence $\bar{U}$ is fully faithful, since $\hat{U}\circ R_{\cat C} \simeq R_{\cat B}\circ \bar{U}$ and both~$R_{\cat C}$ and~$R_{\cat B}$ are fully faithful. Consider the injective hull $\bar{\unit}\hookrightarrow \ECbar$ in $\ACbar$. Since $\bar{F}$ is exact, the object $\bar{U}(\ECbar)$ is still injective. We claim that the composite 
	\begin{equation}\label{eq:essential-extension}
		\bar{\unit} \to \bar{U}(\bar{\unit}) \to \bar{U}(\ECbar)
	\end{equation}
	is an essential extension, so that $\bar{U}(\ECbar)$ is an injective hull of $\bar{\unit}$ in $\ABbar$. The map~\eqref{eq:essential-extension} is a monomorphism by \cite[Proposition~3.5]{Balmer20_bigsupport}. Now consider any morphism $\alpha\colon \bar{U}(\ECbar) \to X$ in $\ABbar$ such that the composite
	\[ 
		\bar{\unit} \to \bar{U}(\bar{\unit}) \to \bar{U}(\ECbar) \xrightarrow{\alpha} X
	\]
	is a monomorphism. Since $\bar{F}$ is exact, the composite
	\[\begin{tikzcd}[row sep=small]
		\bar{F}(\bar{\unit}) \ar[r] & \bar{F}\bar{U}(\bar{\unit}) \ar[r] & \bar{F}\bar{U}(\ECbar) \ar[r,"\bar{F}(\alpha)"] & \bar{F}(X) \\
		\bar{\unit} \ar[u,"\simeq\;"]\ar[r,"="] & \bar{\unit} \ar[r] & \ECbar\ar[u,"\simeq\;","\epsilon^{-1}"']
	\end{tikzcd}\]
	is a monomorphism. Since the bottom morphism is an essential extension, we conclude that~$\bar{F}(\alpha)$ is a monomorphism. It follows that $\alpha$ is a monomorphism since the bottom arrow in 
	\[\begin{tikzcd}
		\bar{U}(\ECbar) \ar[r,"\alpha"]\ar[d,"\simeq"',"\eta"] & X\ar[d,"\eta"]\\
			\bar{U}\bar{F}\bar{U}(\ECbar) \ar[r,hook,"\bar{U}\bar{F}(\alpha)"] & \bar{U}\bar{F}(X)
	\end{tikzcd}\]
	is a monomorphism. This establishes that \eqref{eq:essential-extension} is an essential extension and hence is an injective hull of $\bar{\unit}$ in $\ABbar$. We conclude that $\bar{U}(\ECbar) \simeq \EBbar$ in $\ABbar$. It follows that
	\[
		\hat{E}_{\cat B} \simeq R_{\cat B}(\EBbar) \simeq R_{\cat B}\bar{U}(\ECbar) \simeq \hat{U}R_{\cat C}(\ECbar) \simeq \hat{U}(\hat{E}_{\cat C})
	\]
	and hence 
	\[
		h_{\cat T}(\EB) \simeq \hat{E}_{\cat B} \simeq \hat{U}(\hat{E}_{\cat C}) \simeq \hat{U}h_{\cat S}(\EC) \simeq h_{\cat T}U(\EC).
	\]
	We conclude that $\EB \simeq U(\EC)$, as desired.
\end{proof}

\begin{Prop}\label{prop:relative-nest}
	Let $\cat T\to \LAT\eqqcolon\cat S$ be a smashing localization and let $B\in \cat T$ be an $A$-local object. Then $\LBT \simeq \LBS$ and the following statements hold:
	\begin{enumerate}
		\item $\LBT$ has the relative h-detection property with respect to $\cat T$
			if and only if it has the relative h-detection property with respect to $\cat S$.
		\item $\LBT$ has the relative h-codetection property with respect to $\cat T$ if and only if
			it has the relative h-codetection property with respect to $\cat S$.
		\item $\LBT$ is relatively h-stratified with respect to $\cat T$ if and only if
			it is relatively h-stratified with respect to $\cat S$.
	\end{enumerate}
\end{Prop}

\begin{proof}
	We use the same notation as in the proof of \cref{prop:smash-injective}. The localization of the unit $\LA(\unit_{\cat T})=f_*(\unit_{\cat S})$ is an idempotent ring in $\cat T$ and $\bc{A}=\bc{\LA(\unit_{\cat T})}$. Replacing $A$ by $\LA(\unit_{\cat T})$ we may assume without loss of generality that $A=\LA(\unit_{\cat T})$. \Cref{prop:smash-injective} then provides $\im(\varphi^h) = \Supph(A)$. Moreover, by \cref{lem:smashing-EB} we have $\EvarphiC \simeq f_*(\EC)$ for any $\cat C \in \Spc^h(\cat S^c)$. Next observe that for any $s \in \cat S$, 
	\begin{align*}
		\cat C \in \Supph(s) \iff \ihom{s,\EC} \neq 0 \iff& \ihom{f_*(s),f_*(\EC)} \neq 0\\
		\iff& \ihom{f_*(s),\EvarphiC} \neq 0\\
		\iff& \varphiC \in \Supph(f_*(s))
	\end{align*}
	where we have used that $f_*$ preserves the internal hom. In other words, $\Supph(s) = (\varphi^h)^{-1}(\Supph(f_*(s)))$. It follows that
	\begin{equation}\label{eq:img-ff}
		\varphi^h(\Supph(s)) = \Supph(f_*(s)) \cap \im(\varphi^h) = \Supph_A(f_*(s))
	\end{equation}
	bearing in mind that $f_*(s) \otimes A \simeq f_*(s)$.

	We now consider the $A$-local object $B$. Note that $\bc{A} = \GammaA\unit_{\cat T} \otimes \cat T$. Hence $\GammaA\unit_{\cat T} \otimes B=0$ implies that every $A$-acyclic object is $B$-acyclic. Consequently, every $B$-local object is $A$-local. We have a nesting of localizations:
	\[\begin{tikzcd}
		\cat T\ar[d,"\LA"'] \ar[dr,bend left,"\LB"] & \\
		\cat S \ar[r,"\LB"'] &\LBS \simeq \LBT \eqqcolon \cat R.
	\end{tikzcd}\]
	We will be pedantic and regard $B$ as an object of $\cat T$ and write $f^*(B)$ for $B$ regarded as an object of $\cat S$. Consider any $s \in \cat S$. Note that $\Supph_{\cat T}(f_*(s))\subseteq \Supph(A)$ by~\eqref{eq:supph-local} since $f_*(s)$ is $A$-local. Observe that
	\begin{equation}
	\begin{aligned}\label{eq:supp-var-smash}
	\varphi^h(\Supph_{\cat S,B}(s)) &= \varphi^h(\Supph_{\cat S}(s) \cap \Supph_{\cat S}(f^*(B)))\\
									&= \varphi^h(\Supph_{\cat S}(s))\cap\varphi^h(\Supph_{\cat S}(f^*(B)))\\
									&= \Supph_{\cat T}(f_*(s)) \cap \Supph_{\cat T}(f_*f^*(B)) \cap \Supph_{\cat T}(A)\\
									&= \Supph_{\cat T}(f_*(s)) \cap \Supph_{\cat T}(B)\\
									&= \Supph_{\cat T\!,B}(f_*(s)).
	\end{aligned}
	\end{equation}
	Here we have used~\eqref{eq:img-ff} and the fact that $\varphi^h$ is injective (\cref{prop:smash-injective}). Part~$(a)$ is immediate from~\eqref{eq:supp-var-smash}.

	Next note that $s\simeq f^!f_*(s)$ since $f_*$ is fully faithful. Applying \cref{prop:base-change}(c) to $t\coloneqq f_*(s)$ we obtain $\Cosupp_{\cat S}^h(s) = (\varphi^h)^{-1}(\Cosupp_{\cat T}^h(f_*(s)))$. Hence
	\begin{equation}\label{eq:cosupp-var-smash}
		\varphi^h(\Cosupp_{\cat S}^h(s)) = \Cosupp_{\cat T}^h(f_*(s)) \cap \im(\varphi^h)=\Cosupp_{\cat T}^h(f_*(s))
	\end{equation}
	using $\Cosupp_{\cat T}^h(f_*(s)) \subseteq \Supphnaive_{\cat T}(A)$ by \cref{rem:cosupp-in-supp}, together with $\Supphnaive_{\cat T}(A)=\Supph_{\cat T}(A)=\im(\varphi^h)$ for the last equality. Part~$(b)$ is immediate from \eqref{eq:cosupp-var-smash}.

	Finally, \eqref{eq:supp-var-smash} implies that the diagram
	\[
	\begin{tikzcd}[column sep=large]
		\{\text{localizing ideals of $\cat R$}\} \ar[r,"\Supph_{\cat T\!,B}"] \ar[dr,bend right=15,"\Supph_{\cat S\!,B}"'] & \{\text{subsets of $\Supph_{\cat T}(B)$}\} \\
		&\{\text{subsets of $\Supph_{\cat S}(f^*(B))$}\}\ar[u,"\varphi^h"',"\simeq"]
	\end{tikzcd}
	\]
	commutes. This establishes $(c)$ bearing in mind~\cref{thm:hstratfundamental}.
\end{proof}

\begin{Cor}\label{cor:relative-absolute}
	Let $\cat T\to \LAT$ be a smashing localization. Then:
	\begin{enumerate}
		\item $\LAT$ has the relative h-detection property if and only if it has the absolute h-detection property.
		\item $\LAT$ has the relative h-codetection property if and only if it has the absolute h-codetection property.
		\item $\LAT$ is relatively h-stratified if and only if it is absolutely h-stratified.
	\end{enumerate}
\end{Cor}

\begin{proof}
	Without loss of generality, we may assume that $A=\LA \unit_{\cat T}$, so that $A$ is itself $A$-local. The result is then provided by the $B=A$ case of~\cref{prop:relative-nest}.
\end{proof}

\begin{Exa}
	If $A$ is a $p$-local spectrum, then the $A$-relative notions do not depend on whether we take the ambient category to be spectra or $p$-local spectra.
\end{Exa}

\begin{Prop}\label{prop:h-strat-smashing}
	Let $e \to \unit \to f \to \Sigma e$ be an idempotent triangle in $\cat T$. The following statements hold:
	\begin{enumerate}
		\item $\cat T$ has the h-detection property if and only if $\LfT$ has the (relative) h-detection property and $\LeT$ has the relative h-detection property.
		\item $\cat T$ has the h-codetection property if and only if $\LfT$ has the (relative) \mbox{h-codetection} property and $\LeT$ has the relative h-codetection property.
		\item $\cat T$ is h-stratified if and only if $\LfT$ is (relatively) h-stratified and $\LeT$ is relatively h-stratified.
	\end{enumerate}
\end{Prop}

\begin{proof}
	Recall that the localization $\cat T \to \LfT$ is smashing and given by
	\[
		f\otimes -\colon\cat T \to f\otimes \cat T\cong\LfT
	\]
	while the localization $\cat T \to \LeT$ is given by 
	\[
		\ihom{e,-}\colon\cat T \to \ihom{e,\cat T}\cong\LeT.
	\]
	Moreover, $e\otimes\ihom{e,t} \simeq e\otimes t$ and $\ihom{e,t} \simeq \ihom{e,e\otimes t}$ for any $t\in \cat T$. See \cref{exa:localization-completion} or \cite[Section~3]{BalmerSanders25} if necessary. Note that $\LfT$ has the relative properties if and only if it has the absolute properties by~\cref{cor:relative-absolute}.

	The ``only if'' parts of $(a)$ and $(b)$ are immediate from the definitions, while the ``only if'' part of $(c)$ is provided by~\cref{thm:local-stratification}. For the ``if'' parts, first observe that we have a disjoint union $\Spc^h(\cat T^c) = \Supph(e) \sqcup \Supph(f)$ since $e\otimes f=0$. Hence
	\[
		\Supph(t) = \Supph(e\otimes t) \sqcup \Supph(f\otimes t) = \Supph_e(t) \sqcup \Supph_f(t)
	\]
	for any $t \in \cat T$. Thus, if $\Supph(t)=\emptyset$ then $\Supph_e(t)=\emptyset$ and $\Supph_f(t)=\emptyset$. The relative h-detection hypotheses would imply that $\ihom{e,t}=L_e(t)=0$ and $f\otimes t=L_f(t)=0$. Moreover, $e\otimes t=e\otimes\ihom{e,t}=0$, as well. Hence $t=0$. This proves $(a)$. Next note that 
	\begin{align*}
		\Cosupph(\ihom{e,t}) &\subseteq \Supph(e) \cap \Cosupph(t), \text{ and}\\
		\Cosupph(\ihom{f,t}) &\subseteq \Supph(f) \cap \Cosupph(t)
	\end{align*}
	bearing in mind~\cref{rem:inclusion} Hence, if $\Cosupph(t)=\emptyset$ then the relative h-codetection hypotheses would imply that $\ihom{e,t}=0$ and $\ihom{f,t} =0$. (Note that $\ihom{f,t}\simeq f\otimes\ihom{f,t}$ is $f$-local.) Hence $t \simeq \ihom{\unit,t}=0$. This proves $(b)$.

	Now suppose that $\LeT$ and $\LfT$ are h-stratified. We claim that $\cat T$ is h-stratified. Given $(a)$ and $(b)$, it suffices to prove that $\Loco{\EB}$ is a minimal localizing ideal of $\cat T$ for any $\cat B \in \Spc^h(\cat T^c)$. To this end, suppose $0 \neq t \in \Loco{\EB}$. There are two mutually exclusive possibilities: either $\cat B \in \Supph(e)$ or $\cat B \in \Supph(f)$.

	If $\cat B\in\Supph(f)$, then $\cat B\notin\Supph(e)$ since $e\otimes f=0$. Thus $\Supph(e\otimes \EB)=\emptyset$, and by part (a) we have $e\otimes \EB=0$. Hence $e \otimes t=0$ so that $t \simeq t\otimes f$. It follows $0 \neq f \otimes t \in \Loco{\EB}$ in the tt-category $f\otimes \cat T=\LfT$. Since this category is h-stratified by hypothesis, we conclude that $\EB \in \Loco{f \otimes t}$ in $f \otimes \cat T$ from which it follows that $\EB \in \Loco{f\otimes t}=\Loco{t}$ in $\cat T$.

	On the other hand, if $\cat B \in \Supph(e)$, then $\cat B\notin\Supph(f)$ since $e\otimes f=0$. Thus $\Supph(f\otimes\EB)=\emptyset$, and by part (a) we have $f\otimes\EB=0$. Hence $f\otimes t=0$ so that $t \simeq e\otimes t$. Also, $\EB \simeq e\otimes \EB$. We then have $0\neq e\otimes t \in \Loco{\EB}$ in $\LeT$. Since $\LeT$ is relatively h-stratified by hypothesis, minimality of $\Loco{\EB}$ in $\LeT$ gives $\EB\in\Loco{e\otimes t}$. Viewing this localizing ideal inside $\cat T$, it follows that $\EB\in\Loco{e\otimes t}=\Loco{t}$.
\end{proof}

\begin{Cor}
	Let $Y \subseteq \Spc(\cat T^c)$ be a Thomason subset. Then:
	\begin{enumerate}
		\item $\cat T$ has the h-detection property if and only if $\cat T|_{Y^c}$ has the (relative) h-detection property
			and $\cat T_Y^{\wedge}$ has the relative h-detection property.
		\item $\cat T$ has the h-codetection property if and only if $\cat T|_{Y^c}$ has the (relative) \mbox{h-codetection} property
			and $\cat T_Y^{\wedge}$ has the relative h-codetection property.
		\item $\cat T$ is h-stratified if and only if $\cat T|_{Y^c}$ is (relatively) h-stratified
			and $\cat T_Y^{\wedge}$ is relatively h-stratified.
	\end{enumerate}
\end{Cor}

\begin{proof}
	This is \cref{prop:h-strat-smashing} applied to the finite localizations of~\cref{exa:localization-completion}.
\end{proof}

\section{Cosupport and cohomological Bousfield classes}

We now relate h-stratification to the classification of \emph{cohomological} Bousfield classes. The following is the cohomological analog of \cref{def:hbc-map}.

\begin{Def}
	Let $A\in\cat T$ and define a map
	\begin{equation}\label{eq:cbc-map}
		\big\{\text{cohomological Bousfield classes of } \LA\cat T\big\}
		\longrightarrow
		\big\{\text{subsets of }\Supphnaive(A)\big\}
	\end{equation}
	by sending $\cbc{x}$ to $\Cosupp^h(x)$.
\end{Def}

\begin{Lem}\label{lem:cbc-defined-surjective}
	The map \eqref{eq:cbc-map} is a well-defined surjective map. It is order-preserving when cohomological Bousfield classes are ordered by reverse inclusion.
\end{Lem}

\begin{proof}
	First recall that $\Cosupph(x) \subseteq \Supphnaive(A)$ for any $A$-local~$x$ by \cref{rem:cosupp-in-supp}. If $\cbc{x}=\cbc{y}$ in $\LA\cat T$ then for all $z\in \LA\cat T$,
	\[
	  \ihom{z,x}=0 \iff \ihom{z,y}=0.
	\]
	In particular, we can take $z=\LA\EB$ for any $\cat B\in \Spc^h(\cat T^c)$. Since $x$ and $y$ are \mbox{$A$-local,} we have that $\ihom{\EB,x}\simeq \ihom{\LA \EB,x}$ and $\ihom{\EB,y}\simeq\ihom{\LA\EB,y}$. Hence $\Cosupp^h(x)=\Cosupp^h(y)$.
	Hence the map~\eqref{eq:cbc-map} is well-defined.
	That it is order-preserving is immediate from the definition.
	Given any subset $S \subseteq \Supphnaive(A)$, set $x \coloneqq \LA(\coprod_{\cat B \in S} \EB) \in \LAT$.
	We have $\Cosupp^h(I_A(x)) = \Supphnaive_A(x) = S$
	using~\cref{lem:cosupp-IAx} and~\cref{lem:Supph-EB}.
	Thus $\cbc{I_A(x)}$ maps to $S$ under the map~\eqref{eq:cbc-map}.
\end{proof}

\begin{Lem}\label{lem:injective-supp}
	If $\Supph(A) = \Supphnaive(A)$ and the map~\eqref{eq:cbc-map} is injective then $\Supph_A = \Supphnaive_A$.
\end{Lem}

\begin{proof}
	For any $\cat B \in \Supphnaive(A)$, we have $\bc{\LA\EB} = \cbc{I_A\EB}$ by 
	\cref{rem:bc-vanishing} and \cref{thm:injection-hbc-cbc}. 
	Moreover, since \[\Cosupp^h(I_A\EB)=\Supphnaive_A(\EB) = \{\cat B\}=\Cosupph(\LA\EB)\] by
  \cref{lem:cosupp-EB} and 
		\cref{lem:cosupp-IAx}, our hypothesis implies that $\cbc{I_A\EB} = \cbc{\LA\EB}$.
		We conclude that 
		\begin{equation}\label{eq:bc-is-cbc-EB}
			\bc{\LA\EB} = \cbc{\LA\EB}
		\end{equation}
		in $\LAT$ for all $\cat B \in \Supphnaive(A)$.
		Now let $t \in \cat T$ be any object and suppose that $\cat B \in \Supphnaive_A(t)=\Supphnaive_A(\LA t)$.
		By definition, this means $\EB \otimes A \otimes \LA t \neq 0$
		which is equivalent to 
		$\LA\EB \otimes \LA t \neq 0$ in $\LAT$.
		It follows that 
		$\ihom{\LA t,\LA\EB} \neq 0$
		using~\eqref{eq:bc-is-cbc-EB}.
		By hypothesis, $\cat B \in \Supphnaive(A) = \Supph(A)$
		and hence $\LA\EB = \EB$ by \cref{lem:completion-of-eb}.
		Therefore
		$\ihom{t,\EB}=\ihom{t,\LA\EB} = \ihom{\LA t,\LA\EB} \neq 0$.
		It follows that $\cat B \in \Supph(t) \cap \Supph(A) = \Supph_A(t)$
		as desired.
\end{proof}

\begin{Rem}
	The following surprising result is new even in the absolute ($A=\unit$) case. It shows that the classification of cohomological Bousfield classes via homological cosupport is actually equivalent to the classification of \emph{all} localizing ideals via homological support.
	Moreover, this is equivalent to the classification of homological Bousfield classes via homological support together with the statement that every cohomological Bousfield class is homological.
\end{Rem}

\begin{Thm}\label{thm:cbc-equivalence}
	Let $\cat T$ be a rigidly-compactly generated tt-category and let $A\in\cat T$.
	The following are equivalent:
    \begin{enumerate}
        \item $\LAT$ is relatively h-stratified. 
        \item $\Supph(A) = \Supphnaive(A)$ and the map \eqref{eq:cbc-map} is a bijection. 
		\item $\LAT$ has relative h-detection and every cohomological Bousfield class of $\LAT$ is homological.
    \end{enumerate}
\end{Thm}

\begin{proof}
	$(a) \Rightarrow (b)$: Since relative h-stratification implies relative h-detection, we have 
	$\Supph(A)=\Supphnaive(A)$ by \cref{lem:local-h-detection}.
	The map \eqref{eq:cbc-map} is surjective by~\cref{lem:cbc-defined-surjective}, so it suffices to establish injectivity. Since $\LAT$ is relatively \mbox{h-stratified,} \cref{thm:equivalent-to-stratification} implies that
	\begin{align}\label{eq:hom-detection-bijection}
		\ihom{x,y}= 0 &\iff \Supph_A(x) \cap \Cosupph(y) = \emptyset
	\intertext{for any $x,y\in \LAT$. Thus, if $y_1,y_2 \in \LAT$ satisfy $\Cosupph(y_1) = \Cosupph(y_2)$ then}
		\ihom{x,y_1}=0 &\iff \ihom{x,y_2}=0\nonumber
	\end{align}
	for any $x \in \LAT$. Hence $\cbc{y_1}=\cbc{y_2}$ in $\LAT$.

	$(b) \Rightarrow (c)$:
	By~\cref{lem:injective-supp},
	we have $\Supph_A = \Supphnaive_A$.
	With this in hand, we establish the relative h-detection property.
	Suppose $x \in \LAT$ has $\Supph_A(x) = \emptyset$.
	Then $\Cosupp^h(I_A(x)) = \Supphnaive_A(x) = \Supph_A(x) = \emptyset$
	by~\cref{lem:cosupp-IAx}.
	Since~\eqref{eq:cbc-map} is injective, the cohomological Bousfield class of $I_A(x)$
	coincides with the cohomological Bousfield class of $0$.
	Hence $I_A(x) = 0$ so that $x=0$ by \cref{rem:bc-vanishing}.
	This establishes relative h-detection.
	Now consider any $y \in \LAT$. Let $z \coloneqq \LA(\coprod_{\cat B \in S} \EB)$ for $S \coloneqq \Cosupp^h(y)$.
	Then 
	\[ \bc{z} = \cbc{I_A(z)} = \cbc{y}\]
	where the first 
	equality is \cref{thm:injection-hbc-cbc}
	and the second equality is from the hypothesis that~\eqref{eq:cbc-map} is injective, since
	\[ \Cosupph(I_A(z)) = \Supphnaive_A(z) = S = \Cosupp^h(y)\]
	using~\cref{lem:cosupp-IAx}
	and \cref{lem:Supph-EB}.

	$(c) \Rightarrow (a)$:
	By \cref{thm:equivalent-to-stratification}, it suffices to show that \eqref{eq:hom-detection-bijection} holds for any $x,y \in \LAT$.
	By hypothesis $\cbc{y} = \bc{z}$ in $\LAT$ for some $z \in \LAT$.
	Recall from
\cref{rem:cosupp-in-supp},
\cref{lem:local-h-detection}
and \cref{lem:completion-of-eb}
	that for any $\cat B \in \Cosupp^h(y)$
	we have $\cat B \in \Supphnaive(A) = \Supph(A)$
	so that $\EB \in \LAT$.
	It follows that
	\begin{align*}
		\Cosupp^h(y) =& \SETT{\cat B \in \Supph(A)}{\EB \not\in \cbc{y}} \\
					 =& \SETT{\cat B \in \Supph(A)}{\EB \not\in \bc{z}} = \Supphnaive_A(z).
	\end{align*}
	Therefore, for any $x \in \LAT$, we have
	\begin{align*}
		\ihom{x,y}=0\Longleftrightarrow x\in \cbc{y}	&\Longleftrightarrow x \in \bc{z} \\
						&\Longleftrightarrow \LA(x \otimes z) = 0 \\
						&\Longleftrightarrow \Supph_A(x \otimes z) =\emptyset \\
						&\Longleftrightarrow \Supph_A(x) \cap \Supph_A(z) = \emptyset \\
						&\Longleftrightarrow \Supph_A(x) \cap \Cosupp^h(y) = \emptyset
	\end{align*}
	as required. Note that we have used the relative h-detection property in the equivalence between the second and third line.
\end{proof}

\begin{Def}\label{def:relative-tt-fields}
	We say that $\LAT$ \emph{admits enough relative tt-fields} (with respect to~$\cat T$) if for every $\cat B\in\Supph(A)$ there is a geometric functor $f_{\cat B}^*\colon\cat T\to\cat F_{\cat B}$ to a tt-field whose unique homological prime maps to $\cat B$ and such that $f_{\cat B}^*(A) \neq 0$.
\end{Def}
\begin{Rem}\label{rem:factorization}
The latter condition ensures that $f_{\cat B}^*$ annihilates the $A$-acyclic objects and hence factors through a coproduct-preserving tensor functor $\LAT \to \cat F_{\cat B}$. 
If $\Supph(A)=\Supphnaive(A)$ then the condition $f_{\cat B}^*(A)\neq0$ is automatic for every \mbox{tt-field} realizing $\cat B\in\Supph(A)$, since otherwise the projection formula and the split monomorphism $\EB\to(f_{\cat B})_*(\unit)$ of \cref{prop:base-change} would imply $A\otimes\EB=0$.
\end{Rem}
\begin{Cor}\label{cor:relative-tt-fields-codetection}
	Let $\cat T$ be a rigidly-compactly generated tt-category and let $A\in\cat T$. Assume that $\Supph(A)=\Supphnaive(A)$ and that $\LAT$ admits enough relative tt-fields. Then the following are equivalent:
	\begin{enumerate}
		\item $\LAT$ is relatively h-stratified.
		\item $\LAT$ has the relative h-codetection property.
		\item The map \eqref{eq:cbc-map} is a bijection.
	\end{enumerate}
\end{Cor}
\begin{proof}
	The equivalence of $(a)$ and $(c)$ is \cref{thm:cbc-equivalence}, so it remains to show the equivalence of $(a)$ and $(b)$. The implication $(a)\Rightarrow(b)$ follows from \cref{prop:hLGP-hcodetect}.

For the converse, suppose that relative h-codetection holds. We first claim that $\Supph_A = \Supphnaive_A$. Indeed, if $\cat B \in \Supphnaive_A(x)$, then $A\otimes x\otimes \EB \neq0$, and hence $\cat B \in \Supphnaive(A) \cap \Supphnaive(x)$. The equality $\Supphnaive(A) = \Supph(A)$ ensures that there is a relative tt-field realizing $\cat B$. By \cite[Lemma 2.15]{BarthelHeardSandersZou26}, this tt-field gives
	\[
		\cat B\in\Supphnaive(x) \iff \cat B\in\Supph(x).
	\]
	Consequently, $\cat B\in\Supph_A(x)$, proving the reverse inclusion $\Supphnaive_A(x)\subseteq\Supph_A(x)$. Hence the relative h-local-to-global principle holds by \cref{prop:hLGP-hcodetect}, and in particular $\LAT$ has relative h-detection. It remains to verify the minimality condition in \cref{thm:hstratfundamental}.

    Fix $\cat B \in \Supph(A)$ and write $f^* \colon \cat T \to \cat F$ for a relative tt-field realizing $\cat B$. By \Cref{rem:factorization}, the functor $f^*$ factors through a coproduct-preserving tensor functor $\bar f^* \colon \LAT \to \cat F$, whose right adjoint is the restriction of $f_*$. 	Let $0 \neq x\in\Loco{\EB}\subseteq\LAT$. If $\bar f^*(x)=0$, then the projection formula gives $f_*(\unit)\otimes x=0$. Since $\EB$ is a direct summand of $f_*(\unit)$, we have $x\in\Loco{f_*(\unit)}$. The kernel of $-\otimes x$ is therefore a localizing ideal containing $x$, and hence $x\otimes x=0$. The tensor-product formula and relative h-detection would then imply $x = 0$, a contradiction. Thus $\bar f^*(x) \neq 0$. Since $\cat F$ is a tt-field, $\unit\in\Loco{\bar f^*(x)}$ by \cite[Example 4.4]{BarthelHeardSandersZou26}. It then follows from \cite[(13.4)]{BarthelCastellanaHeardSanders24} that
	\[
		f_*(\unit)\in f_*\Loco{\bar f^*(x)}
		\subseteq\Loco{x} 
	\]
    in $\cat T$. 
    Applying $\LA$ to this, and using that both $f_*(\unit)$ and $x$ are $A$-local, gives $f_*(\unit)\in\Loco{x}$ in $\LAT$. Consequently, $\EB\in\Loco{x}$. This proves that $\Loco{\EB}$ is minimal for every $\cat B\in\Supph(A)$.
    Hence $\LAT$ is relatively h-stratified by \cref{thm:hstratfundamental}.
\end{proof}
    
\begin{Thm}\label{thm:dictionary}
	Let $\cat T$ be a rigidly-compactly generated tt-category and let $A\in\cat T$. Assume that $\LAT$ is relatively h-stratified. For any subset $S\subseteq \Supph(A)$ define
	\[
		J_S \coloneqq \prod_{\cat B\in S} \EB\ \in \LA\cat T
		\qquad\text{and}\qquad
		K_S \coloneqq \LA\Bigl(\coprod_{\cat B\in S} \EB\Bigr)\ \in \LA\cat T.
	\]
	In the following diagram, all maps are bijections, and both triangles commute:
	\[\begin{tikzcd}[column sep = small]
		\Biggl\{
		\begin{tabular}{c}
		Homological \\
		Bousfield classes \\
		of $\LAT$
		\end{tabular}
		\Biggr\} && \Biggl\{\begin{tabular}{c} Cohomological\\ Bousfield classes\\ of $\LA\cat T$ \end{tabular}\Biggr\}\\
				  & \Biggl\{\begin{tabular}{c} Localizing ideals\\ of $\LA\cat T$ \end{tabular}\Biggr\} 
				  \arrow["F", shift left, from=1-1, to=1-3]
				  \arrow["\Phi_H", shift left, from=1-1, to=2-2]
				  \arrow["G", shift left, from=1-3, to=1-1]
				  \arrow["\Phi_C"', shift right, from=1-3, to=2-2]
				  \arrow["\Psi_H", shift left, from=2-2, to=1-1]
				  \arrow["\Psi_C"', shift right, from=2-2, to=1-3]
	\end{tikzcd}\]
	The maps are given by
	\[
		F\colon  \bc{x} \longmapsto \cbc{J_{\Supph_A(x)}},
		\qquad
		G\colon  \cbc{x} \longmapsto  \bc{K_{\Cosupph(x)}},
	\]
	\[
	  \Phi_H \colon  \bc{x} \longmapsto \Loco{E_{\cat B}\mid \cat B\in \Supph_A(x)},
	  \qquad
	  \Psi_H \colon  \cat L \longmapsto \bc{K_{\Supph_A(\cat L)}},
	\]
	\[
	  \Phi_C \colon  \cbc{x} \longmapsto \Loco{E_{\cat B}\mid \cat B\in \Cosupph(x)},
	  \qquad
	  \Psi_C \colon  \cat L \longmapsto \cbc{ J_{\Supph_A(\cat L)}}.
	\]
	Moreover, we have
	\[
		\Psi_C\circ \Phi_H = F,\qquad
		\Psi_H\circ \Phi_C =G,\qquad
		\Phi_C\circ F = \Phi_H,\qquad
		\Phi_H\circ G = \Phi_C.
	\]
\end{Thm}

\begin{proof}[Sketch of proof]
	It is routine but tedious to show that each map is well-defined and that each pair of maps are mutually inverse to each other. For example,
	\[
		FG(\cbc{x})
		=F(\bc{K_{\Cosupph(x)}})
		=\cbc{J_{\Supph_A(K_{\Cosupph(x)})}}
	\]
	and since $\Supph_A(K_{\Cosupph(x)})=\Cosupph(x)$ and the cohomological Bousfield classes are classified by cosupport, $FG(\cbc{x})=\cbc{x}$. Similarly,
	\[
		GF(\bc{x})
		=G(\cbc{J_{\Supph_A(x)}})
		=\bcnothuge{K_{\Cosupph(J_{\Supph_A(x)})}}
	\]
	and since $\Cosupph(J_{\Supph_A(x)}) = \Supph_A(x)$ and the homological Bousfield classes are classified by support, $GF(\bc{x}) = \bc{x}$. The other bijections can be shown in a similar manner. 

	The commutativity is also a direct but tedious computation. For example,
	\[
		(\Psi_C\circ \Phi_H)(\bc{x})
		=\Psi_C\bigl(\Loco{E_{\cat B}\mid \cat B\in \Supph_A(x)}\bigr)
		=\cbc{J_{\Supph_A(x)}}
		=F(\bc{x}).
	\]
	Similarly, 
	\[
		(\Psi_H\circ \Phi_C)(\cbc{x})
		=\Psi_H\bigl(\Loco{E_{\cat B}\mid \cat B\in \Cosupph(x)}\bigr)
		=\bc{K_{\Cosupph(x)}}
		=G(\cbc{x}).
	\]
	We leave the complete details for the interested reader.
\end{proof}

\begin{Rem}
	In \cref{thm:injection-hbc-cbc} we constructed a map $\bc{x} \mapsto \cbc{I_A(x)}$ from the homological Bousfield classes of $\LAT$ to its cohomological Bousfield classes. Under the hypotheses of \cref{thm:dictionary}, this agrees with the map $F$. To see this, it suffices to show that $J_{\Supph_A(x)}$ and $I_A(x)$ have the same cosupport. Using $I_A(x)=\ihom{x,I_A}$ and \cref{thm:equivalent-to-stratification} we obtain
	\[
		\Cosupph(I_A(x))=\Supph_A(x)\cap \Cosupph(I_A).
	\]
	Moreover, a minor modification of \cite[Proposition~12.9]{barthel2023cosupport} shows that 
	\[
		\Cosupph(I_A)=\Supphnaive(A)=\Supph(A).
	\]
	It follows that $\Cosupph(I_A(x))=\Supph_A(x)=\Cosupph(J_{\Supph_A(x)})$, as required.
\end{Rem}

\section{Oddball cohomological Bousfield classes}\label{sec:oddball_cohomological_bousfield_classes}

We now produce examples of categories which have cohomological Bousfield classes that are not homological Bousfield classes. In particular, these give examples of localizing ideals that are not homological Bousfield classes. Our general method is as follows.

\begin{Prop}\label{prop:empty-cosupp-gives-nonHBC}
	Let $\cat T$ be a rigidly compactly generated tt-category and let $A\in \cat T$. Assume relative h-detection holds for $\LAT$. If there exists a nonzero $x\in \LAT$ with $\Cosupp^h(x)=\emptyset$ then the cohomological Bousfield class $\cbc{x}$ in $\LAT$ is not a homological Bousfield class. 
\end{Prop}
\begin{proof}
	Since $\Cosupph(x)=\emptyset$, we have $\ihom{\EB,x}=0$ for all $\cat B\in \Spc^h(\cat T^c)$. Since the internal hom in $\LAT$ is computed in $\cat T$, it follows that $\EB \in \cbc{x}$ in $\LAT$ for all $ \cat B$ such that $\EB$ is $A$-local. This includes all $\cat B \in \Supph(A)$ by \cref{lem:completion-of-eb}. Suppose for contradiction that~$\cbc{x}$ is homological, i.e., $\cbc{x}=\bc{y}$ for some $y \in \LAT$. Then $\EB\in \bc{y}$ for all $\cat B\in\Supph(A)$, i.e., $\LA(\EB \otimes y) = 0$. This means that $A \otimes \EB \otimes y = 0$ so that $\Supph(A \otimes \EB \otimes y) = \emptyset$, and hence $\cat B \not\in \Supph_A(y)$. This establishes that $\Supph_A(y)=\emptyset$ so $y=0$ by relative h-detection. Thus $\bc{y}=\LAT$ which contradicts the fact that $x\notin \cbc{x}$ since $\ihom{x,x}\neq 0$. We conclude that $\cbc{x}$ is not homological.
\end{proof}

\begin{Rem}
    \Cref{prop:empty-cosupp-gives-nonHBC} admits the following alternative proof. Taking the contrapositive, the statement is that relative h-detection, together with the condition that every cohomological Bousfield class is homological, implies relative h-codetection. By \cref{thm:cbc-equivalence}, these two hypotheses are equivalent to relative h-stratification, which implies relative h-codetection by \Cref{prop:hLGP-hcodetect}. 
\end{Rem}
\begin{Rem}
	We now return to the category of $p$-local spectra. The following result provides a convenient way to produce spectra with empty homological cosupport.
\end{Rem}

\begin{Lem}\label{lem:cosupport-comp}
	Let $X$ be a bounded below $p$-local spectrum such that each $\pi_iX$ is bounded $p$-torsion. Then $\ihom{K(i),X}= 0$ for all $0\le i<\infty$ and hence 
	\[
		\Cosupp^h(X)\subseteq\{\cat C_{\infty}\}.
	\]
	If $X$ is also harmonic (i.e., $K(\bbN)$-local) then $\Cosupp^h(X)=\emptyset$.
\end{Lem}

\begin{proof}
	Since $X$ is bounded below we have $L_{\HFp}X \simeq L_{\bbS/p}X$ by \cite[Theorem~1.12]{Ravenel84}. The Bousfield localization $L_{\bbS/p}X$ is the derived $p$-completion of $X$ and our assumption on the homotopy groups of $X$ implies that it is already derived $p$-complete. Therefore $L_{\HFp}X \simeq L_{\bbS/p}X \simeq X$ and $X$ is $\HFp$-local.  

	Each Morava $K$-theory $K(i)$ is $\HFp$-acyclic by \cite[Theorem~2.1(i)]{Ravenel84}, so for $0\le i<\infty$ we have $\ihom{K(i),X}= 0$. This shows $\Cosupp^h(X)\subseteq\{\cat C_{\infty}\}$. If moreover $X$ is harmonic, then because $\HFp$ is dissonant (i.e., $K(\bbN)$-acyclic) by \cite[Theorem~4.7]{Ravenel84}, the same argument yields $\ihom{\HFp,X}= 0$. Together with the previous paragraph this forces $\Cosupp^h(X)=\emptyset$.
\end{proof}

\begin{Prop}\label{prop:bp/p-cosupport}
	Let $BP$ denote the Brown--Peterson spectrum. In the category of $p$-local spectra, $\Cosupph(BP/p)=\emptyset$.
\end{Prop}

\begin{proof}
	This follows from \cref{lem:cosupport-comp} because $BP$ is harmonic \cite[Corollary~4.5]{Ravenel84}, and hence so is $BP/p$. 
\end{proof}

\begin{Prop}\label{prop:harmonic-failure}
	Let $S\subseteq\mathbb{N}\cup\{\infty\}$ be infinite and set $K(S)\coloneqq \bigvee_{i\in S}K(i)$. Then the relative h-detection property holds for $\Sp_{K(S)}$ but the cohomological Bousfield class $\smash{\cbc{BP/p}}$ is not a homological Bousfield class. 
\end{Prop}

\begin{proof}
	Relative h-detection holds by \cref{cor:bc-chromatic-localizations}. For the second statement, it suffices to show that $BP/p$ is $K(S)$-local by \cref{prop:empty-cosupp-gives-nonHBC} and \cref{prop:bp/p-cosupport}. First suppose $\infty\notin S$. If $S\subseteq\mathbb N$ is infinite, then by \cite[Corollary~3.5]{Hovey95a} the spectrum $BP^{\wedge}_p$ is $K(S)$-local and hence so is $(BP^{\wedge}_p)/p\simeq BP/p$. Now suppose $\infty\in S$, and set $J\coloneqq S\setminus\{\infty\}$, which is an infinite subset of $\mathbb N$. Then $BP/p$ is $K(J)$-local by our argument above. Since $\bc{K(J)} \le \bc{K(S)}$ every $K(J)$-local object is $K(S)$-local. In particular $BP/p$ is $K(S)$-local. Therefore $BP/p\in L_{K(S)}\Sp$ for every infinite $S\subseteq \mathbb N\cup\{\infty\}$ and the conclusion follows from \cref{prop:empty-cosupp-gives-nonHBC}.
\end{proof}

\begin{Rem}
	In \cite[Question~4.5]{Wolcott15} Wolcott poses the problem of classifying the localizing subcategories of the harmonic category, i.e., the case $S=\bbN$. He writes that ``it seems likely that every localizing subcategory [...] is a Bousfield class.'' The previous result shows that this is not true. 
\end{Rem}

\begin{Prop}\label{prop:cosupp-HFp}
	Let $\mathbb{S}/p$ be the mod $p$ Moore spectrum. In the category of $p$-local spectra, $\Cosupp^h(\mathbb{S}/p) = \emptyset$.
\end{Prop}

\begin{proof}
	This is established in the proof of \cite[Proposition~6.4]{Wolcott15}; cf.~\cref{exa:hovey-wolcott}. It also follows from \cref{lem:cosupport-comp} since all finite spectra are harmonic by \cite[Corollary~4.5]{Ravenel84}.
\end{proof}

\begin{Prop}\label{prop:has-infty}
	Let $S \subseteq \mathbb{N}\cup\{\infty\}$ be a subset which contains $\infty$ and set $K(S) \coloneqq \bigvee_{i \in S} K(i)$. Then the relative h-detection property holds for $\Sp_{K(S)}$ but the cohomological Bousfield class $\cbc{\mathbb{S}/p}$ is not a homological Bousfield class.
\end{Prop}

\begin{proof}
	Relative h-detection holds by \cref{cor:bc-chromatic-localizations}. If $\infty \in S$ then $\bc{\HFp} \le \bc{K(S)}$. Hence, since $\mathbb{S}/p$ is $\HFp$-local, it is also $K(S)$-local. The result then follows from \cref{prop:cosupp-HFp} and \cref{prop:empty-cosupp-gives-nonHBC}.
\end{proof}

\begin{Cor}\label{cor:harmonic-failure}
	Let $S \subseteq\mathbb{N}\cup\{\infty\}$ be a subset. The category $\Sp_{K(S)}$ is relatively h-stratified if and only if $S$ is finite and does not contain $\infty$.
\end{Cor}

\begin{proof}
	If $\Sp_{K(S)}$ were relatively h-stratified then every localizing ideal must be a homological Bousfield class by \cref{prop:bousfield-stratified}. This is false when $S$ is infinite by \cref{prop:harmonic-failure} and it is false when $S$ contains $\infty$ by \cref{prop:has-infty}. On the other hand, $\Sp_{K(S)}$ is relatively h-stratified when $S$ is a finite subset of $\mathbb{N}$ by \cref{prop:K(S)-stratified}.
\end{proof}

\begin{Rem}
	One might hope that localization at the product $\prod_{i\in\mathbb{N}} K(i)$ behaves better than localization at the wedge. However, by \cite{Yosimura1985Acyclicity} we have
	\[
		\big\langle \prod_{i\in\mathbb{N}} K(i) \big\rangle_*
		=
		\big\langle \bigvee_{i\in\mathbb{N}} K(i) \big\rangle_*
		\vee
		\big\langle \HFp \big\rangle_*
	\]
	and so \cref{prop:harmonic-failure} applies in this case as well. 
\end{Rem}

\begin{Rem}
	We now extend this to an infinite family of examples.  Let $T(n)$ denote the $p$-local spectra fitting into the sequence\footnote{Beware the clash of notation: these are \emph{not} the telescopes of finite type-$n$ complexes, as appear in \cref{exa:other-chromatic-completion}.}
	\[
		\mathbb{S}_{(p)} = T(0) \longrightarrow T(1) \longrightarrow \cdots \longrightarrow T(n) \longrightarrow \cdots \longrightarrow T(\infty)=BP,
	\]
	and satisfying
	\begin{equation}\label{eq:bp}
		BP_*T(n) \cong BP_*[t_1,\ldots,t_n] \subseteq BP_*(BP) \cong BP_*[t_1,t_2,\ldots]
	\end{equation}
	where $|t_i| = 2(p^i-1)$. See \cite[Section~6.5]{ravenel-cc} for the construction.
\end{Rem}

\begin{Rem}
	In the following, we write $I(X)$ for the Brown--Comenetz dual of $X$ in the $p$-local stable homotopy category. We will also write $\cbc{A} \le \cbc{B}$ to denote reverse inclusion of cohomological Bousfield classes; cf.~\cref{def:bousfield-lattice}.
\end{Rem}

\begin{Lem}\label{prop:s(n)-properties}
	The following hold:
	\begin{enumerate}
		\item For each $0\le n \le \infty$, $\cbc{T(n)/p} = \bc{I(T(n)/p)}$.
		\item For each $0\le n < \infty$, $\ihom{T(n+1)/p,T(n)/p} = 0$.
		\item For each $0\le n < \infty$, $\ihom{BP/p,T(n)/p} = 0$.
	\end{enumerate}
\end{Lem}

\begin{proof}
	Part $(a)$ is a special case of \cite[Theorem~3.1]{hovey-cbc}: each $T(n)/p$ is $p$-local, of finite type, and satisfies $T(n)/p \otimes \HQ = 0$. For part $(b)$, \cite[Lemma~3.2(a)]{Ravenel84} gives $\ihom{T(n{+}1)/p, T(n)}=0$, and the cofiber sequence $T(n)\xrightarrow{p}T(n)\to T(n)/p$ yields $\ihom{T(n+1)/p, T(n)/p}=0$. Part $(c)$ follows in the same way from \cite[Lemma~3.2(b)]{Ravenel84}, which gives $\ihom{BP/p, T(n)}=0$.
\end{proof}

\begin{Lem}\label{lem:bc-decomposition}
	There is a strict chain of cohomological Bousfield classes
	\[
		\cbc{\bbS/p} \lneq \cbc{T(1)/p} \lneq \cdots \lneq \cbc{T(n)/p} \lneq \cbc{T(n{+}1)/p} \lneq \cdots \lneq \cbc{BP/p}.
	\]
\end{Lem}

\begin{proof}
	By \cite[Theorem~8.4]{HoveyPalmieri99} there is a strict chain of homological Bousfield classes
	\[
		\bc{I} = \bc{IT(0)} \lneq \bc{IT(1)} \lneq \cdots \lneq \bc{IT(n)} \lneq \bc{IT(n{+}1)} \lneq \cdots \lneq \bc{IBP}.
	\]
	By \cref{prop:s(n)-properties}(a) we have $\bc{I(T(k)/p)} =\cbc{T(k)/p}$ for all $0\le k\le\infty$, so we obtain the displayed chain. For the finite steps, strictness follows from \cref{prop:s(n)-properties}(b): each $T(n+1)/p \in \cbc{T(n)/p}$ and yet $T(n+1)/p \not \in \cbc{T(n+1)/p}$. For the last step, note that \cref{prop:s(n)-properties}(c) gives $BP/p\in \cbc{T(n)/p}$, while $BP/p\notin\cbc{BP/p}$ since $\ihom{BP/p,BP/p}\ne 0$.
\end{proof}

The following extends \cref{prop:bp/p-cosupport} and \cref{prop:cosupp-HFp}.

\begin{Prop}\label{prop:cosupp-ti}
	If $A$ is any $p$-local spectrum such that
	\[
	  \cbc{\bbS/p} \le \cbc{A} \le \cbc{BP/p}
	\]
	then $\Cosupph(A)=\emptyset$. In particular, this holds for all $T(n)/p$.
\end{Prop}

\begin{proof}
	Suppose $\cbc{\bbS/p}\le \cbc{A}\le \cbc{BP/p}$. By definition,
	\[
		\cat B\notin \Cosupph(A) \Longleftrightarrow \EB\in \cbc{A}.
	\]
	Since $\Cosupph(\bbS/p)=\emptyset$ and $\Cosupph(BP/p)=\emptyset$ (\cref{prop:cosupp-HFp} and \cref{prop:bp/p-cosupport}), both $\cbc{\bbS/p}$ and $\cbc{BP/p}$ contain every $\EB$. The inequalities then force~$\cbc{A}$ to contain every $\EB$ as well, hence $\Cosupph(A)=\emptyset$. The final claim follows from \cref{lem:bc-decomposition}.
\end{proof}

\begin{Rem}
	Our next step is to prove that $T(n)/p$ is $K(S)$-local for all $0\le n\le\infty$ and infinite $S \subseteq \mathbb{N} \cup \{\infty\}$. We will achieve this by reviewing the proof of \cite[Theorem~4.4]{Ravenel84}. 
\end{Rem}

\begin{Lem}\label{lem:locality-for-modules}
	Let $E$ be a bounded below associative $BP$-module spectrum. Then 
	\[
		E \otimes BP \in \Coloco{E}.
	\]
\end{Lem}

\begin{proof}
	Since $E$ is a $BP$-module, there is a K\"unneth isomorphism
	\[
		E_*BP \cong E_* \otimes_{BP_*} BP_*BP.
	\]
	Consequently,
	\[
		E_*BP \cong E_*[t_1,t_2,\dots] \cong \bigoplus_A \Sigma^{|A|}E_*
	\]
	where the sum ranges over all multi-indices $A=(a_1,a_2,\dots)$ with finitely many nonzero entries, and 
	\[
		|A| = \sum_i |t_i|a_i = \sum_i 2(p^i-1)a_i.
	\]
	This algebraic decomposition is realized topologically by maps
	\[
		t^A \colon \Sigma^{|A|}E \longrightarrow E\otimes BP
	\]
	corresponding to the monomials $t_1^{a_1}\cdots t_n^{a_n}$ in $E_*BP$.  Assembling these maps and using the bounded below hypothesis (which identifies an increasing wedge with the corresponding product because the degrees $|A|$ tend to infinity, so only finitely many summands contribute in each fixed degree) we obtain an equivalence
	\[
		E\otimes BP \simeq \prod_A \Sigma^{|A|}E.
	\]
	It follows that $E\otimes BP \in \Coloco{E}$.
\end{proof}

\begin{Lem}\label{lem:locality-reduction}
	Let $X$ be a bounded below $p$-local spectrum. Then 
	\[
		X \in \Coloco{BP \otimes X}.
	\]
\end{Lem}

\begin{proof}
	Let $\overline{BP}$ be the cofiber of the unit $\mathbb S \to BP$. We claim, by induction on $n\ge 0$, that
	\[
		BP \otimes \overline{BP}^{\otimes n} \otimes X \in \Coloco{BP \otimes X}.
	\]
	For $n=0$ there is nothing to prove. Now suppose the claim holds for some $n\ge 0$. From the cofiber sequence
	\[
		BP \otimes \overline{BP}^{\otimes n} \otimes X 
		\longrightarrow BP \otimes BP \otimes \overline{BP}^{\otimes n} \otimes X 
		\longrightarrow BP \otimes \overline{BP}^{\otimes (n+1)} \otimes X
	\]
	we see that the first term lies in $\Coloco{BP \otimes X}$ by the inductive hypothesis. Applying \cref{lem:locality-for-modules} we get 
	\[
		BP \otimes BP \otimes \overline{BP}^{\otimes n} \otimes X 
		\in \Coloco{ BP \otimes \overline{BP}^{\otimes n} \otimes X}.
	\]
	By the inductive hypothesis we also have $ BP \otimes \overline{BP}^{\otimes n} \otimes X \in \Coloco{BP \otimes X}$. Hence
	\[
		BP \otimes BP \otimes \overline{BP}^{\otimes n} \otimes X \in \Coloco{BP \otimes X}.
	\]
	It follows that $BP \otimes \overline{BP}^{\otimes (n+1)} \otimes X$ also lies in $\Coloco{BP \otimes X}$, as required.

	By convergence of the $BP$-Adams spectral sequence for bounded below $p$-local spectra, there is an equivalence $X \simeq \lim_s K_s X$, where $K_0 X = 0$ and for each $s\ge 1$ there is a cofiber sequence
	\[
		K_s X \longrightarrow K_{s-1} X \longrightarrow \Sigma^{-s+1} BP \otimes \overline{BP}^{\otimes (s-1)} \otimes X.
	\]
	By the claim, each term 
	\[
		\Sigma^{-s+1} BP \otimes \overline{BP}^{\otimes (s-1)} \otimes X
	\]
	belongs to $\Coloco{BP \otimes X}$. An inductive argument then establishes that 
	\[
		K_s X \in \Coloco{BP \otimes X}
	\]
	for all $s\ge 0$. Therefore, $X \simeq \lim_s K_s X$ also lies in ${\Coloco{BP \otimes X}}$.
\end{proof}

\begin{Prop}[Ravenel]\label{prop:ravenel-4.4}
	Let $X$ be a $p$-local bounded below spectrum such that
	\begin{enumerate}
		\item $\pi_iX$ is finitely generated over $\bbZ_{(p)}$ for each $i$, and 
		\item the homological dimension of $BP_*(X)$ as a $BP_*$-module is finite.
	\end{enumerate}
	Then $X \otimes BP \in \Coloco{BP}$. 
\end{Prop}

\begin{proof}
	This is contained in the proof of \cite[Theorem~4.4(b)]{Ravenel84}. 
\end{proof}

\begin{Thm}\label{thm:ravenel-ti-version}
	Let $X$ be a $p$-local bounded below spectrum such that
	\begin{enumerate}
		\item $\pi_iX$ is finitely generated over $\bbZ_{(p)}$ for each $i$, and 
		\item the homological dimension of $BP_*(X)$ as a $BP_*$-module is finite.
	\end{enumerate}
	Then $X \in \Coloco{BP}$. 
\end{Thm}

\begin{proof}
	By \cref{lem:locality-reduction,prop:ravenel-4.4} we have
	\[
		X \in \Coloco{BP \otimes X} \subseteq \Coloco{BP}.\qedhere
	\]
\end{proof}

\begin{Prop}\label{prop:locality-Tn}
	Let $S\subseteq\mathbb{N}\cup\{\infty\}$ be an infinite subset and set $K(S)\coloneqq\bigvee_{i\in S}K(i)$. Then $T(n)/p$ is $K(S)$-local for all $0\le n\le\infty$.
\end{Prop}

\begin{proof}
	As in \cref{prop:harmonic-failure}, we reduce to the case $S\subseteq\mathbb{N}$. We will apply \cref{thm:ravenel-ti-version} for $X = T(n)$ and so it suffices to verify that
	\begin{enumerate}
		\item $T(n)$ is $p$-local and connective;
		\item each $\pi_i T(n)$ is finitely generated over $\Z_{(p)}$;
		\item $BP_*T(n)$ has finite homological dimension as a $BP_*$-module.
	\end{enumerate}
	The first is part of the construction. For the second, we use a Serre class argument using that the homology is finitely generated; see \cite[Section~1.3]{Hopkins1984Stable}. The final part follows from \eqref{eq:bp}. It follows that $T(n) \in \Coloco{BP}$ and therefore  $T(n)/p \in \Coloco{BP}$ as well. 

	Applying \cite[Lemma~2.5]{BarthelCastellanaHeardValenzuela19} to the $p$-completion functor, we deduce that 
	\[
		T(n)/p \in \Coloco{BP^{\wedge}_p}.
	\]
	By \cite[Corollary~3.5]{Hovey95a} $BP^{\wedge}_p$ is $K(S)$-local, and hence so is $T(n)/p$.
\end{proof}

\begin{Thm}\label{thm:harmonic-failure-Tn}
	Let $S\subseteq\mathbb{N}\cup\{\infty\}$ be an infinite subset, and set $K(S)\coloneqq\bigvee_{i\in S}K(i)$. For all $0\le n\le\infty$, the cohomological Bousfield class $\cbc{T(n)/p}$ is not a homological Bousfield class in $\Sp_{K(S)}$.
\end{Thm}

\begin{proof}
	This follows from \cref{prop:empty-cosupp-gives-nonHBC}: the cosupport condition is verified for $T(n)/p$ in \cref{prop:cosupp-ti}, and $T(n)/p$ is $K(S)$-local by \cref{prop:locality-Tn}.
\end{proof}

We note that the following also holds. 

\begin{Thm}\label{thm:exotic-HFp}
	For $0 \le n \le \infty$ each $T(n)/p$ is $\HFp$-local and $\cbc{T(n)/p}$ is a cohomological Bousfield class in the category $L_{\HFp}\Sp$ that is not a homological Bousfield class. 
\end{Thm}

\begin{proof}
    Each $T(n)/p$ is bounded below and has bounded $p$-torsion homotopy groups, so the first paragraph of the proof of \cref{lem:cosupport-comp} shows that it is $\HFp$-local. The cosupport condition is \cref{prop:cosupp-ti}; hence the claim follows from \cref{prop:empty-cosupp-gives-nonHBC} applied in $L_{\HFp}\Sp$.
\end{proof}

\begin{Rem}
	The Bousfield lattice of $L_{\HFp}\Sp$ has only two elements, since $\HFp$ is a field spectrum. For $n=0$, the fact that $\cbc{\bbS/p}$ is a cohomological Bousfield class that is not homological is due to Hovey; see \cref{exa:bad}. Wolcott also showed that $\cbc{L_{\HFp}BP}$ is not a homological Bousfield class. The case $n=\infty$ above is a mod $p$ analogue of his result, and in general the $T(n)$ interpolate between these extremal examples. 
\end{Rem}

We conclude this section with another collection of cohomological Bousfield classes that are not homological.

\begin{Prop}\label{prop:bpj}
	For every infinite invariant regular sequence $J\neq(p,v_1,v_2,\ldots)$ the cohomological Bousfield class $\cbc{BPJ}$ in $L_{\HFp}\Sp$ is not a homological Bousfield class.
\end{Prop}

\begin{proof}
	Each $BPJ$ is connective and bounded $p$-torsion, hence $\HFp$-local; thus $BPJ\in L_{\HFp}\Sp$. Suppose $\cbc{BPJ}=\bc{E}$ is homological in $L_{\HFp}\Sp$. The homological Bousfield lattice of $L_{\HFp}\Sp$ has only two elements, namely $\bc{0}=L_{\HFp}\Sp$ and $\bc{\HFp}=\{0\}$. As $BPJ\neq 0$ and $BPJ\notin\cbc{BPJ}$, we cannot have $\cbc{BPJ}=\bc{0}$, hence $\cbc{BPJ}=\{0\}$. But by the proof of \cite[Theorem~2.10]{Ravenel84} one has $\ihom{\HFp,BPJ}=0$ whenever $J\neq(p,v_1,v_2,\ldots)$, so $\HFp\in\cbc{BPJ}\neq\{0\}$, a contradiction. Therefore $\cbc{BPJ}$ is not homological.
\end{proof}

\begin{Rem}\label{rem:uncountable}
	In \cite[Conjecture~2.8]{Ravenel84} Ravenel conjectured that $\bc{BPJ}=\bc{BPK}$ if and only if $J\sim K$, where $\sim$ is the equivalence relation of \cite[Definition~2.7]{Ravenel84}. We likewise conjecture that $\cbc{BPJ}=\cbc{BPK}$ if and only if $J\sim K$. In particular, \cref{prop:bpj} would imply that there are uncountably many cohomological Bousfield classes in $L_{\HFp}\Sp$ that are not homological.
\end{Rem}

\section{Counterexamples behaving badly}\label{sec:bad-behaviour}

In the previous section, we produced cohomological Bousfield classes in certain localizations of $\Sp$ which are not homological. However, we have not proved that~$\Sp$ itself has such ``oddball'' or ``exotic'' cohomological Bousfield classes. The reader may find this somewhat curious because if $\cat T$ is h-stratified (or has the h-codetection property or the h-detection property) then any localized category $\LAT$ satisfies the relative version; see \cref{thm:local-stratification}, for example. However, the property ``every cohomological Bousfield class is a homological Bousfield class'' does not obviously pass from $\cat T$ to $\LAT$. If it did, we would obtain counterexamples for $\Sp$ from the counterexamples of \cref{sec:oddball_cohomological_bousfield_classes}. Our final goal in this paper is to explore this issue.

\begin{Lem}\label{lem:coho-homo-ascent}
	Let $\cat T$ be a rigidly-compactly generated tt-category and let $A \in \cat T$. For $x \in \LAT$ and $t \in \cat T$, the following are equivalent:
	\begin{enumerate}
		\item There is an equality of Bousfield classes $\cbc{x}=\bc{\LA t}$ in $\LAT$.
		\item There is an equality of Bousfield classes $\cbc{x}=\bc{A \otimes t}$ in $\cat T$.
	\end{enumerate}
\end{Lem}

\begin{proof}
	This is immediate from the definitions; see \cref{rem:local-bc} and \cref{rem:local-coho}.
\end{proof}

\begin{Prop}\label{prop:char-coho-homo}
	For any $A \in \cat T$, the following are equivalent:
	\begin{enumerate}
		\item Every cohomological Bousfield class of $\LAT$ is a homological Bousfield class.
		\item For every $x \in \LAT$, the cohomological Bousfield class $\cbc{x}$ in $\cat T$ coincides with a homological Bousfield class $\bc{A\otimes s}$ for some $s \in \cat T$.
	\end{enumerate}
\end{Prop}

\begin{proof}
	\Cref{lem:coho-homo-ascent} immediately gives $(a) \Rightarrow (b)$. It also gives $(b) \Rightarrow (a)$ once we note that $\bc{A\otimes s} = \bc{A\otimes \LA s}$.
\end{proof}

\begin{Rem}
	For any local object $x\in \LAT$, the localizing ideal $\cbc{x}$ in $\cat T$ has the following special features:
	\begin{enumerate}
		\item It contains all the $A$-acyclic objects.
		\item It contains an object $t \in \cat T$ if and only if it contains its localization $\LA t$.
	\end{enumerate}
	Note that any localizing ideal of the form $\bc{A\otimes s}$ also shares these two features. However, it is not \emph{a priori} true that a homological Bousfield class $\bc{t}$ having those features need be of the form $\bc{A\otimes s}$. This provides one perspective on why the existence of exotic cohomological Bousfield classes in $\LAT$ does not necessarily imply the same for $\cat T$: a cohomological Bousfield class $\cbc{t}$ in $\cat T$ might be of the form $\bc{s}$ for some $s \in \cat T$ but not of the form $\bc{A\otimes s}$ for any $s \in \cat T$.
\end{Rem}

\begin{Prop}\label{prop:dodgy-a}
	Suppose that every cohomological Bousfield class of $\cat T$ is a homological Bousfield class and let $A \in \cat T$. If $\LAT$ has a cohomological Bousfield class which is not homological then there exists an object $a\in\cat T$ such that
	\begin{equation}\label{eq:a-nest}
		0 \neq a \in \bc{A} \subseteq \bc{a}.
	\end{equation}
\end{Prop}

\begin{proof}
	Let $x \in \LAT$ be a local object whose cohomological Bousfield class in $\LAT$ is not homological. By \cref{lem:coho-homo-ascent}, this implies that $\cbc{x} \neq \bc{A\otimes s}$ in $\cat T$ for any $s \in \cat T$. On the other hand, by our hypothesis on $\cat T$, we have $\cbc{x} = \bc{t}$ for some $t \in \cat T$. Thus $\bc{t}\neq \bc{A\otimes t}$. Therefore, $\bc{A\otimes t} \not\subseteq \bc{t}$ since the other inclusion is trivial. We conclude that there exists an object $b \in \cat T$ such that $t \otimes b \neq 0$ but $A \otimes t \otimes b = 0$. Thus $t \otimes b$ is a nonzero $A$-acyclic object. Moreover, since $x$ is $A$-local, $\bc{t}=\cbc{x}$ contains all $A$-acyclic objects. Thus $\bc{A} \subseteq \bc{t} \subseteq \bc{t\otimes b}$. We have established that $a \coloneqq t\otimes b$ is an object satisfying \eqref{eq:a-nest}.
\end{proof}

\begin{Rem}
	For an object $a \in \cat T$, the following two statements are equivalent:
	\begin{enumerate}
		\item $a$ satisfies~\eqref{eq:a-nest}.
		\item $a$ is a nonzero $A$-acyclic object whose Brown--Comenetz dual $I_a$ is $A$-local.
	\end{enumerate}
	Note that any such object satisfies $a \otimes a = 0$.
\end{Rem}

\begin{Cor}\label{cor:anti-nilpotent-coho-lift}
	Suppose that every cohomological Bousfield class of $\cat T$ is a homological Bousfield class. If no nontrivial $A$-acyclic object is tensor-nilpotent then every cohomological Bousfield class of $\LAT$ is a homological Bousfield class.
\end{Cor}

\begin{proof}
	This follows from \cref{prop:dodgy-a} since any object $a \in \cat T$ satisfying \eqref{eq:a-nest} is a nontrivial $A$-acyclic tensor-nilpotent object.
\end{proof}

\begin{Exa}
	The $\HFp$-acyclics contain the tensor-nilpotent object $I$. This gives one perspective for why our ``exotic'' cohomological Bousfield classes in $L_{\HFp}\Sp$ do not provide ``exotic'' cohomological Bousfield classes in $\Sp$. More generally, for any subset $S \subseteq \Spc^h(\cat T^c)$, the $K(S)$-acyclics contain all tensor-nilpotent objects. Thus, we cannot use the corollary to lift counterexamples in $L_{\KS}\cat T$ to counterexamples in $\cat T$ so long as $\cat T$ has nonzero tensor-nilpotent objects.
\end{Exa}

\begin{Cor}\label{cor:ascent-equiv}
	Suppose that every cohomological Bousfield class of $\cat T$ is a homological Bousfield class. If $A \in \cat T$ satisfies $\bc{A} = \bc{A\otimes A}$ then the following are equivalent:
	\begin{enumerate}
		\item There exists a cohomological Bousfield class in $\LAT$ which is not homological.
		\item There exists an object $a \in \cat T$ satisfying \eqref{eq:a-nest}.
	\end{enumerate}
\end{Cor}

\begin{proof}
	The $(a) \Rightarrow (b)$ direction is provided by \cref{prop:dodgy-a} and does not require the hypothesis on $A$. Conversely, suppose $a \in \cat T$ satisfies \eqref{eq:a-nest}. The inclusion $\bc{A} \subseteq \bc{a}$ is equivalent to the statement that the Brown--Comenetz dual $I_a$ is \mbox{$A$-local.} This follows from the fact that $\ihom{t,I_a}=0$ if and only if $t \otimes a = 0$. We claim that the cohomological Bousfield class $\cbc{I_a}$ in $\LAT$ is not homological. Suppose for a contradiction that it was a homological Bousfield class. By \cref{lem:coho-homo-ascent}, this would mean that we have an equality of Bousfield classes $\cbc{I_a} = \bc{A\otimes s}$ in~$\cat T$ for some $s \in \cat T$. On the other hand, we know that $\cbc{I_a} = \bc{a}$. Thus, we would have $\bc{a} = \bc{A\otimes s}$. Now, by hypothesis $a \in \bc{A}$, that is, $A\otimes a=0$. Hence $A \otimes A \otimes s = 0$. By our hypothesis that $\bc{A}=\bc{A\otimes A}$, this implies that $A \otimes s=0$. Hence $a=0$, which is a contradiction.
\end{proof}

\begin{Lem}
	The following conditions are equivalent:
	\begin{enumerate}
		\item $\bc{A}=\bc{A\otimes A}$ for every $A \in \cat T$.
		\item There are no nonzero tensor-nilpotent objects in $\cat T$.
	\end{enumerate}
\end{Lem}

\begin{proof}
	Suppose $t$ is a nonzero tensor-nilpotent object in $\cat T$. Let $n_0$ be the largest positive integer such that $t^{\otimes n_0} \neq 0$. Then $A\coloneqq t^{\otimes n_0}$ is a nonzero object satisfying $A\otimes A=0$. Hence $\bc{A} \neq \bc{A\otimes A}$. This gives the contrapositive of $(a)\Rightarrow (b)$. Now for $(b) \Rightarrow (a)$. Consider an arbitrary object $A \in \cat T$. We always have $\bc{A} \subseteq \bc{A\otimes A}$. On the other hand, if $A\otimes A\otimes t=0$ then $A\otimes t$ is tensor-nilpotent. Thus $(b)$ implies $A\otimes t=0$. Hence $\bc{A}=\bc{A\otimes A}$.
\end{proof}

\begin{Cor}
	Let $\cat T$ be a rigidly-compactly generated tt-category which has no nonzero tensor-nilpotent objects. Suppose that every cohomological Bousfield class of $\cat T$ is a homological Bousfield class. Then for any $A \in \cat T$, every cohomological Bousfield class of $\LAT$ is a homological Bousfield class.
\end{Cor}

\begin{Prop}\label{prop:idempotent-gluing}
	Let $\cat T$ be a rigidly-compactly generated tt-category and let $e \to \unit \to f \to \Sigma e$ be an idempotent triangle. The following statements are equivalent:
	\begin{enumerate}
		\item Every cohomological Bousfield class of $\cat T$ is homological.
		\item Every cohomological Bousfield class of the tt-category $f \otimes \cat T$ is homological, and every cohomological Bousfield class of the tt-category $e \otimes \cat T$ is homological.
	\end{enumerate}
\end{Prop}

\begin{proof}
	Recall that $f \otimes \cat T \simeq L_f \cat T$ and $e \otimes \cat T \simeq \ihom{e,\cat T} \simeq L_e \cat T$. Note that for $A=e$ or $A=f$ the condition $\bc{A}=\bc{A\otimes A}$ is trivially satisfied. Moreover, there cannot exist an object $a \in \cat T$ satisfying~\eqref{eq:a-nest}. For example, $a \in \bc{f}$ means that $a \simeq e\otimes a$ but $\bc{f} \subseteq \bc{a}$ implies that $e\in \bc{a}$ so that $a = 0$. The same argument holds swapping the roles of $e$ and $f$. Thus the $(\Rightarrow)$ direction follows from~\cref{cor:ascent-equiv}.

	For the $(\Leftarrow)$ direction, first recall that $\ihom{f,t}\simeq f\otimes \ihom{f,t}$ for any $t \in \cat T$, since $\ihom{x,\ihom{f,t}}\simeq \ihom{x\otimes f,t} = 0$ for all $f$-acyclic objects $x$. Also recall that $e \otimes \ihom{e,t} \simeq e \otimes t$ and $\ihom{e,e \otimes t} \simeq \ihom{e,t}$ for any $t \in \cat T$. In particular, $e \otimes t = 0$ if and only if $\ihom{e,t}=0$. (This also follows from~\cref{rem:coring}.) With this in hand, consider an arbitrary object $t \in \cat T$. By hypothesis, we have an equality 
	\[
		\cbc{\ihom{f,t}} = \bc{f\otimes a}
	\]
	of Bousfield classes in $f\otimes \cat T$ for some $a \in \cat T$. By \cref{lem:coho-homo-ascent}, it follows that we have an equality
	\begin{equation}\label{eq:fcbc}
		\cbc{\ihom{f,t}} = \bc{f^{\otimes 2} \otimes a}=\bc{f\otimes a}
	\end{equation}
	of Bousfield classes in $\cat T$. On the other hand, by hypothesis, we have an equality
	\[
		\cbc{e\otimes t} = \bc{e\otimes b}
	\]
	of Bousfield classes in $e \otimes \cat T$ for some $b \in \cat T$. We claim that we have an equality
	\begin{equation}\label{eq:ecbc}
		\cbc{\ihom{e,t}} = \bc{e\otimes b}
	\end{equation}
	of Bousfield classes in $\cat T$. Recall that the functor $e \otimes -\colon\cat T \to e\otimes \cat T$ is \emph{closed} symmetric monoidal, i.e., preserves the internal hom. In particular, the internal hom in $e \otimes \cat T$ of $e \otimes s$ and $e \otimes t$ is given by $e \otimes \ihom{s,t}$. Observe that
	\begin{align*}
		s \in \cbc{\ihom{e,t}} \text{ in }\cat T &\iff \ihom{s,\ihom{e,t}} = 0 \\
							   &\iff \ihom{e,\ihom{s,t}} = 0 \\
							   &\iff e \otimes \ihom{s,t} =0 \\
							   &\iff e\otimes s \in \cbc{e\otimes t} \text{ in } e\otimes \cat T\\
							   &\iff e \otimes s \in \bc{e\otimes b} \text{ in } e\otimes \cat T\\
							   &\iff s \in \bc{e\otimes b} \text{ in } \cat T
	\end{align*}
	which establishes \eqref{eq:ecbc}. We now claim that 
	\begin{equation}\label{eq:tcbc}
		\cbc{t} = \bc{(f\otimes a)\oplus(e\otimes b)}.
	\end{equation}
	Note that if $\ihom{s,t}=0$ then $\ihom{s,\ihom{e,t}}=\ihom{e,\ihom{s,t}}=0$ and $\ihom{s,\ihom{f,t}}=\ihom{f,\ihom{s,t}}=0$. Conversely, the exact triangle 
	\[
		\ihom{f,t}\to t \to \ihom{e,t} \to \Sigma\ihom{f,t}
	\]
	shows that $\ihom{s,t}=0$ if $\ihom{s,\ihom{e,t}}=0$ and $\ihom{s,\ihom{f,t}}=0$. Hence,
	\begin{align*}
		s \in \cbc{t} &\iff s \in \cbc{\ihom{f,t}} \text{ and } s \in \cbc{\ihom{e,t}} \\
					  &\iff s \in \bc{f\otimes a} \text{ and } s \in \bc{e\otimes b} \\
					  &\iff s \in \bc{(f \otimes a) \oplus (e \otimes b)}
	\end{align*}
	by \eqref{eq:fcbc} and \eqref{eq:ecbc}. This establishes that every cohomological Bousfield class in~$\cat T$ is homological.
\end{proof}

\begin{Exa}
	For any Thomason subset $Y \subseteq \Spc(\cat T^c)$, the category $\cat T$ has no exotic cohomological classes if and only if the same is true for the localization $\cat T|_{Y^c}$ and the completion $\cat T_Y^{\wedge}$.
\end{Exa}

\bibliographystyle{alphasort}
\bibliography{bib}

@article {BarthelHeardSandersZou26,
    AUTHOR = {Barthel, Tobias and Heard, Drew and Sanders, Beren and Zou,
              Changhan},
     TITLE = {Homological stratification and descent},
   JOURNAL = {J. Inst. Math. Jussieu},
  FJOURNAL = {Journal of the Institute of Mathematics of Jussieu. JIMJ.
              Journal de l'Institut de Math\'ematiques de Jussieu},
    VOLUME = {25},
      YEAR = {2026},
    NUMBER = {2},
     PAGES = {1081--1125},
      ISSN = {1474-7480,1475-3030},
   MRCLASS = {18G80 (18F99 55P91 55U35)},
  MRNUMBER = {5023884},
       DOI = {10.1017/S1474748025101515},
       URL = {https://doi.org/10.1017/S1474748025101515},
}

@article {Ohkawa89,
    AUTHOR = {Ohkawa, Tetsusuke},
     TITLE = {The injective hull of homotopy types with respect to
              generalized homology functors},
   JOURNAL = {Hiroshima Math. J.},
  FJOURNAL = {Hiroshima Mathematical Journal},
    VOLUME = {19},
      YEAR = {1989},
    NUMBER = {3},
     PAGES = {631--639},
      ISSN = {0018-2079},
   MRCLASS = {55P42 (55P60)},
  MRNUMBER = {1035147},
MRREVIEWER = {John F. Oprea},
       URL = {http://projecteuclid.org/euclid.hmj/1206129296},
}

@article {hovey-cbc,
    AUTHOR = {Hovey, Mark},
     TITLE = {Cohomological {B}ousfield classes},
   JOURNAL = {J. Pure Appl. Algebra},
  FJOURNAL = {Journal of Pure and Applied Algebra},
    VOLUME = {103},
      YEAR = {1995},
    NUMBER = {1},
     PAGES = {45--59},
      ISSN = {0022-4049},
   MRCLASS = {55P42 (55N20 55P15 55P60)},
  MRNUMBER = {1354066},
MRREVIEWER = {Andrew J. Baker},
       DOI = {10.1016/0022-4049(94)00096-2},
       URL = {https://doi.org/10.1016/0022-4049(94)00096-2},
}

@book {ravenel-cc,
    AUTHOR = {Ravenel, Douglas C.},
     TITLE = {Complex cobordism and stable homotopy groups of spheres},
    SERIES = {Pure and Applied Mathematics},
    VOLUME = {121},
 PUBLISHER = {Academic Press, Inc., Orlando, FL},
      YEAR = {1986},
     PAGES = {xx+413},
      ISBN = {0-12-583430-6; 0-12-583431-4},
   MRCLASS = {55-02 (55Qxx 57-02)},
  MRNUMBER = {860042},
MRREVIEWER = {Joseph Neisendorfer},
}

@article{Zou25supp,
    AUTHOR = {Zou, Changhan},
     TITLE = {Support theories for non-{N}oetherian tensor triangulated
              categories},
   JOURNAL = {Adv. Math.},
  FJOURNAL = {Advances in Mathematics},
    VOLUME = {482},
      YEAR = {2025},
     PAGES = {Paper No. 110615},
      ISSN = {0001-8708,1090-2082},
   MRCLASS = {18M15 (55P91)},
  MRNUMBER = {4975903},
       DOI = {10.1016/j.aim.2025.110615},
       URL = {https://doi.org/10.1016/j.aim.2025.110615},
}

@unpublished{barthel2023cosupport,
    AUTHOR = {Barthel, Tobias and Castellana, Nat\`alia and Heard, Drew and Sanders, Beren},
   JOURNAL = {},
     TITLE = {Cosupport in tensor triangular geometry},
  YEAR = {2023},
      NOTE = {Preprint, 87~pages, 
      to appear in \emph{Astérisque}},
   MRCLASS = {},
}

@article {DwyerGreenlees02,
    AUTHOR = {Dwyer, W. G. and Greenlees, J. P. C.},
     TITLE = {Complete modules and torsion modules},
   JOURNAL = {Amer. J. Math.},
  FJOURNAL = {American Journal of Mathematics},
    VOLUME = {124},
      YEAR = {2002},
    NUMBER = {1},
     PAGES = {199--220},
      ISSN = {0002-9327},
   MRCLASS = {16E30 (16D90 18E30)},
  MRNUMBER = {1879003},
MRREVIEWER = {Henning Krause},
       URL =
              {http://muse.jhu.edu.oca.ucsc.edu/journals/american_journal_of_mathematics/v124/124.1dwyer.pdf},
}

@article {CasacubertaGutierrezRosicky14,
    AUTHOR = {Casacuberta, Carles and Guti\'{e}rrez, Javier J. and Rosick\'{y},
              Ji\v{r}\'{\i}},
     TITLE = {Are all localizing subcategories of stable homotopy categories
              coreflective?},
   JOURNAL = {Adv. Math.},
  FJOURNAL = {Advances in Mathematics},
    VOLUME = {252},
      YEAR = {2014},
     PAGES = {158--184},
      ISSN = {0001-8708},
   MRCLASS = {55P42 (03E55 18E30 18G55 55P60)},
  MRNUMBER = {3144227},
MRREVIEWER = {Julia Bergner},
       DOI = {10.1016/j.aim.2013.10.013},
       URL = {https://doi-org.oca.ucsc.edu/10.1016/j.aim.2013.10.013},
}

@article {HoveyStrickland99,
    AUTHOR = {Hovey, Mark and Strickland, Neil P.},
     TITLE = {Morava {$K$}-theories and localisation},
   JOURNAL = {Mem. Amer. Math. Soc.},
  FJOURNAL = {Memoirs of the American Mathematical Society},
    VOLUME = {139},
      YEAR = {1999},
    NUMBER = {666},
     PAGES = {viii+100},
      ISSN = {0065-9266},
     CODEN = {MAMCAU},
   MRCLASS = {55P60 (55N22 55P42 55T15)},
  MRNUMBER = {1601906},
MRREVIEWER = {J. P. C. Greenlees},
       DOI = {10.1090/memo/0666},
       URL = {http://dx.doi.org.ep.fjernadgang.kb.dk/10.1090/memo/0666},
}

@unpublished{BeardsleyHarmonic2013,
  title={Some notes on the category of {$p$}-local harmonic spectra},
  author={Beardsley, Jonathan},
  note={Unpublished notes available at \url{https://www.jonathanbeardsley.com/harmoniclattice.pdf}},
  year={2013}
}

@article {BalmerSanders17,
    AUTHOR = {Balmer, Paul and Sanders, Beren},
     TITLE = {The spectrum of the equivariant stable homotopy category of a
              finite group},
   JOURNAL = {Invent. Math.},
  FJOURNAL = {Inventiones Mathematicae},
    VOLUME = {208},
      YEAR = {2017},
    NUMBER = {1},
     PAGES = {283--326},
      ISSN = {0020-9910},
   MRCLASS = {18E30 (55P42 55U35)},
  MRNUMBER = {3621837},
MRREVIEWER = {Geoffrey M. L. Powell},
       DOI = {10.1007/s00222-016-0691-3},
       URL = {https://doi.org/10.1007/s00222-016-0691-3},
}

@article {BalmerDellAmbrogioSanders16,
    AUTHOR = {Balmer, Paul and Dell'Ambrogio, Ivo and Sanders, Beren},
     TITLE = {Grothendieck--{N}eeman duality and the {W}irthm\"uller
              isomorphism},
   JOURNAL = {Compos. Math.},
  FJOURNAL = {Compositio Mathematica},
    VOLUME = {152},
      YEAR = {2016},
    NUMBER = {8},
     PAGES = {1740--1776},
      ISSN = {0010-437X},
   MRCLASS = {18E30 (14F05 55U35)},
  MRNUMBER = {3542492},
       DOI = {10.1112/S0010437X16007375},
       URL = {http://dx.doi.org.ep.fjernadgang.kb.dk/10.1112/S0010437X16007375},
}

@article {Keller94b,
    AUTHOR = {Keller, Bernhard},
     TITLE = {A remark on the generalized smashing conjecture},
   JOURNAL = {Manuscripta Math.},
  FJOURNAL = {Manuscripta Mathematica},
    VOLUME = {84},
      YEAR = {1994},
    NUMBER = {2},
     PAGES = {193--198},
      ISSN = {0025-2611},
     CODEN = {MSMHB2},
   MRCLASS = {18E30 (18A40 18D99 55P60 55U99)},
}

@article {Neeman92a,
    AUTHOR = {Neeman, Amnon},
     TITLE = {The chromatic tower for {$D(R)$}},
   JOURNAL = {Topology},
  FJOURNAL = {Topology. An International Journal of Mathematics},
    VOLUME = {31},
      YEAR = {1992},
    NUMBER = {3},
     PAGES = {519--532},
}

@article {Ravenel84,
    AUTHOR = {Ravenel, Douglas C.},
     TITLE = {Localization with respect to certain periodic homology
              theories},
   JOURNAL = {Amer. J. Math.},
  FJOURNAL = {American Journal of Mathematics},
    VOLUME = {106},
      YEAR = {1984},
    NUMBER = {2},
     PAGES = {351--414},
      ISSN = {0002-9327},
     CODEN = {AJMAAN},
   MRCLASS = {55P42},
}

@article {Bousfield79,
    AUTHOR = {Bousfield, A. K.},
     TITLE = {The localization of spectra with respect to homology},
   JOURNAL = {Topology},
  FJOURNAL = {Topology. An International Journal of Mathematics},
    VOLUME = {18},
      YEAR = {1979},
    NUMBER = {4},
     PAGES = {257--281},
      ISSN = {0040-9383},
     CODEN = {TPLGAF},
   MRCLASS = {55N20 (55N15 55P60)},
}

@book {Neeman01,
    AUTHOR = {Neeman, Amnon},
     TITLE = {Triangulated categories},
    SERIES = {Annals of Mathematics Studies},
    VOLUME = {148},
 PUBLISHER = {Princeton University Press},
      YEAR = {2001},
}

@incollection {Krause10,
    AUTHOR = {Krause, Henning},
     TITLE = {Localization for triangulated categories},
 BOOKTITLE = {Triangulated categories},
    SERIES = {London Math. Soc. Lecture Note Ser.},
    VOLUME = {375},
     PAGES = {161--235},
 PUBLISHER = {Cambridge Univ. Press},
   ADDRESS = {Cambridge},
      YEAR = {2010},
   MRCLASS = {18E30},
}

@book {Hopkins1984Stable,
    AUTHOR = {Hopkins, Michael Jerome},
     TITLE = {Stable decompositions of certain loop spaces},
      NOTE = {Thesis (Ph.D.)--Northwestern University},
 PUBLISHER = {ProQuest LLC, Ann Arbor, MI},
      YEAR = {1984},
     PAGES = {96},
   MRCLASS = {Thesis},
  MRNUMBER = {2633919},
}

@article {Wolcott15,
    AUTHOR = {Wolcott, F. Luke},
     TITLE = {Variations of the telescope conjecture and {B}ousfield
              lattices for localized categories of spectra},
   JOURNAL = {Pacific J. Math.},
  FJOURNAL = {Pacific Journal of Mathematics},
    VOLUME = {276},
      YEAR = {2015},
    NUMBER = {2},
     PAGES = {483--509},
      ISSN = {0030-8730},
   MRCLASS = {55P42 (18E30 55P60 55U35)},
  MRNUMBER = {3374070},
MRREVIEWER = {Julia Bergner},
       DOI = {10.2140/pjm.2015.276.483},
       URL = {https://doi.org/10.2140/pjm.2015.276.483},
}

@article {Stevenson14,
    AUTHOR = {Stevenson, Greg},
     TITLE = {Derived categories of absolutely flat rings},
   JOURNAL = {Homology Homotopy Appl.},
  FJOURNAL = {Homology, Homotopy and Applications},
    VOLUME = {16},
      YEAR = {2014},
    NUMBER = {2},
     PAGES = {45--64},
      ISSN = {1532-0073},
   MRCLASS = {13D09},
  MRNUMBER = {3234500},
MRREVIEWER = {Sunil K. Chebolu},
       DOI = {10.4310/HHA.2014.v16.n2.a3},
       URL = {https://doi.org/10.4310/HHA.2014.v16.n2.a3},
}

@article {IyengarKrause13,
    AUTHOR = {Iyengar, Srikanth B. and Krause, Henning},
     TITLE = {The {B}ousfield lattice of a triangulated category and
              stratification},
   JOURNAL = {Math. Z.},
  FJOURNAL = {Mathematische Zeitschrift},
    VOLUME = {273},
      YEAR = {2013},
    NUMBER = {3-4},
     PAGES = {1215--1241},
      ISSN = {0025-5874},
   MRCLASS = {18E30 (55U35)},
  MRNUMBER = {3030697},
MRREVIEWER = {Stanis\l aw Betley},
       DOI = {10.1007/s00209-012-1051-7},
       URL = {https://doi.org/10.1007/s00209-012-1051-7},
}

@article {Bousfield79_boolean,
    AUTHOR = {Bousfield, A. K.},
     TITLE = {The {B}oolean algebra of spectra},
   JOURNAL = {Comment. Math. Helv.},
  FJOURNAL = {Commentarii Mathematici Helvetici},
    VOLUME = {54},
      YEAR = {1979},
    NUMBER = {3},
     PAGES = {368--377},
      ISSN = {0010-2571},
   MRCLASS = {55P42},
  MRNUMBER = {543337},
MRREVIEWER = {John Roy Dennett},
       DOI = {10.1007/BF02566281},
       URL = {https://doi.org/10.1007/BF02566281},
}

@article {bgh_balmer,
    AUTHOR = {Barthel, Tobias and Greenlees, J. P. C. and Hausmann, Markus},
     TITLE = {On the {B}almer spectrum for compact {L}ie groups},
   JOURNAL = {Compos. Math.},
  FJOURNAL = {Compositio Mathematica},
    VOLUME = {156},
      YEAR = {2020},
    NUMBER = {1},
     PAGES = {39--76},
      ISSN = {0010-437X},
   MRCLASS = {55P42 (55P91)},
  MRNUMBER = {4036448},
MRREVIEWER = {Samik Basu},
       DOI = {10.1112/s0010437x19007656},
       URL = {https://doi.org/10.1112/s0010437x19007656},
}

@article {BalmerKrauseStevenson19,
    AUTHOR = {Balmer, Paul and Krause, Henning and Stevenson, Greg},
     TITLE = {Tensor-triangular fields: ruminations},
   JOURNAL = {Selecta Math. (N.S.)},
  FJOURNAL = {Selecta Mathematica. New Series},
    VOLUME = {25},
      YEAR = {2019},
    NUMBER = {1},
     PAGES = {Paper No. 13, 36},
      ISSN = {1022-1824},
   MRCLASS = {18E30 (20J05 55U35)},
  MRNUMBER = {3911737},
MRREVIEWER = {Fosco Loregian},
       DOI = {10.1007/s00029-019-0454-2},
       URL = {https://doi.org/10.1007/s00029-019-0454-2},
}

@article {Balmer20_nilpotence,
    AUTHOR = {Balmer, Paul},
     TITLE = {Nilpotence theorems via homological residue fields},
   JOURNAL = {Tunis. J. Math.},
  FJOURNAL = {Tunisian Journal of Mathematics},
    VOLUME = {2},
      YEAR = {2020},
    NUMBER = {2},
     PAGES = {359--378},
      ISSN = {2576-7658},
   MRCLASS = {18E30 (20J05 55U35)},
  MRNUMBER = {3990823},
MRREVIEWER = {Markus Szymik},
       DOI = {10.2140/tunis.2020.2.359},
       URL = {https://doi.org/10.2140/tunis.2020.2.359},
}

@article {Balmer20_bigsupport,
    AUTHOR = {Balmer, Paul},
     TITLE = {Homological support of big objects in tensor-triangulated
              categories},
   JOURNAL = {J. \'{E}c. polytech. Math.},
  FJOURNAL = {Journal de l'\'{E}cole polytechnique. Math\'{e}matiques},
    VOLUME = {7},
      YEAR = {2020},
     PAGES = {1069--1088},
      ISSN = {2429-7100},
   MRCLASS = {18G80 (18M05 20J05 55U35)},
  MRNUMBER = {4136434},
       DOI = {10.5802/jep.135},
       URL = {https://doi.org/10.5802/jep.135},
}

@article {BazzoniStovicek17,
    AUTHOR = {Bazzoni, Silvana and {\v{S}}{\v{t}}ov\'{\i}\v{c}ek, Jan},
     TITLE = {Smashing localizations of rings of weak global dimension at
              most one},
   JOURNAL = {Adv. Math.},
  FJOURNAL = {Advances in Mathematics},
    VOLUME = {305},
      YEAR = {2017},
     PAGES = {351--401},
      ISSN = {0001-8708},
   MRCLASS = {13D09 (13F05 16E45 18E35)},
  MRNUMBER = {3570139},
MRREVIEWER = {Lutz Struengmann},
       DOI = {10.1016/j.aim.2016.09.028},
       URL = {https://doi.org/10.1016/j.aim.2016.09.028},
}

@incollection {HoveyPalmieri99,
    AUTHOR = {Hovey, Mark and Palmieri, John H.},
     TITLE = {The structure of the {B}ousfield lattice},
 BOOKTITLE = {Homotopy invariant algebraic structures ({B}altimore, {MD},
              1998)},
    SERIES = {Contemp. Math.},
    VOLUME = {239},
     PAGES = {175--196},
 PUBLISHER = {Amer. Math. Soc., Providence, RI},
      YEAR = {1999},
   MRCLASS = {55U35 (06D10 55P60)},
  MRNUMBER = {1718080},
MRREVIEWER = {Haynes R. Miller},
       DOI = {10.1090/conm/239/03601},
       URL = {https://doi.org/10.1090/conm/239/03601},
}

@article {bhs1,
    AUTHOR = {Barthel, Tobias and Heard, Drew and Sanders, Beren},
     TITLE = {Stratification in tensor triangular geometry with applications
              to spectral {M}ackey functors},
   JOURNAL = {Camb. J. Math.},
  FJOURNAL = {Cambridge Journal of Mathematics},
    VOLUME = {11},
      YEAR = {2023},
    NUMBER = {4},
     PAGES = {829--915},
      ISSN = {2168-0930},
   MRCLASS = {18G80 (14F08 18F99 55P42 55P91 55U35)},
  MRNUMBER = {4650265},
       DOI = {10.4310/cjm.2023.v11.n4.a2},
       URL = {https://doi.org/10.4310/cjm.2023.v11.n4.a2},
	   SORTKEY = {bhs1},
}

@article {Strickl2019Combinatorial,
    AUTHOR = {Strickland, Neil Patrick},
     TITLE = {A combinatorial model for the known {B}ousfield classes},
   JOURNAL = {Algebr. Geom. Topol.},
  FJOURNAL = {Algebraic \& Geometric Topology},
    VOLUME = {19},
      YEAR = {2019},
    NUMBER = {6},
     PAGES = {2677--2713},
      ISSN = {1472-2747},
   MRCLASS = {55P42 (16Y60 55P60)},
  MRNUMBER = {4023326},
MRREVIEWER = {Geoffrey M. L. Powell},
       DOI = {10.2140/agt.2019.19.2677},
       URL = {https://doi.org/10.2140/agt.2019.19.2677},
}

@article {Hill2019Equivariant,
    AUTHOR = {Hill, Michael A.},
     TITLE = {Equivariant chromatic localizations and commutativity},
   JOURNAL = {J. Homotopy Relat. Struct.},
  FJOURNAL = {Journal of Homotopy and Related Structures},
    VOLUME = {14},
      YEAR = {2019},
    NUMBER = {3},
     PAGES = {647--662},
      ISSN = {2193-8407},
   MRCLASS = {55P42 (55P60 55P91)},
  MRNUMBER = {3987553},
MRREVIEWER = {Samik Basu},
       DOI = {10.1007/s40062-018-0226-2},
       URL = {https://doi.org/10.1007/s40062-018-0226-2},
}

@article {bhs2,
    AUTHOR = {Barthel, Tobias and Heard, Drew and Sanders, Beren},
     TITLE = {Stratification and the comparison between homological and
              tensor triangular support},
   JOURNAL = {Q. J. Math.},
  FJOURNAL = {The Quarterly Journal of Mathematics},
    VOLUME = {74},
      YEAR = {2023},
    NUMBER = {2},
     PAGES = {747--766},
      ISSN = {0033-5606},
   MRCLASS = {55P42 (18G80)},
  MRNUMBER = {4596217},
MRREVIEWER = {Geoffrey M. L. Powell},
       DOI = {10.1093/qmath/haac040},
       URL = {https://doi.org/10.1093/qmath/haac040},
	   SORTKEY = {bhs2},
}

@article {PatchkoriaSandersWimmer22,
    AUTHOR = {Patchkoria, Irakli and Sanders, Beren and Wimmer, Christian},
     TITLE = {The spectrum of derived {M}ackey functors},
   JOURNAL = {Trans. Amer. Math. Soc.},
  FJOURNAL = {Transactions of the American Mathematical Society},
    VOLUME = {375},
      YEAR = {2022},
    NUMBER = {6},
     PAGES = {4057--4105},
      ISSN = {0002-9947},
   MRCLASS = {18G80 (19A99 55P91 55U35)},
  MRNUMBER = {4419053},
       DOI = {10.1090/tran/8485},
       URL = {https://doi-org.oca.ucsc.edu/10.1090/tran/8485},
}

@article{BalmerCameron20pp,
   AUTHOR = {Balmer, Paul and Cameron, James C.},
     TITLE = {Computing homological residue fields in algebra and topology},
   JOURNAL = {Proc. Amer. Math. Soc.},
  FJOURNAL = {Proceedings of the American Mathematical Society},
    VOLUME = {149},
      YEAR = {2021},
    NUMBER = {8},
     PAGES = {3177--3185},
      ISSN = {0002-9939},
   MRCLASS = {18D99 (20J05 55U35)},
  MRNUMBER = {4273126},
MRREVIEWER = {Markus Szymik},
       DOI = {10.1090/proc/15412},
       URL = {https://doi.org/10.1090/proc/15412},
}

@article {BarthelCastellanaHeardValenzuela19,
    AUTHOR = {Barthel, Tobias and Castellana, Nat\`alia and Heard, Drew and
              Valenzuela, Gabriel},
     TITLE = {Stratification and duality for homotopical groups},
   JOURNAL = {Adv. Math.},
  FJOURNAL = {Advances in Mathematics},
    VOLUME = {354},
      YEAR = {2019},
     PAGES = {106733, 61},
      ISSN = {0001-8708},
   MRCLASS = {55R35 (20J05 55P42)},
  MRNUMBER = {3989930},
MRREVIEWER = {Zafer Mahmud},
       DOI = {10.1016/j.aim.2019.106733},
       URL = {https://doi.org/10.1016/j.aim.2019.106733},
}

@article {Carrick2022Smashing,
    AUTHOR = {Carrick, Christian},
     TITLE = {Smashing localizations in equivariant stable homotopy},
   JOURNAL = {J. Homotopy Relat. Struct.},
  FJOURNAL = {Journal of Homotopy and Related Structures},
    VOLUME = {17},
      YEAR = {2022},
    NUMBER = {3},
     PAGES = {355--392},
      ISSN = {2193-8407},
   MRCLASS = {55P91},
  MRNUMBER = {4470384},
MRREVIEWER = {Andr\'{e} G. Henriques},
       DOI = {10.1007/s40062-022-00310-1},
       URL = {https://doi.org/10.1007/s40062-022-00310-1},
}

@article {Yosimura1985Acyclicity,
    AUTHOR = {Yosimura, Zen-ichi},
     TITLE = {Acyclicity of {BP}-related homologies and cohomologies},
   JOURNAL = {Osaka J. Math.},
  FJOURNAL = {Osaka Journal of Mathematics},
    VOLUME = {22},
      YEAR = {1985},
    NUMBER = {4},
     PAGES = {875--893},
      ISSN = {0030-6126},
   MRCLASS = {55N22 (55P60)},
  MRNUMBER = {815456},
MRREVIEWER = {W. Stephen Wilson},
       URL = {http://projecteuclid.org/euclid.ojm/1200778773},
}

@article {KrauseLetz23,
    AUTHOR = {Krause, Henning and Letz, Janina C.},
     TITLE = {The spectrum of a well-generated tensor-triangulated category},
   JOURNAL = {Bull. Lond. Math. Soc.},
  FJOURNAL = {Bulletin of the London Mathematical Society},
    VOLUME = {55},
      YEAR = {2023},
    NUMBER = {2},
     PAGES = {680--705},
      ISSN = {0024-6093},
   MRCLASS = {18G80 (06D22)},
  MRNUMBER = {4575930},
MRREVIEWER = {Alex Martsinkovsky},
       DOI = {10.1112/blms.12749},
       URL = {https://doi.org/10.1112/blms.12749},
}

@incollection {Casacuberta23,
    AUTHOR = {Casacuberta, Carles},
     TITLE = {Cohomological localizations and set-theoretical reflection},
 BOOKTITLE = {Mathematics going forward---collected mathematical
              brushstrokes},
    SERIES = {Lecture Notes in Math.},
    VOLUME = {2313},
     PAGES = {167--181},
 PUBLISHER = {Springer, Cham},
      YEAR = {2023},
      ISBN = {978-3-031-12243-9; 978-3-031-12244-6},
   MRCLASS = {55P60 (55N25)},
  MRNUMBER = {4627957},
MRREVIEWER = {Constanze\ Roitzheim},
       DOI = {10.1007/978-3-031-12244-6\{_}13}

@unpublished{BalmerSanders_perfect,
    AUTHOR = {Balmer, Paul and Sanders, Beren},
     TITLE = {Perfect complexes and completion},
  YEAR = {2024},
      NOTE = {Preprint, 20 pages, to appear in \emph{Forum Math. Sigma}},
   MRCLASS = {},
}

@article{BarthelCastellanaHeardSanders24,
    AUTHOR = {Barthel, Tobias and Castellana, Nat\`alia and Heard, Drew and Sanders, Beren},
title={On surjectivity in tensor triangular geometry},
journal={Mathematische Zeitschrift},
year={2024},
volume={308},
number={4},
pages={65},
issn={1432-1823},
doi={10.1007/s00209-024-03618-1},
url={https://doi.org/10.1007/s00209-024-03618-1}
}

@incollection {Hovey95a,
    AUTHOR = {Hovey, Mark},
     TITLE = {Bousfield localization functors and {H}opkins' chromatic
              splitting conjecture},
 BOOKTITLE = {The \v{C}ech centennial ({B}oston, {MA}, 1993)},
    SERIES = {Contemp. Math.},
    VOLUME = {181},
     PAGES = {225--250},
 PUBLISHER = {Amer. Math. Soc., Providence, RI},
      YEAR = {1995},
      ISBN = {0-8218-0296-8},
   MRCLASS = {55P42 (55N20 55N22 55P60)},
  MRNUMBER = {1320994},
       DOI = {10.1090/conm/181/02036},
       URL = {https://doi.org/10.1090/conm/181/02036},
}

@unpublished{BurklundHahnLevySchlank23pp,
    AUTHOR = {Burklund, Robert and Hahn, Jeremy and Levy, Ishan and Schlank, Tomer M.},
     TITLE = {{$K$}-theoretic counterexamples to {R}avenel's telescope conjecture},
   JOURNAL = {},
	YEAR = {2023},
      NOTE = {Preprint, 100~pages, available online at \href{https://arxiv.org/abs/2310.17459}{arXiv:2310.17459}},
   MRCLASS = {},
}

@article {Lin76,
    AUTHOR = {Lin, T. Y.},
     TITLE = {Duality and {E}ilenberg-{M}ac {L}ane spectra},
   JOURNAL = {Proc. Amer. Math. Soc.},
  FJOURNAL = {Proceedings of the American Mathematical Society},
    VOLUME = {56},
      YEAR = {1976},
     PAGES = {291--299},
      ISSN = {0002-9939,1088-6826},
   MRCLASS = {55E10},
  MRNUMBER = {402738},
MRREVIEWER = {Donald\ W.\ Kahn},
       DOI = {10.2307/2041622},
       URL = {https://doi.org/10.2307/2041622},
}

@article {BalmerSanders25,
    AUTHOR = {Balmer, Paul and Sanders, Beren},
     TITLE = {The {T}ate intermediate value theorem},
   JOURNAL = {Adv. Math.},
  FJOURNAL = {Advances in Mathematics},
    VOLUME = {483},
      YEAR = {2025},
     PAGES = {Paper No. 110675, 43},
      ISSN = {0001-8708,1090-2082},
   MRCLASS = {18F99},
  MRNUMBER = {4989711},
       DOI = {10.1016/j.aim.2025.110675},
       URL = {https://doi.org/10.1016/j.aim.2025.110675},
}

@misc{NaumannPolRamzi24,
      title={A symmetric monoidal fracture square},
      author={Niko Naumann and Luca Pol and Maxime Ramzi},
      year={2024},
      eprint={2411.05467},
      archivePrefix={arXiv},
      primaryClass={math.AT},
      url={https://arxiv.org/abs/2411.05467},
      NOTE = {Preprint, 40~pages, available online at \href{https://arxiv.org/abs/2411.05467}{arXiv:2411.05467}},
}

@article {Krause05,
    AUTHOR = {Krause, Henning},
     TITLE = {Cohomological quotients and smashing localizations},
   JOURNAL = {Amer. J. Math.},
  FJOURNAL = {American Journal of Mathematics},
    VOLUME = {127},
      YEAR = {2005},
    NUMBER = {6},
     PAGES = {1191--1246},
      ISSN = {0002-9327,1080-6377},
   MRCLASS = {18E30},
  MRNUMBER = {2183523},
MRREVIEWER = {Friedrich\ W.\ Bauer},
       URL =
              {http://muse.jhu.edu/journals/american_journal_of_mathematics/v127/127.6krause.pdf},
}

@unpublished{CasacubertaGutierrez24pp,
    AUTHOR = {Casacuberta, Carles and Guti\'{e}rrez, Javier J.},
	TITLE = {Homotopy reflectivity is equivalent to the weak {V}opĕnka principle},
   JOURNAL = {},
	YEAR = {2024},
      NOTE = {Preprint, 30~pages, available online at \href{https://arxiv.org/abs/2410.21244}{arXiv:2410.21244}},
   MRCLASS = {},
}

@ARTICLE{Putzstuck2026pp,
       author = {P{\"{u}}tzst{\"{u}}ck, Phil},
        title = {Tensor Nilpotence and the Size of the {B}ousfield Lattice},
      journal = {arXiv e-prints},
         year = 2026,
        month = aug,
          eid = {arXiv:2608.06104},
        pages = {arXiv:2608.06104},
          doi = {10.48550/arXiv.2608.06104},
archivePrefix = {arXiv},
       eprint = {2608.06104},
 primaryClass = {math.AT},
       adsurl = {https://ui.adsabs.harvard.edu/abs/2026arXiv260806104P}
}
\end{document}